\documentclass[10pt]{amsart} 
\usepackage{amssymb, latexsym, amsmath, amsthm, color, marginnote}
\usepackage{amsrefs}

\usepackage{geometry} 
\usepackage{booktabs} 
\usepackage{array} 
\usepackage{paralist} 
\usepackage{verbatim} 
\usepackage{subfig} 
\usepackage{mathtools}
\usepackage{amsfonts}
\usepackage{amsmath}
\usepackage[english]{babel}
 \usepackage{amsthm}
\usepackage{graphicx}
\usepackage{url}

\theoremstyle{plain}
\newtheorem{theorem}{Theorem}[section]

\newtheorem{lemma}[theorem]{Lemma}
\newtheorem{proposition}[theorem]{Proposition}
\newtheorem{corollary}[theorem]{Corollary}

\theoremstyle{definition}
\newtheorem{definition}[theorem]{Definition}
\newtheorem{remark}[theorem]{Remark}
\usepackage{fancyhdr} 
\newcommand{\field}{k}

\title{Plane curves which are quantum homogeneous spaces}
\author{Ken Brown  \and Angela Ankomaah Tabiri}
\thanks{University of Glasgow, School of Mathematics
and Statistics, University Place, Glasgow G12 8QQ,
UK}
\thanks{E-mail address: Ken.Brown@glasgow.ac.uk}
\thanks{E-mail address: a.tabiri.1@research.gla.ac.uk}
\address{University of Glasgow, School of Mathematics
and Statistics, University Place, Glasgow G12 8QQ,
UK}
\email{Ken.Brown@glasgow.ac.uk}
\email{a.tabiri.1@research.gla.ac.uk}
\date{} 

\begin{document}
\nocite{*} 
\begin{abstract}
Let $\mathcal{C}$ be a decomposable plane curve over an algebraically closed field $k$ of characteristic 0. That is, $\mathcal{C}$ is defined in $k^2$ by an equation of the form $g(x) = f(y)$, where $g$ and $f$ are polynomials of degree at least 2. We use this data to construct 3 pointed Hopf algebras, $A(x,a,g)$, $A(y,b,f)$ and $A(g,f)$, in the first two of which $g$ [resp. $f$] are skew primitive central elements, and the third being a factor of the tensor product of the first two. We conjecture that $A(g,f)$ contains the coordinate ring $\mathcal{O}(\mathcal{C})$ of $\mathcal{C}$ as a quantum homogeneous space, and prove this when each of $g$ and $f$ has degree at most 5 or is a power of the variable. We obtain many properties of these Hopf algebras, and show that, for small degrees, they are related to previously known algebras. For example, when $g$ has degree 3 $A(x,a,g)$ is a PBW deformation of the localisation at powers of a generator of the downup algebra $A(-1,-1,0)$.   
\end{abstract}
\maketitle

\section*{$0$. Introduction}

\subsection{}\label{xx0.1} Throughout this paper, let $k$ be an algebraically closed field of characteristic 0. By a \emph{quantum homogeneous space} of a Hopf $k$-algebra $H$ we mean a right or left coideal subalgebra $C$ of $H$ such that $H$ is a faithfully flat right and left $C$-module. The study of quantum homogeneous spaces has been a pervasive theme of research in Hopf algebras over the past 25 years, from both the algebraic and analytic perspectives. For the algebraic point of view which concerns us here see, for example,  \cite{Ma},\cite{MaW},\cite{MR1710737},\cite{Brown},\cite{KT},\cite{KM}. 

\subsection{}\label{xx0.2} Reasons for the ubiquity of quantum homogeneous spaces are manifest. For example, when $C$ is a right coideal subalgebra of an affine commutative Hopf $k$-algebra $H$, so that $H = \mathcal{O}(G)$ is the coordinate algebra of the affine algebraic group $G$, $H$ is always a flat $C$-module, but not always faithfully flat. As explained for example in \cite[$\S$3]{MaW}, \cite[Theorem 3 and Introduction]{Tak}, faithful flatness in this case encodes the algebraic counterpart of results of Demazure and Gabriel on quotient schemes of affine group schemes, hence intimately involved here as elsewhere with the analysis of symmetry. And in the wider, not necessarily commutative context, whenever one considers how to decompose a given Hopf algebra $H$ as an ``extension'' $1\longrightarrow C \longrightarrow H \longrightarrow \overline{H} \longrightarrow 1$ with $\overline{H}$ a Hopf algebra quotient of $H$, one is forced to allow $C$ to be a quantum homogeneous space rather than simply a Hopf subalgebra.

\noindent A question which arises almost immediately is: which classical (algebraic) spaces can occur as quantum homogeneous spaces? That is, translated into algebraic language:
\begin{eqnarray*} &\textit{Which commutative affine } k\textit{-algebras can occur as }\\
&\textit{faithfully flat right coideal subalgebras of a Hopf algebra }H?
\end{eqnarray*} 

\subsection{}\label{xx0.3}The purpose of this paper is to address this question in the simplest possible case, when $C = \mathcal{O}(\mathcal{C})$ is the coordinate ring of a curve. Homological considerations lead one to suspect that an affine commutative algebra $C$ can be a quantum homogeneous space only when $C$ is Gorenstein. Since plane algebraic curves are Gorenstein, we focus on plane curves, specifically on the case of a \emph{decomposable} plane curve $\mathcal{C}$, by which is meant a curve $\mathcal{C}$ of the form 
\begin{equation}\label{decomp} \mathcal{C} = \{  (x_0,y_0) \in \field^{2} \; | \; f(y_0)=g(x_0) \},
\end{equation}
where $f \in k[y], \, g \in k[x]$ and $ \deg(f),\deg(g) \geq 2$. Our main theorem is

\begin{theorem} \label{main}(= {\rm Theorem \ref{coideal}(i)}) Let $\mathcal{C}$ be the decomposable plane curve defined by (\ref{decomp}). Suppose hypothesis {\bf (H)}: each of $f(y)$ and $g(x)$ either has degree at most 5 or is a power of the variable. Then $\mathcal{C}$ is a quantum homogeneous space in a pointed affine Hopf algebra, denoted $A(g,f)$.
\end{theorem}

\noindent There are some precursors of this result in the literature. Goodearl and Zhang showed \cite[Construction 1.2]{MR2732991} that every cusp $y^m = x^n,$ where $m$ and $n$ are coprime integers with $m > n \geq 2$, is a quantum homogeneous space in a noetherian Hopf domain of Gel'fand-Kirillov dimension 2. In \cite{KT}, Kraehmer and the second author showed that the nodal cubic $y^2 = x^3 + x^2$ is a quantum homogeneous space, and Kraehmer and Martins refined this result in \cite{KM} to show that the containing Hopf algebra $H$ can in this case be taken to be affine noetherian of Gel'fand-Kirillov dimension 1. From a slightly different perspective, some further embeddings of this sort appear in Liu's massive endeavour \cite{Liu} aimed at extending the classification of prime affine Hopf $k$-algebras of Gel'fand-Kirillov dimension 1 to the non-regular case. 

\subsection{Proof of Theorem \ref{main}}\label{xx0.4} There are two steps to the proof of Theorem \ref{main}. In the first, given a monic polynomial $g(x)$ of degree $n$, $(n \geq 2)$, with constant term 0, one constructs a Hopf $k$-algebra $A(x,a,g)$ with generators $\{x, a^{\pm 1} \}$, in which $a$ is group-like, $x$ is $(1,a)$-primitive, $a^n$ and $g(x)$ are central and $g(x)$ is $(1,a^n)$-primitive. Following the same recipe, construct a corresponding Hopf algebra $A(y,b, f)$ on generators $\{y, b^{\pm 1}\}$. Then define the desired Hopf algebra as
$$ A(g,f) \quad := \quad A(x,a,g) \otimes_k A(y,b,f) / \langle a^n - b^m, g-f \rangle.$$ 
Here, in the relations on the right, we are writing $a^n$ for $a^n \otimes 1$, $b^m$ for $1 \otimes b^m$, $g$ for $g(x) \otimes 1$ and $f$ for $1 \otimes f(y)$. Finally, we use the Diamond Lemma \cite{MR506890} to show that $\mathcal{O}(\mathcal{C}) = k[x,y]/\langle g-f \rangle$ is a faithfully flat right coideal subalgebra of $A(g,f)$.

We believe that Theorem \ref{main} should remain valid in the absence of Hypothesis {\bf (H)}: it is required at present to permit an application of the Diamond Lemma to ensure that the constructed Hopf algebras $A(x,a,g)$ and $A(y,b,f)$ do not collapse to $k$.

\subsection{Structure of $A(g,f)$ and $A(x,a,g)$}\label{xx0.5} The Hopf algebras $A(x,a,g)$ and $A(g,f)$ have many interesting properties, some of which we summarise here. For unexplained terminology, see the relevant parts of the paper. 

\begin{theorem}\label{halfprop} Let $g = g(x)$ be a monic polynomial of degree $n$, with $n \geq 2$ and $g(0) = 0$, satisfying hypothesis {\bf (H)} of Theorem \ref{main}.
\begin{enumerate}
\item[(i)] The Hopf algebra $A(x,a,g)$ is pointed, generated by the group-like $a$ and the $(1,a)$-primitive element $x$.
\item[(ii)] If $n \leq 5$ then $A(x,a,g)$ has an explicit PBW basis. 
\item[(iii)] Suppose $n \leq 3$. Then $A(x,a,g)$ is noetherian, a finite module over its affine centre, and has Gel'fand-Kirillov dimension $n$.
\item[(iv)] Suppose $n \geq 4$. Then $A(x,a,g)$ contains a noncommutative free subalgebra.
\item[(v)] When $n = 2$, $A(x,a,g)$ is isomorphic as a Hopf algebra to the Borel in $U_{-1}(\mathfrak{sl}(2))$, or equivalently to the quantum plane at paramter $q = -1$, localised at the powers of one of its two generators.
\item[(vi)] Suppose $g(x) = x^3$. Then $A(x,a, g)$ is isomorphic as an algebra to the localisation of the down-up algebra $A(-1,-1, 0)$ at the powers of a generator. It is a noetherian domain of Gel'fand-Kirillov dimension 3, finite over its centre, a maximal order, AS-regular and GK-Cohen Macaulay. 
\item[(vii)] Suppose $n = 3$. Then $A(x,a,g)$ is a PBW deformation of $A(x,a,x^3)$, (and so of the localised down-up algebra), having the same properties as listed above for $A(x,a,x^3)$.
\end{enumerate}
\end{theorem}

Part (i) of Theorem \ref{halfprop} is Proposition \ref{cent}(ii), part (ii) is Corollary \ref{basis}, part (iii) comes from Propositions \ref{degtwo} and \ref{cube}, part (iv) from Corollary \ref{basis} and Proposition \ref{xnprop}(iv), part (v) from Proposition \ref{degtwo}, and parts (vi) and (vii) from Proposition \ref{three}. 

Our current knowledge of the containing Hopf algebras $A(g,f)$ is summarised in

\begin{theorem}\label{wholeprop} Let $g(x), f(y)$ be as in Theorem \ref{main}.
\begin{enumerate}
\item[(i)] $A(x,a,g)$ and $A(y,b,f)$ embed in $A(g,f)$, and these sub-Hopf algebras generate $A(g,f)$.
\item[(ii)] Suppose $\mathrm{max}\{n,m\} >  3$. Then $A(g,f)$ contains a noncommutative free subalgebra.
\item[(iii)] Suppose $\mathrm{max}\{n,m\} \leq 3$. Then $A(g,f)$ is a finite module over its affine centre. It is AS-Gorenstein and GK-Cohen Macaulay, with 
$$ \mathrm{GKdim} (A(g,f) )= n + m - 2. $$
\item[(iv)] If the curve $\mathcal{C}$ is singular, then $A(g,f)$ has infinite global dimension. 
\end{enumerate}
\end{theorem}

Part (i) of the above is Theorem \ref{coideal}(ii), part (ii) is Theorem \ref{coideal}(iii)(a), and parts (iii) and (iv) come from Theorem \ref{theworks}.

\subsection{Layout and notation}\label{hypoth} The structure of the paper is as follows. In $\S$1 the Hopf algebra $A(x,a,g)$ is defined, with a PBW theorem proved for $A(x,a,g)$ in 
$\S$2 when the polynomial $g(x)$ has degree at most 5. (However the tedious and very long calculations needed are relegated to an Appendix, available separately.) The special case $g(x) = x^n,$ $n$ arbitrary, is studied in $\S$3. The subject of $\S$4 is the detailed structure of $A(x,a,g)$ when $g(x)$ has degree 2 or 3. The algebras $A(g,f)$ are defined and studied in $\S$5 and Theorems \ref{main} and \ref{wholeprop} are proved. In the final $\S$6 a substantial number of examples are explicitly presented by generators and relations, and their properties discussed. 

All unadorned tensor products are assumed to be over $k$. We shall use the standard notation $\Delta, \varepsilon, S$ for the coproduct, counit and antipode of a Hopf algebra, as defined in \cite{montgomery}, for example. The centre of a $k$-algebra $R$ will be denoted by $Z(R)$.


\section{The Hopf algebra $A(x,a,g)$ }\label{1}
\subsection{Generators and relations}\label{nota}
Let $F= \field[x]*\field[a^{\pm 1}]$ be the free product of the polynomial algebra $\field[x]$ with the Laurent polynomial algebra $\field[a^{\pm 1}]$. It is easy to check that $F$ is a Hopf algebra with $x$ $(1,a)$-primitive and $a$ group-like. More precisely, we have the following easy lemma whose verification is left to the reader.
\begin{lemma}
\label{hopf}
The algebra $F$  admits a unique Hopf algebra structure whose coproduct $ \Delta $, counit $ \varepsilon $ and antipode $S$ satisfy:
\begin{gather*}
	\Delta(x) 
= 1 \otimes  x  + x \otimes
a,\qquad
	\Delta (a)=a \otimes a,\\
	\varepsilon(x) = 0, \quad
	\varepsilon(a)=1, \qquad
           S(x) =  - x a^{-1},\quad
	S(a) = a^{-1}.
\end{gather*}
\end{lemma}
Define a $(\mathbb{Z}_{\geq 0},\mathbb{Z}_{\geq 0})$-grading on $F$ by giving $a$ and $x$ the degrees $(1,0)$ and $(0,1)$ respectively.  For $i, j \in \mathbb{Z}_{\geq 0}$, let $P(j,i)_{(a,x)}$ denote the sum of all monomials in $F$ of degree $(j,i)$; in particular we set $P(0,0)_{(a,x)} = 1$. It is convenient in the proof of Lemma \ref{xa} to extend the argument of $P(j, i)_{(a,x)}$ to $\mathbb{Z} \times \mathbb{Z}$, by setting $P(j,i)_{(a,x)} = 0$ if $i$ or $j$ is less than 0. We'll omit the suffix $(a,x)$ whenever the variables involved are clear from the context.

Let $n \geq 2$  and $ g := g(x)=\sum_{i=1}^{n}r_{i}x^{i} \in \field[x]$ with $r_n \neq 0$. Define $A(x,a,g)$ to be the quotient of $F$ by the ideal $I$ generated by the elements
\begin{equation}
\label{ij}
        \sigma_j = \sigma_j (x,a,g) := \sum_{i=j}^{n} r_{i} P(j,i-j)_{(a,x)} - r_{j}a^{n}, \; \; \quad j = 1, \ldots , n-1.
\end{equation}
As indicated, we'll omit the variables and simply write $\sigma_j$ whenever possible. It will be useful also for us sometimes to view the elements $\sigma_j$ of $F$ as members of the free subalgebra $E := k \langle x,a \rangle$ of $F$. We define $I_g := \langle \sigma_j : 1 \leq j \leq n-1 \rangle$, an ideal of $E$, and 
\begin{equation}\label{bialg}A_0(x,a,g) :=  E / I_g.
\end{equation}
Throughout the paper, we shall slightly abuse notation by using the same symbols for the elements $x$ and $a$ of $E$ and $F$ and for their cosets in $A_0(x,a,g)$ and $A(x,a,g)$. We shall on some occasions find it convenient to assume that $g(x)$ is normalised so that its highest coefficient  $r_n$ equals $1$. The form of the relations (\ref{ij}) ensures that this does not affect the definition of $A_0(x,a,g)$ or $A(x,a,g)$.
\subsection{$A(x,a,g)$ is a Hopf algebra }\label{A}
 To prove that the Hopf algebra structure on $F$ descends to $A(x,a,g)$  we need the following lemma.

\begin{lemma}\label{xa} Retain the notation of $\S$1.1.
\begin{enumerate}
\item[(i)] In $F$, for $\ell \in \mathbb{Z}_{\geq 0}$, 
$$ \Delta (x^{\ell}) = \sum_{s=0}^{\ell} x^s \otimes P(s, \ell - s). $$

\item[(ii)] In $F$, for all $j, t \in \mathbb{Z}_{\geq 0}$
$$ \Delta (P(j,t)) = \sum_{\ell = 0}^t P(j, \ell) \otimes P(j+ \ell, t - \ell). $$

\item[(iii)] Modulo $I \otimes F + F \otimes I$,
$$ \Delta(g) \equiv 1 \otimes g  +  g \otimes a^{n} \label{x} . $$

\item[(iv)]  Let $j \in \mathbb{Z}$ with $1 \leq j \leq n-1$. Modulo $I \otimes F + F \otimes I$,
$$ \Delta( \sum_{\ell =j}^{n} r_{\ell} P(j,\ell -j)  ) \equiv  r_{j}a^{n} \otimes a^{n}. $$
\end{enumerate}
\end{lemma}

\begin{proof}(i) Let $\ell \in \mathbb{Z}_{\geq 0}$ Then
$$
\Delta (x^{\ell}) = \Delta(x)^{\ell}\\
= (1 \otimes x + x \otimes a)^{\ell}\\
= \sum_{s=0}^{\ell} x^s \otimes P(s, \ell - s).
$$
\noindent (ii) We prove this by induction on $j+t$. For $n \in \mathbb{Z}_{> 0}$, let $\mathcal{P}(n)$ be the statement that, in $F$, for all $j,t \in \mathbb{Z}_{\geq 0}$ with $n=j+t$,
$$
  \Delta (P(j,t)) = \sum_{\ell = 0}^t P(j, \ell) \otimes P(j+ \ell, t - \ell).
$$
Then $\mathcal{P}(1)$ is true since $\Delta(P(1,0))= a \otimes a$ and $\Delta(P(0,1))= 1 \otimes x + x \otimes a.$ 

Let $m \in \mathbb{Z}_{> 0}$ and suppose that $\mathcal{P}(\ell)$  is true for all $\ell = 1, \ldots , m$. Consider $P(j,t)$ where $j,t \in \mathbb{Z}_{\geq 0}$ with $j + t = m + 1$. If $t = 0$ then $j = m + 1$, and 
$$ \Delta (P(m + 1, 0)) = a^{m+1} \otimes a^{m+1} = P(m + 1, 0) \otimes P(m + 1, 0). $$
If $j= 0$ then $t = m+1$ and, by part (1),
$$ \Delta (P( 0, m+1)) = \Delta (x^{m+1}) 
= \sum_{s=0}^{m+1} x^s \otimes P(s, m+1 - s)
= \sum_{s=0}^{m+1} P(0,s) \otimes P(s, m+1 - s).
$$
Thus $\mathcal{P}(m+1)$ holds when either $j$ or $t$ is $0$.

In what follows we shall frequently use the trivial identities
\begin{equation}
\label{heart}
 P(r,s) \quad = \quad P(r, s-1)x + P(r-1,s)a,
\end{equation}
which hold for all $r,s \geq 0$.

Suppose now that $j + t= m +1$ with $j,t \in \mathbb{Z}_{>0}$. By (\ref{heart}) and $\mathcal{P}(\ell)$ for $\ell = 1, \ldots , m$,
\begin{eqnarray*} \Delta(P(j,t)) &=&\Delta(P(j,t-1))\Delta(x) + \Delta(P(j-1,t))\Delta(a)\\
 &=&\left( \sum_{\ell=0}^{t-1}P(j,\ell) \otimes P(j+\ell, t-1-\ell) \right) \left(1 \otimes x + x \otimes a \right) \\
&& \qquad \qquad \qquad + \left( \sum_{\ell=0}^{t}P(j-1,\ell) \otimes P(j-1+\ell, t-\ell) \right) (a \otimes a)\\
&=& \sum_{\ell=0}^{t-1}P(j,\ell) \otimes P(j+\ell, t-1-\ell)x + \sum_{\ell=0}^{t-1}P(j,\ell)x \otimes P(j+\ell, t-1-\ell)a\\
&& \qquad \qquad \qquad + \sum_{\ell=0}^{t}P(j-1,\ell)a \otimes P(j-1+\ell, t-\ell)a\\
&=& P(j,0) \otimes P(j,t-1)x +  \sum_{\ell=1}^{t-1}P(j,\ell) \otimes P(j+\ell, t-1-\ell)x\\
&& \qquad + \sum_{\ell=0}^{t-2}P(j,\ell)x \otimes P(j+\ell, t-1-\ell)a + P(j,t-1)x \otimes P(j+t-1,0)a\\
&& \qquad  + P(j-1,0)a \otimes P(j-1,t)a + \sum_{\ell=1}^{t-1}P(j-1,\ell)a \otimes P(j-1+\ell, t-\ell)a\\ 
&& \qquad \qquad \qquad + P(j-1,t)a \otimes P(j+t-1,0)a. 
\end{eqnarray*} 
\vspace{-.01cm}
Replacing $P(j-1,0)a$ and $P(j+t-1,0)a$ in this last expression by, respectively, $P(j,0)$ and $P(j+t,0)$ yields the first equality below, and the second follows by two applications of (\ref{heart}) to gather terms:
\vspace{-.5cm}
\begin{eqnarray*}\Delta(P(j,t)) &=& P(j,0) \otimes \big( P(j,t-1)x + P(j-1,t)a \big) + \sum_{\ell=1}^{t-1}P(j,\ell) \otimes P(j+\ell, t-1-\ell)x\\
&& \qquad + \sum_{\ell=0}^{t-2}P(j,\ell)x \otimes P(j+\ell, t-1-\ell)a \\
&&+ \sum_{\ell=1}^{t-1}P(j-1,\ell)a \otimes P(j-1+\ell, t-\ell)a\\
&&  \qquad +  (P(j,t-1)x + P(j-1,t)a )\otimes P(j+t,0)\\
&=& P(j,0) \otimes P(j,t) +\sum_{\ell=1}^{t-1}P(j,\ell) \otimes P(j+\ell, t-1-\ell)x\\
&&  \qquad \qquad + \sum_{\ell=0}^{t-2}P(j,\ell)x \otimes P(j+\ell, t-1-\ell)a \qquad \qquad (*)\\
&& \qquad \qquad + \sum_{\ell=1}^{t-1}P(j-1,\ell)a \otimes P(j-1+\ell, t-\ell)a\\
&& \qquad \qquad + P(j,t) \otimes P(j+t, 0).
\end{eqnarray*}
\vspace{-.1cm}
We analyse further the second summation term (*) in the final expression above, using the identities (\ref{heart}) for the first equality, and identities (\ref{heart}) along with renumbering of the summations for the third equality.
\begin{eqnarray*}
 &&  (**) \qquad \qquad \sum_{\ell = 0}^{t-2} P(j,\ell)x \otimes P(j+\ell, t-1-\ell)a \\
&& \qquad  = \sum_{\ell=0}^{t-2}\left(P(j,\ell +1) - P(j-1, \ell + 1)a\right) \otimes \left(P(j + \ell + 1, t - \ell -1)) - P(j+ \ell + 1, t - \ell -2)x\right)\\
&& \qquad  = \sum_{\ell=0}^{t-2}P(j,\ell + 1) \otimes P(j + \ell + 1, t - \ell - 1) -   \sum_{\ell=0}^{t-2}P(j,\ell + 1) \otimes P(j + \ell + 1, t - \ell - 2)x\\
&& \qquad  -  \sum_{\ell=0}^{t-2}P(j-1, \ell + 1)a \otimes P(j + \ell + 1, t - \ell - 1) +  \sum_{\ell=0}^{t-2} P(j-1, \ell + 1)a \otimes P(j + \ell + 1, t - \ell - 2) x\\
&& \qquad = \sum_{\ell=1}^{t-1}P(j,\ell) \otimes P(j + \ell , t - \ell ) -   \sum_{\ell=1}^{t-1}P(j,\ell + 1) \otimes P(j + \ell, t - \ell - 1)x\\
&& \qquad \qquad - \sum_{\ell=1}^{t-1} P(j-1,\ell )a \otimes P(j -1 + \ell, t-\ell )a. \end{eqnarray*}
Finally, substituting the right side of (**) for the summation (*) in the above expression for $\Delta(P(j,t))$, we obtain
\vspace{-.1cm}
\begin{eqnarray*} \Delta(P(j,t)) &=& P(j,0) \otimes P(j,t) + P(j,t) \otimes P(j+t,0)\\
&& \qquad \sum_{\ell = 1}^{t-1}P(j,\ell) \otimes P(j+\ell, t- \ell)\\
&=& \sum_{\ell = 0}^t P(j,\ell) \otimes P(j+\ell, t- \ell),\end{eqnarray*}
proving $\mathcal{P}(m)$, as required for the induction step.


\noindent (iii) It is convenient to define $r_0 =0$. By part (i), in $F$,
\vspace{-.1cm}
\begin{eqnarray*}
\Delta (g) &=& \sum_{\ell = 0}^n r_{\ell}\Delta( x^{\ell})\\
&=& \sum_{\ell = 0}^n r_{\ell}\left(\sum_{s=0}^{\ell} (x^s \otimes P(s, \ell - s)\right)\\
&=& \sum_{s=0}^n x^s \otimes\left (\sum_{\ell = s}^n r_{\ell} P(s, \ell - s)\right).
\end{eqnarray*}
Thus, recalling the generators (\ref{ij}) of $I$ and noting that $r_0 = 0$, the above identity in $F$ implies that, \emph{modulo}$(I \otimes F + F \otimes I),$
\begin{eqnarray*}
\Delta (g) &\equiv& 1 \otimes \left(\sum_{\ell = 1}^n r_{\ell}P(0, \ell)\right) + \sum_{s=1}^n x^s \otimes r_s a^n\\
&\equiv& 1 \otimes \left(\sum_{\ell=1}^n r_{\ell}x^{\ell}\right) + \left(\sum_{s=1}^n r_s x^s\right) \otimes  a^n,\\
\end{eqnarray*}
proving (iii).


\noindent (iv) Let $j \in \mathbb{Z}_{\geq 0}$, with $1 \leq j \leq n-1$. We calculate, using (ii) for the second equality, regrouping terms for the third, and then by two applications of the relations (\ref{ij}), that
\begin{eqnarray*} 
 \Delta \left( \sum_{i =j}^{n} r_{i} P(j,i -j)  \right) &=&   \sum_{i =j}^{n} r_{i} \Delta (P(j,i -j)  )\\
&=&   r_{j}P(j,0) \otimes P(j,0) + r_{j+1}\left(\sum_{i = 0}^{1}P(j,i) \otimes P(j+ i , 1 - i)\right) +\\ 
&& \qquad \cdots +  r_{n}\sum_{i = 0}^{n-j}\left(P(j,i) \otimes P(j+ i, n-j - i)\right)\\
 &=&  P(j,0) \otimes \sum_{s= j}^{n} r_{s} P(j,s- j) + P(j,1) \otimes \sum_{s = j+1}^{n} r_{s}P(j+1,s- j-1) \\
 &+ \cdots + &P(j,n-1-j) \otimes \left(\sum_{s = n-1}^{n} r_{\ell}P(n-1,s- n+1)\right)  \\
&+&  P(j,n-j) \otimes r_{n}P(n,0)\\
&\equiv&   \left(P(j,0) \otimes  r_{j} P(n,0)\right) + \left(P(j,1) \otimes  r_{j+1} P(n,0)\right)  \\
&& \qquad + \cdots + \left(P(j,n-j) \otimes r_{n}P(n,0)\right) \qquad (\textit{mod}I).
\end{eqnarray*}
We may thus conclude that, for $j = 1, \ldots , n-1$,
\begin{eqnarray*}
 \Delta( \sum_{i =j}^{n} r_{i} P(j,i -j)  ) &\equiv& \left(\sum_{i =j}^{n} r_{i} P(j,i-j)\right) \otimes   P(n,0) \\
&\equiv& r_j a^n \otimes   a^n,\\
\end{eqnarray*} 
as required.

\end{proof}

It is now a simple matter to deduce the following theorem. Notice in particular that the existence of the counit includes the implication that $A(x,a,g)$ is not $\{0\}$, a conclusion to be strengthened in Proposition \ref{cent}(iv).

\begin{theorem}\label{gee} Assume the notation and hypotheses from $\S$\ref{nota}. Then the $\field$-algebra $A:=A(x,a,g)$ is a Hopf algebra with the coproduct, counit and antipode defined in Lemma \ref{hopf}.
\end{theorem}

\begin{proof}
 By Lemma \ref{hopf}, it suffices to show that, for each generator $\sigma_{j}$ of $I_g$,
$$
\Delta(\sigma_{j}) \in I_g F\otimes F + F \otimes I_g F; \quad S(\sigma_{j}) \in I_g F; \quad \textit{and} \; \varepsilon(\sigma_{j})=0.
$$
First, for $j= 1, \ldots n-1$, $\Delta(\sigma_{j}) \in I_g F \otimes F + F \otimes I_g F$ by Lemma \ref{xa}(iv), and it is easy to check that $\varepsilon(\sigma_{j})=0$.  Thus $A(x,a,g)$  is a bialgebra, which by $\S$\ref{nota} is generated by the invertible grouplike element $a$, and the $\{1,a\}$-primitive element $x$. Hence, by \cite[Proposition 7.6.3]{radford}, it is a Hopf algebra, with antipode given by the formulae in Lemma \ref{hopf}. 
\end{proof}

\subsection{Scaling isomorphisms}\label{leading}

\begin{lemma}\label{coef}  Assume the notation and hypotheses from $\S$\ref{nota}, and let $\lambda \in k \setminus \{0\}$. Define  a $k$-algebra automorphism $\theta_{\lambda}$ of  $F = k\langle x,a \rangle$ by $\theta_{\lambda} (a) = a, \, \theta_{\lambda} (x) = \lambda x$. Set $$g^{\lambda}  := g(\lambda x) = \sum_{i=1}^n r_i (\lambda x)^i .$$
\begin{enumerate}
\item[(i)] $\theta_{\lambda} (I_g) = I_{g^{\lambda}}.$
\item[(ii)] $\theta_{\lambda}$ induces an algebra isomorphism from $A_0(x,a,g)$ to $A_0(x,a, g^{\lambda})$, and a Hopf algebra isomorphism from $A(x,a,g)$ to $A(x,a, g^{\lambda})$.
\end{enumerate} 
\end{lemma}
\begin{proof} (i) One checks easily that, for $j = 1, \ldots , n-1,$ 
$$ \theta_{\lambda} (\sigma_j (x,a,g) ) \; = \; \lambda^{-j} \sigma_j (x,a,g^{\lambda}). $$

\noindent (ii) It is immediate from (i) that $\theta_{\lambda}$ induces an algebra isomorphism from $A_0(x,a,g)$ to $A_0(x,a, g^{\lambda})$, which clearly extends to $A(x,a,g)$ since $\theta_{\lambda}$ extends to an automorphism of $F$. It is routine to check that the map respects the Hopf operations.
\end{proof}

\begin{remark}\label{scale} Lemma \ref{coef} permits us to adjust the defining polynomial $g$, for example by ensuring that the polynomial is monic, as we shall frequently do in the sequel, to ease calculations. 
\end{remark}

\subsection{First  properties of  the algebras $A(x,a,g)$ and $A_0(x,a,g)$}\label{indep}
Here are some elementary consequences of the defining relations for $A(x,a,g)$ and $A_0(x,a,g)$, valid for all choices of $g = g(x)$.

\begin{proposition} 
\label{cent}
Let $A := A(x,a,g)$ be as defined in $\S$\ref{nota}, so $g$ has degree $n$. Then the following hold.
\begin{enumerate}
\item[(i)] The element $g$ of $A$ is $(1,a^{n})$-primitive.
\item[(ii)] $A$ is a pointed Hopf algebra.
\item[(iii)] $k\langle g, a^{\pm n} \rangle$ of $A$ is a central Hopf subalgebra of $A$.
\item[(iv)] The polynomial algebra $k[x]$  embeds in $A_0(x,a,g)$.
\end{enumerate}
\end{proposition}

\begin{proof}(i) This is immediate from Lemma \ref{xa}(iii).

\noindent (ii) The Hopf algebra $A$ is pointed by \cite[Corollary 5.1.14(a)]{radford}, since it is generated by the grouplike element $a$ and the skew-primitive element $x$.

\noindent (iii) In view of (i) and the fact that $a^n$ is grouplike, it suffices to show that $a^{n}$ commutes with $x$ and that $g$ commutes with $a$. We can assume without loss of generality that $g$ is a monic polynomial.

From the defining relations (\ref{ij}) in $\S$1.1, the elements $a,x$ of $A$ satisfy $\sigma_{n-1}$, namely
$$ r_{n-1} a^{n-1} + P(n-1,1) = r_{n-1}a^n. $$ 
Thus
\begin{equation}\label{bingo} a^{n-1}x = -( xa^{n-1} + axa^{n-2}  + \cdots + a^{n-2}xa ) - r_{n-1}a^{n-1} +   r_{n-1} a^{n},
\end{equation}
and so, pre-multiplying (\ref{bingo}) by $a$, yields
$$
a^{n}x = - ( axa^{n-1} + a^{2}xa^{n-2} + \cdots + a^{n-1}xa ) - r_{n-1}a^{n} +  r_{n-1}a^{n+1}.
$$
Using (\ref{bingo}) to replace the term $ a^{n-1}xa = (a^{n-1}x)a$ in this identity yields 
\begin{eqnarray*}
a^{n}x \quad &=& \quad  - ( axa^{n-1} + a^{2}xa^{n-2}  + \cdots + a^{n-2}xa^{2} )  - r_{n-1}a^{n} +   r_{n-1}a^{n+1}\\ 
&& \qquad + xa^{n} +  axa^{n-1} + a^{2}xa^{n-2}  + \cdots + a^{n-2}xa^{2}  + r_{n-1}a^{n} -   r_{n-1}a^{n+1}.
\end{eqnarray*}
That is, $a^{n}x =  xa^{n} $ as required. 

To show that $g$ commutes with $a$, begin with relation $\sigma_{1}$ from (\ref{ij}), namely
$$
\sum_{i=1}^{n}r_{i}P(1, i-1) = r_{1}a^{n}.
$$
That is, in $A$,
\begin{equation}\label{ludo}
ax^{n-1} = - ( x^{n-1}a + x^{n-2}ax + \cdots + xax^{n-2} ) -  \sum_{i=1}^{n-1} r_{i}P(1,i-1) + r_{1}a^{n}.
\end{equation}
Post-multiplying (\ref{ludo}) by $x$ and then using (\ref{ludo}) to replace the term $xax^{n-1}$ in the resulting identity, we obtain
\begin{eqnarray*}
ax^{n} &=& - ( x^{n-1}ax + x^{n-2}ax^{2} + \cdots + x^{2}ax^{n-2} ) -  \sum_{i=1}^{n-1} r_{i}P(1,i-1)x + r_{1} a^{n}x\\
&& \quad + ( x^{n}a + x^{n-1}ax + \cdots + x^{2}ax^{n-2} ) +  \sum_{i=1}^{n-1} r_{i} x P(1,i-1) -  r_{1} xa^{n}.
\end{eqnarray*}
The monomials in the above identity which begin and end with $x$ cancel, so that, in $A$,
$$
ax^{n} = x^{n}a -  a \left( \sum_{i=1}^{n-1} r_{i}x^{i}\right) +  \left( \sum_{i=1}^{n-1} r_{i}x^{i} \right) a + r_1(a^nx - xa^n).  
$$
Since $a^n$ is central as proved above, it follows that $ag = ga$, so $g \in Z(A)$. Therefore, by (i), the proof of (iii) is complete.

\noindent (iv) Recall that $E$ denotes the free algebra $k\langle x,a \rangle$. Since $\sigma_j \in EaE$ for all $j = 1, \ldots , n-1$, it follows that
$$ A_0(x,a,g)/A_0(x,a,g)a A_0(x,a,g) \; \cong \; E/EaE \; \cong \; k[x]. $$
Hence, $k[x]$ is a subalgebra of $A_0(x,a,g)$. 

\end{proof}

\begin{remark}\label{prob} \rm{The restriction of (iv) of the above proposition to $A_0(x,a,g)$ rather than $A(x,a,g)$ constitutes a gap in our analysis: we would like to be able to say that no equation of the form 
\begin{equation}\label{torsion} a^m h(x) = 0, 
\end{equation}
for $m > 0$ and $h(x) \in k[x] \setminus \{0\}$, is valid in $A_0(x,a,g)$. This would then imply that $k[x]$ embeds as a right coideal subalgebra of the localisation $A(x,a,g)$ of $A_0(x,a,g)$, and hence, by \cite[Theorem 1.3]{Ma}, that $A(x,a,g)$ is a faithfully flat $k[x]$-module. While we shall prove below that (\ref{torsion}) cannot occur in many important cases, the general statement remains open.}
\end{remark}


\section{PBW theorem for $n \leq 5$}

This section is devoted to the proof of the PBW theorem for $A(x,a,g)$, when the degree of $g(x)$ is at most 5. This is Corollary \ref{basis}.

\subsection{The Diamond Lemma }\label{diamond}

We recall the set-up needed  to apply  Bergman's diamond lemma \cite{MR506890}. For more details, see for example \cite[pp. 97-101]{bg}. Let $\langle X \rangle $ denote the free semigroup on a set $X$ and $\field \langle X \rangle$ the free $\field$-algebra with generators $X$. Let  $R$ be the quotient of $\field \langle X \rangle$ by a set of relations $\Sigma$. Suppose that every relation $\sigma \in \Sigma$ can be written in the form $W_{\sigma} = f_{\sigma}$ with $W_{\sigma} \in \langle X \rangle$ and $f_{\sigma} \in \field \langle X \rangle$. We write $\Sigma$ as the set of pairs of the form $\sigma = ( W_{\sigma}, f_{\sigma} )$. 

For each $\sigma \in \Sigma$ and $A,B \in \langle X \rangle$, let $r_{A\sigma B}$ denote the linear endomorphism of $ \field \langle X \rangle$ which fixes all the elements of  $\langle X \rangle$ other than $AW_{\sigma}B$, and which sends this basis element of $k \langle X \rangle$ to $Af_{\sigma}B$. We call $\Sigma$ a \emph{reduction system}, with the maps $r_{A\sigma B} : \field \langle X \rangle \rightarrow \field \langle X \rangle$ called \emph{elementary reductions}, and a composition of elementary reductions called a \emph{reduction}. An elementary reduction $r_{A\sigma B}$ acts \emph{trivially} on an element $ \alpha \in \field \langle X \rangle  $ if the coefficient of $AW_{\sigma}B$ in $\alpha$ is zero, and we call $\alpha$ \emph{irreducible} (under $\Sigma$) if every elementary reduction is trivial on $\alpha$. The $\field$-vector space of irreducible elements of $ \field \langle X \rangle$ is denoted by $ \field \langle X \rangle_{irr}$.

A \emph{semigroup ordering} on $\langle X \rangle$ is a partial order $\leq$  such that  if $a,b,c,d \in \langle X \rangle$ and $a < b$, then $cad < cbd$. A semigroup ordering on $\langle X \rangle$ is \emph{compatible with $\Sigma$} if, for all $\sigma \in \Sigma$, $f_{\sigma}$ is a linear combination of words $W$ with $W < W_{\sigma}$.  

Define an \emph{overlap ambiguity} of $\Sigma$ to be a 5-tuple $(\sigma, \tau, A, B, C)$ with $\sigma, \tau \in \Sigma$ and $A,B,C \in  \langle X \rangle $, such that $W_{\sigma} = AB$, $W_{\tau} = BC$. An overlap ambiguity $(\sigma, \tau, A, B, C)$ is \emph{resolvable} if there exist reductions $r$ and $r'$ such that 
$$r \circ r_{1\sigma C}(ABC) = r' \circ r_{A\tau 1}(ABC);$$
that is, $r(f_{\sigma}C)=r'(Af_{\tau})$. An \emph{inclusion ambiguity} is a 5-tuple $(\sigma, \tau, A, B, C)$ with $\sigma, \tau \in \Sigma$ and $A,B,C \in  \langle X \rangle $, such that $ABC = W_{\sigma}$ and $B= W_{\tau}$. An inclusion ambiguity $(\sigma, \tau, A,B,C)$ is \emph{resolvable} if there are reductions $r, r'$ such that $r\circ r_{1\sigma 1}(ABC) = r' \circ r_{A\tau C}(ABC)$. Observe that if the $W_{\sigma}$ for $\sigma \in \Sigma$ are distinct words of the same length, then there are no non-trivial inclusion ambiguities. 

Bergman's theorem \cite[Theorem 1.2]{MR506890} can now be stated, as follows.

\begin{theorem}\label{berg}  With the above notation and terminology, suppose that $\leq$ is a semigroup ordering on $\langle X \rangle$ which is compatible with $\Sigma$ and satisfies the descending chain condition. Suppose that all overlap and inclusion ambiguities are resolvable. Let $I$ be the ideal $\langle W_{\sigma} - f_{\sigma} : \sigma \in \Sigma \rangle$ of the free $k$-algebra $k\langle X \rangle$. Then the map $\omega \mapsto \omega + I$ gives a vector space isomorphism from $k\langle X \rangle_{irr}$ to $k\langle X \rangle/I$ ; that is, the irreducible words in $\langle X \rangle$ map bijectively to a $k$-basis of  $k\langle X \rangle / I$.
\end{theorem}

\subsection{The PBW theorem for $A(x,a,g), \; n \leq 5$}\label{PBW} In this subsection we first obtain a PBW theorem, when $g(x)$ has degree at most 5, for the bi-algebras $A_0(x,a,g)$, as defined in $\S$\ref{nota}, and then use localisation to obtain a similar result for the corresponding Hopf algebras $A(x,a,g)$. 

Fix an integer $n > 1$, and consider the setup of $\S$\ref{nota}, with $ g(x)$ normalised so that its top coefficient $r_n = 1$, as permitted by Lemma \ref{coef}. Thus $X = \{x,a\}$ and $E = k\langle x, a \rangle = k \langle X \rangle$ is the free algebra, with $E \subset F = k[x]\ast k[a^{\pm 1}]$. For non-negative integers $m$ and $q$, define
\begin{equation}\label{qdef} Q(m,q) \quad := \quad P(m,q) - a^mx^q \in k\langle X \rangle;
\end{equation}
in particular, $Q(0,q)=Q(m,0) = 0$ for all $q,m \geq 0$. The generators (\ref{ij}) in $F$ of the defining ideal $I$ of $A(x,a,g)$ can thus be regarded as a set $\Sigma_g = \{\sigma_1, \ldots , \sigma_{n-1}\}$ of relations for the free algebra $k \langle X \rangle$; namely, we view them as
\begin{equation}\label{graded}  \sigma_j : \quad \quad a^jx^{n-j} = -Q(j, n-j) - \sum_{i=j}^{n-1}r_i P(j,i-j) + r_j a^n, \qquad 1 \leq j \leq n-1.
\end{equation}
For $j=1, \cdots, n-1$, denote the left side of the relation $\sigma_{j}$ by $\omega_{j}$. We shall view $\Sigma_g$ as a reduction system on $k\langle X \rangle$, as in $\S$\ref{diamond}.

A convenient semigroup ordering on $\langle X \rangle$ to use with $\Sigma_g$ is a \emph{weighted graded lexicographic order}, which we denote by $>_{grlex+}$, and define as follows. 

\begin{definition} Let $w =x_{1}\cdots x_{t} \in \langle X \rangle$, where $x_{j} \in \{a,x\}$ for $j = 1, \ldots , t$. 
\begin{enumerate}
\item[(i)] Define
\begin{enumerate}
\item[$\bullet$] the \emph{length} $|w|$ to be $t$;
\item[$\bullet$] the \emph{$x$-weight} $wt_{x}(w) := |\{i : x_i = x \}|$;
\item[$\bullet$] the \emph{lexicographic order} on $\langle X \rangle$ to be given by declaring $a >_{lex} x$.
\end{enumerate}
\item[(ii)] For $u,v \in \langle X \rangle$, set
$$
	u >_{grlex+} v \; \Leftrightarrow 
	\; (|u|>|v|) \; \vee \;
	(|u|=|v| \,\wedge \,
	wt_x (u) > wt_x ( v))
\; \vee \; (|u|=|v| \,\wedge \,
	wt_x (u) = wt_x ( v) \, \wedge \, u >_{lex} v).
$$
\end{enumerate}
\end{definition}

\begin{lemma}\label{reduce2} Retain the notation and hypotheses as above.
\begin{enumerate}
\item[(i)] $>_{grlex+}$ is a semigroup ordering on $\langle X \rangle$.
\item[(ii)] $>_{grlex+}$ satisfies the descending chain condition.
\item[(iii)] $>_{grlex+}$ is compatible with $\Sigma_g$.
\item[(iv)] $\Sigma_g$ has no nontrivial inclusion ambiguities.
\item[(v)] The only overlap ambiguities in $\Sigma_g$ are $(\sigma_j, \sigma_{j+t}, a^t, a^j x^{n-j-t}, x^t)$, for $1 \leq j < j + t \leq n - 1$.
\end{enumerate}
\end{lemma}
\begin{proof} (i)-(iv) are easy to check. For (v), the listed cases are indeed clearly overlap ambiguities. For the converse, note that $\omega_j$ has total degree $n$, for $j=1, \ldots , n-1$. Now every overlap ambiguity has the form
$$
ABC= \mathrlap{\overbrace{\phantom{a^{r'}(a^{r}x^{s})}}^{\omega_{j'}}}
      a^{r'} 
      \mathrlap{\underbrace{\phantom{(a^{r}x^{s}) x^{s'}}}_{\omega_{j}}}
      (a^{r}x^{s}) x^{s'} 
$$
for some $j' > j$, since $\{\sigma_1, \ldots, \sigma_{n-1}\}$ are the only relations. Comparing degrees,

$$
n = r' + r + s = r +s + s'.$$
Therefore, $r' = s' := t,$ say, and then $j = r = r + r' - t = j' - t$. Thus $j' = j + t$ as claimed.
\end{proof}

Retain the integer $n$, $n > 1$, and continue with $X = \{x,a\}$. Define a subset of $\langle X \rangle$,
\begin{equation} \label{free} \mathcal{L}_{n} := \{ a^{i}x^{j} \; | \; 0 < i,j < n, \; i+j<n \}. \end{equation}

\begin{lemma}\label{irred} Keep the above notation.
\begin{enumerate}
\item[(i)] $|\mathcal{L}_n | = \frac{1}{2}(n-1)(n-2).$
\item[(ii)] The subsemigroup $\langle \mathcal{L}_n \rangle$ of $\langle X \rangle$ generated by $\mathcal{L}_n$ is free of rank $\frac{1}{2}(n-1)(n-2).$ 
\item[(iii)] The set of irreducible words in $\langle X \rangle$ with respect to the reduction system $\Sigma_g$ is
$$ \langle x \rangle \langle \mathcal{L}_n \rangle \langle a \rangle := \{x^i \omega a^{\ell} : i, \ell \in \mathbb{Z}_{\geq 0}, \, \omega \in \langle \mathcal{L}_n \rangle \}. $$
\end{enumerate}
\end{lemma}
\begin{proof} 
(i) This is an easy induction.

\noindent (ii) Write $u_{ij}$ for the element $a^ix^j$ of $\mathcal{L}_n$, and consider $w = u_{i_1 j_1} \cdots u_{i_m j_m} \in \langle \mathcal{L}_n \rangle$. Since each element of $\mathcal{L}_n$ begins with $a$ and ends with $x$, and does not involve the subword $xa$, the given expression for $w$ as an element of $\langle \mathcal{L}_n \rangle$ is unique. Thus $\langle \mathcal{L}_n \rangle$ is free with basis $\mathcal{L}_n$.

\noindent (iii) A word in $\langle X \rangle $ is $\Sigma_g$-reducible if and only if it contains $\omega_j$  as a subword for some $j = 1, \ldots , n-1.$ Thus $ \langle x \rangle \langle \mathcal{L}_n \rangle \langle a \rangle$ consists of irreducible words. For, $\omega_j$ starts with $a$ and ends with $x$, so if $\omega_j$ is a subword of $u = x^{\ell} u_0 a^m \in  \langle x \rangle \langle \mathcal{L}_n \rangle \langle a \rangle$, where $u_0 \in \langle \mathcal{L}_n \rangle$, then $\omega_j $ is a subword of $u_0$. But this is impossible since every word in $\mathcal{L}_n$ has length less than $n$, starts with $a$ and ends with $x$, so that every string of $n$ letters in $u_0$ involves the subword $xa$.

Conversely, if $\omega \in \langle X \rangle$ and $\omega \notin \langle x \rangle \langle \mathcal{L}_n \rangle \langle a \rangle$, then it must contain a subword of the form $a^i x^j$ for $i,j > 0$ and $i + j > n$. Thus $\omega $ is reducible.
\end{proof}

\subsection{Resolving ambiguities.}\label{resolve} The aim in this subsection is to prove the following result.

\begin{theorem}\label{gotcha} Retain the notation and definitions of $\S\S$\ref{nota}, \ref{diamond} and \ref{PBW}. Let $n \in \mathbb{Z}$, $2 \leq n \leq 5$. Then $A_0(x,a,g)$ has PBW basis the set of monomials $ \langle x \rangle \langle \mathcal{L}_n \rangle \langle a \rangle$ defined in Lemma \ref{irred}(iii).
\end{theorem}

Note that it follows from the discussion in $\S$\ref{diamond} and $\S$\ref{PBW}, in particular from Lemma \ref{irred}(iii), that, for \emph{every} $n \geq 2$, $ \langle x \rangle \langle \mathcal{L}_n \rangle \langle a \rangle$ is a spanning set for the vector space $A(x,a,g)$. By Bergman's Theorem \ref{berg} and by Lemma \ref{reduce2}(iv), to prove that this set is linearly independent and hence a $k$-basis it remains only to show that the overlap ambiguities in $\Sigma_{g}$ listed in Lemma \ref{reduce2}(v) are resolvable. We shall achieve this for $n \leq 5$ in Proposition \ref{amb}, for which a couple of preliminary lemmas are needed.

\begin{lemma}
\label{reduce}
Let $n \geq 4$ and let $r,t \in \{ 1,2, \cdots, n-3 \}$ with $r+t < n$. Let $w$ and $v$ be words of length $t$ in $a$ and $x$. Then $wP(r,n-t-r)v$ is reducible $\Leftrightarrow$ there exist $i,j \in \mathbb{Z}_{\geq 0}$ with $i+j=t$ such that $w$ ends with $a^{i}$ and $v$ starts with $x^{j}$.
\end{lemma}

\begin{proof}
Let $i,j \in \mathbb{Z}_{\geq 0}$, $i+j=t$, $w=w_{1}a^{i}$ and $v=x^{j}v_{2}$. Then $wP(r,n-t-r)v$ contains the reducible subword $a^{i}(a^{r}x^{n-t-r})x^{j}$ as claimed.

\noindent Conversely, suppose $wP(r,n-t-r)v$  contains a reducible word. Then, since the left sides of the defining relations (\ref{graded}) are homogeneous of total degree $n$, this reducible word must occur in one of the following homogeneous elements of $k \langle X \rangle$ of total degree $n$:
$$
wP(r,n-t-r), \qquad P(r,n-t-r)v \qquad \text{and} \qquad w_{2}P(r,n-t-r)v_{1}.
$$
The left hand sides of the defining relations begin with $a^{j}$ and end with $x^{n-j}$ for $1 \leq j \leq n-1$. If the reducible word is in $wP(r, n - t- r)$, this forces $w$ to be $a^{t}$ since $r \geq 1$ and $w$ has length $t$. Similarly, $v$ in $P(r,n-t-r)v$ must be $x^{t}$ since $t+r < n$. Finally, a reducible word contained in $w_{2}P(r,n-t-r)v_{1}$, with $w_{2}$ and $v_{1}$ both non-identity words, must have $w_{2}=a^{i}$ and $v_{1}=x^{j}$ for $i+j=t$, since $ r \geq 1$ and $t+r<n$. 
\end{proof}

\begin{lemma}
\label{irreduce}
For $n \geq 4$, $t \in \{ 1,2, \cdots, n-3 \}, \, r \in \{ 1,2, \cdots, n-2 \}$ with $r+t < n$, $w,v$ words of length $t$ in $a$ and $x$, all the words in $wQ(r,n-t-r)v$ are irreducible.
\end{lemma}
\begin{proof}
For $r,t \leq n-3$, this is a corollary of Lemma \ref{reduce} since reducible words occurred there only from the word $a^{r}x^{n-t-r} \in P(r,n-t-r)$, which  no longer appears in $Q(r,n-t-r)$. When $r = n-2$ and $t = 1$ the lemma is clear, since $\omega_j \notin aQ(n-2,1)x$, for all $j = 1, \ldots , n-1$.
\end{proof}

In the following proof we will use the symbol $``\rightarrow"$ whenever we replace a monomial $\omega_{j}$ (that is $a^{j}x^{n-j}$) with the right hand side $\left(-\left(Q(j,n-j) + \sum_{i=j}^{n-1} r_{i}P(j,i-j)\right) + r_{j}a^{n}\right)$ of the $j$th relation. Whenever a linear combination of irreducible words appear during the reduction process, we will underline it. For instance, $wQ(r,n-t-r)v$ from Lemma \ref{irreduce} above is irreducible so is written as $\underline{wQ(r,n-t-r)v}$. We use the symbol $\in$  to indicate that a word $\omega$ appears in an element of $k\langle X \rangle$: for example, $a^j x^{n-j} \in P(j,n-j)$. We shall frequently make use of the identities (\ref{heart}) from the proof of Lemma \ref{xa}, together with companion identities for $Q(r,s)$, namely: for $r \geq 0, s > 0$,
\begin{equation}\label{lungs} Q(r,s) \quad = \quad Q(r,s-1)x +  P(r-1, s)a.
\end{equation}

The arguments used to prove Proposition \ref{amb} are elementary, but long and involved beyond $t=1$. So we include only the proof for the case $t=1$ here, relegating the proofs of the cases $t=2$, $t=3$ to the Appendix.

\begin{proposition}
\label{amb}
Retain the notation of $\S\S$\ref{diamond} and \ref{PBW}. Then the overlap ambiguities 
$$(\sigma_j, \sigma_{j+t}, a^t, a^j x^{n-j-t}, x^t)$$ 
are resolvable when 
\begin{enumerate}
\item[(i)] $t=1$ and $n \geq 3$;
\item[(ii)]$t=2$ and $n \geq 4$;
\item[(iii)] $t=3$ and $n \geq 5$.
\end{enumerate}
\end{proposition}

\begin{proof}
(i) Let $ 1 \leq j \leq n-1$, and consider the overlap ambiguity $\{ \omega_{j}, \omega_{j+1} \}$ for $A(x,a,g)$. That is we consider the word in the free algebra $\field \langle a,x \rangle$ given by
$$
a \omega_{j} = a (a^{j}x^{n-j}) = ( a^{j+1}x^{n-j-1})x=\omega_{j+1}x.
$$
We show that applying either the relation $\sigma_{j}$ or the relation $\sigma_{j+1}$ to this word leads to the same linear combination of irreducible words in $A_0(x,a,g)$. 

Beginning with the resolution of $a\omega_j$, write $\omega_j$ with the aid of (\ref{heart}) and (\ref{lungs}) as
\begin{eqnarray*}
a^{j}x^{n-j} &\rightarrow& - \left( \sum_{i=j}^{n-1} r_{i}P(j,i-j) +   Q(j,n-j) \right) + r_{j}a^{n}\\
&=& - \sum_{i=j}^{n-1} r_{i}P(j-1,i-j)a -  \sum_{i=j+1}^{n-1}r_i P(j,i-j-1)x\\  
&\,& \qquad \qquad   - Q(j,n-j-1)x + P(j-1,n-j)a  + r_{j}a^{n}\\
&=& -\left( \sum_{i=j}^{n} r_{i}P(j-1,i-j)a +   \sum_{i=j+1}^{n-1}r_i P(j,i-j-1)x + Q(j,n-j-1)x \right) + r_{j}a^{n}.
\end{eqnarray*}
Now premultiply the above by $a$, and  use Lemmas \ref{reduce} and \ref{irreduce} to separate reducible and irreducible words, yielding
\begin{eqnarray*}
(\gamma) \qquad \qquad &\,&\\
a( a^{j}x^{n-j}) &\rightarrow& - \left( \sum_{i=j}^{n} r_{i}aP(j-1,i-j)a + \sum_{i=j+1}^{n-1} r_{i}aP(j,i-j-1)x + \underline{aQ(j,n-j-1)x} \right)\\
&\,& \qquad \qquad  + r_{j}\underline{a^{n+1}}\\
&=& -  \sum_{i=j}^{n-1} r_{i}\underline{aP(j-1,i-j)a} -a^jx^{n-j}a -\underline{aQ(j-1,n-j)a}\\ 
&\,& \qquad \qquad - \sum_{i=j+1}^{n-1} r_{i}aP(j,i-j-1)x - \underline{aQ(j,n-j-1)x} + r_{j}\underline{a^{n+1}}.
\end{eqnarray*}
\noindent As indicated by underlining above, the only irreducible words in ($\gamma$) are $a^{j}x^{n-j}a \in  aP(j-1,n-j)a$ and $ a^{j+1}x^{n-j-1} \in a P(j, n-j-2 )x $. We reduce the first of these, first applying the relation $\sigma_j$ to give
\begin{eqnarray*}
- a^{j}x^{n-j} &\rightarrow& \sum_{i=j}^{n-1} r_{i}P(j,i-j) + Q(j,n-j)  -  r_{j}a^{n}\\
 &=&   \sum_{i=j}^{n-1} r_{i}aP(j-1,i-j) +  \sum_{i=j+1}^{n-1} r_{i}xP(j,i-j-1) + aQ(j-1,n-j) +xP(j, n-j-1) -  r_{j}a^{n}\\
&=&   \sum_{i=j}^{n-1} r_{i}\underline{aP(j-1,i-j)}   + \sum_{i=j+1}^{n} r_{i}\underline{xP(j,i-j-1)} + \underline{aQ(j-1,n-j)} -  r_{j}\underline{a^{n}}.
\end{eqnarray*}
\noindent Postmultiplying this by $a$ yields 
$$
- a^{j}x^{n-j}a \rightarrow  \sum_{i=j}^{n-1} r_{i}\underline{aP(j-1,i-j)a} + \underline{aQ(j-1,n-j)a} + \sum_{i=j+1}^{n} r_{i}\underline{xP(j,i-j-1)a}  -  r_{j}\underline{a^{n+1}},
$$
where irreducibility is assured by Lemmas \ref{reduce} and \ref{irreduce}. Substitute this reduction in the reduction ($\gamma$) for $a(a^j x^{n-j})$, to obtain
\begin{eqnarray*}
(\alpha) \qquad \qquad &\,&\\
a( a^{j}x^{n-j}) &\rightarrow& -\sum_{i=j}^{n-1}r_i\underline{aP(j-1,i-j)a}  - \underline{aQ(j-1,n-j)a} -\sum_{i=j+1}^{n-1}r_i\underline{aP(j,i-j-1)x}\\
&\,& \qquad  -\underline{aQ(j,n-j-1)x} + r_j \underline{a^{n+1}} + \sum_{i=j}^{n-1} r_{i}\underline{aP(j-1,i-j)a}\\
&\,& \qquad  + \underline{aQ(j-1,n-j)a} + \sum_{i=j+1}^{n} r_{i}\underline{xP(j,i-j-1)a}  -  r_{j}\underline{a^{n+1}}\\
&=&   -\sum_{i=j+1}^{n-1}r_iaP(j,i-j-1)x + \sum_{i=j+1}^n r_i\underline{xP(j, i-j-1)a} - \underline{aQ(j,n-j-1)x}.
\end{eqnarray*}

Consider now the second side of the ambiguity, namely $\omega_{j+1}x$. The relation for $\omega_{j+1}$ is
\begin{eqnarray*}
a^{j+1}x^{n-j-1} &\rightarrow&  -  \left( \sum_{i=j+1}^{n-1} r_{i}P(j+1,i-j-1) + Q(j+1,n-j-1)  \right) + r_{j+1}\underline{a^{n}}\\
&=&  - \sum_{i=j+1}^{n-1} r_{i}aP(j,i-j-1) - aQ(j,n-j-1) - xP(j + 1, n - j - 2) \\
&\,& \qquad \qquad  - \sum_{i=j+2}^{n-1} r_{i}xP(j+1,i-j-2)  + r_{j+1}a^{n}\\
&=& -  \left( \sum_{i=j+1}^{n-1} r_{i}\underline{aP(j,i-j-1)} + \underline{aQ(j,n-j-1)} + \sum_{i=j+2}^{n} r_{i}\underline{xP(j+1,i-j-2)}  \right) + r_{j+1}\underline{a^{n}}.
\end{eqnarray*}
\noindent Postmultiplying this reduction by $x$ and using Lemmas \ref{reduce} and \ref{irreduce} to separate reducible and irreducible words yields
\begin{eqnarray*}
(\beta) \qquad \qquad \quad &\,&\\
(a^{j+1}x^{n-j-1})x &\rightarrow& -  \left( \sum_{i=j+1}^{n-1} r_{i}aP(j,i-j-1)x + \underline{aQ(j,n-j-1)x} + \sum_{i=j+2}^{n} r_{i}xP(j+1,i-j-2)x  \right) \\
&\,& \qquad \qquad + r_{j+1}\underline{a^{n}x}.
\end{eqnarray*}
\noindent Now $x a^{j+1}x^{n-j-1} \in xP(j+1,n-j-2)x$ and $ a^{j+1}x^{n-j-1} \in a P(j, n-j-2 )x $ are the only reducible words remaining the above reduction. To reduce $xa^{j+1}x^{n-j-1}$, we first write $\omega_{j+1}$ as
\begin{eqnarray*}
-a^{j+1}x^{n-j-1} &\rightarrow&   \sum_{i=j+1}^{n-1} r_{i}\underline{P(j,i-j-1)a} + \sum_{i=j+2}^{n-1} r_{i}\underline{P(j+1,i-j-2)x} + \underline{Q(j+1, n-j-2)x}\\
&\,& \qquad \qquad  + \underline{P(j,n-j-1)a} - r_{j+1}\underline{a^{n}}\\
&=&  \sum_{i=j+1}^{n} r_{i}\underline{P(j,i-j-1)a} + \sum_{i=j+2}^{n-1} r_{i}\underline{P(j+1,i-j-2)x} + \underline{Q(j+1, n-j-2)x}\\
&\,& \qquad \qquad - r_{j+1}\underline{a^{n}}.
\end{eqnarray*}
Premultiplying this reduction by $x$ and using the centrality of $a^n$, Proposition \ref{cent}(iii), we get
\begin{eqnarray*}
- x a^{j+1}x^{n-j-1} &\rightarrow& \left( \sum_{i=j+1}^{n} r_{i}\underline{xP(j,i-j-1)a} + \sum_{i=j+2}^{n-1} r_{i}\underline{xP(j+1,i-j-2)x} + \underline{xQ(j+1, n-j-2)x} \right)\\
&\,& \qquad \qquad - r_{j+1}\underline{a^{n}x}.
\end{eqnarray*}
Substituting this reduction in ($\beta$) yields
\begin{equation}
(\tau) \qquad (a^{j+1}x^{n-j-1})x \rightarrow   \sum_{i=j+1}^{n} r_{i}\underline{xP(j,i-j-1)a}  -  \sum_{i=j+1}^{n-1} r_{i}aP(j,i-j-1)x - \underline{aQ(j,n-j-1)x}.
\end{equation}
Comparing the processes ($\alpha$) and ($\tau$), we conclude that the overlap ambiguity $\{ \omega_{j}, \omega_{j+1} \}$ is resolvable for all $j=1, \cdots, n-1$. This proves (i).
\end{proof}

\begin{proof}[Proof of Theorem \ref{gotcha}] This is immediate from Theorem \ref{berg} and Proposition \ref{amb}.
\end{proof}

It is now a simple matter to deduce the PBW theorem for the corresponding algebras $A(x,a,g)$, as follows.

\begin{corollary}\label{basis} Let $n \in \mathbb{Z}$ with $2 \leq n \leq 5$, and let $A := A(x,a,g)$ be defined as in $\S$\ref{nota}. Then $A$ has $k$-basis
$$ \{  x^{\ell}  \langle \mathcal{L}_n \rangle  a^m : \ell \in \mathbb{Z}_{\geq 0}, m \in \mathbb{Z} \}. $$
\end{corollary}
\begin{proof} Consider the algebra $A_0(x,a,g)$, defined in $\S$\ref{PBW}. By Theorem \ref{gotcha}, $A_0(x,a,g)$ has vector space basis $\langle x \rangle \langle \mathcal{L}_n \rangle \langle a \rangle$, from which it follows that $A_0(x,a,g)$ is a free right $k[a]$-module on the basis $\langle x \rangle \langle \mathcal{L}_n \rangle $. Since $a^n$ is central in $A_0(x,a,g)$ by the proof of Proposition \ref{cent}(iii), $A_0(x,a,g)$ is thus a free left and right $k[a^n]$-module. It follows in particular that $a^n$ is not a zero divisor in $A_0(x,a,g)$, so that we can form the partial quotient algebra $Q$ of $A_0(x,a,g)$ got by inverting the central regular elements $\{a^{nt} : t \geq 0 \}$. It is clear that (\emph{a}) $Q$ has vector space basis $ \{  x^{\ell}  \langle \mathcal{L}_n \rangle  a^m : \ell \in \mathbb{Z}_{\geq 0}, m \in \mathbb{Z} \}, $ and (\emph{b}) $Q$ has the same generators and relations as $A(x,a, g)$. This proves the corollary.
\end{proof}

\section{The nested Hopf algebras $A(x,a,x^n)$}\label{xnstructure}

 In this section we examine the Hopf algebras $A(x,a,g)$ when $g(x)$ is a power of $x$. Although the PBW theorem, Corollary \ref{basis}, is only proved for $A(x,a,x^n)$ when $n \leq 5$, we can use this to obtain structural information for all values of $n$, starting from the simple observation in (ii) below. More precise information for $n \leq 3$ is then obtained in $\S$\ref{small}.

\begin{lemma}\label{epi} Retain the notation of $\S\S$\ref{nota} and \ref{PBW}. Let $n,m \in \mathbb{Z}$ with $m \geq n \geq 2$.
\begin{enumerate}
\item[(i)] $A(x,a,x^m)$ is spanned as a vector space by $\langle x \rangle \langle \mathcal{L}_m \rangle \langle a \rangle$.
\item[(ii)] Let $I(m)$ denote the ideal $\langle P(j,m-j) : 1 \leq j \leq m-1 \rangle$ of $k\langle x,a \rangle$. Then 
$$ I(m) \subseteq I(n), $$
so there is a Hopf algebra surjection from $A(x,a,x^m)$ onto $A(x,a, x^n)$.

\end{enumerate}
\end{lemma}
\begin{proof}(i) This is simply a restatement of Lemma \ref{irred}(iii), which does not require the degree of $g(x)$ to be at most 5.

\noindent (ii) Recall from $\S$\ref{PBW} that for $m\geq 2$ 
$A_{0}(x,a,,x^{m}) = \field \langle x, a \rangle / I(m)$, where $I(m)= \langle P(j,m-j) : 1 \leq j \leq m-1 \rangle$. Let $j= 1, 2, \cdots, m-1$. Then
\begin{eqnarray*}
P(j,m-j) &=& P(j-1,m-j)a + P(j,m-j-1)x \\
&=& P(j-1, (m-1)-(j-1))a + P(j, (m-1)-j)x \\
&\subseteq& I(m-1).
\end{eqnarray*}
Thus, $I(m) \subseteq I(m-1)$ and by induction on $m$, $I(m) \subseteq I(n)$ if $m \geq n$. 

Therefore, setting $K=I(n)/I(m)$, there is a short exact sequence of $k$-algebras,
$$
0 \rightarrow K \rightarrow A_{0}(x,a, x^m) \rightarrow A_{0}(x,a,x^n) \rightarrow 0.
$$
Localise this sequence at the central Ore set
$$
\mathcal{G}= \{ a^{mn\ell} \; | \; \ell \geq 0 \}
$$
of $A_{0}(x,a, x^m)$. By Proposition \ref{cent}(iii), $\mathcal{G}$ consists of central elements modulo $K$. So, by the exactness of Ore localisation at  central elements, we obtain the exact sequence
$$
0 \rightarrow K\mathcal{G}^{-1} \rightarrow A_{0}(x,a,x^m)\mathcal{G}^{-1} \rightarrow A_{0}(x,a,x^n)\mathcal{G}^{-1} \rightarrow 0.
$$
This proves the second part of (ii).
\end{proof}

Recall that a \emph{(right) quantum homogeneous space} $B$ of a Hopf algebra $H$ is a right coideal subalgebra of $H$ such that $H$ is a faithfully flat right and left $B$-module.

\begin{proposition}\label{xnprop} Retain the notation of $\S\S$\ref{nota} and \ref{PBW}, so $F = k[x]\ast k[a^{\pm 1}]$. Let $n,m \in \mathbb{Z}$ with $m \geq n \geq 2$.
\begin{enumerate}
\item[(i)] Construct $A(x,a,x^m)$ as the factor $F/I(m)F$, as in Lemma \ref{epi}(ii). Choose a primitive $n$th root of unity $q$ in $k$. Then $A(x,a, x^m)$ has as a quotient Hopf algebra  the localised quantum plane
$$ k_{q}\langle x, a^{\pm 1} \rangle := F/\langle xa - qax \rangle. $$
\item[(ii)] The commutative subalgebra $k[x,a^{\pm m}]$ of $A(x,a,x^m)$ is a quantum homogeneous subspace of $A(a,x,x^m)$.
\item[(iii)] The polynomial subalgebra $k[x]$ is a quantum homogeneous subspace of $A(x,a,x^m)$, and the Laurent polynomial algebra $k[a^{\pm 1}]$ is a Hopf subalgebra over which $A(x,a,x^m)$ is faithfully flat.
\item[(iv)] Suppose that $m \geq 4$. Then $A(x,a,x^m)$ contains a noncommutative free subalgebra.
\end{enumerate}
\end{proposition}

\begin{proof}(i) By Lemma \ref{epi}(ii) it is enough to prove this when $m = n$. So we need to show that in $F$, for $j = 1, \ldots , n-1$,
\begin{equation}\label{jammy} P(j, n- j) \subseteq \langle xa - qax \rangle. 
\end{equation}
In $k_q \langle x, a^{\pm 1} \rangle = F/\langle xa - qax \rangle$, 
$$
( x + a )^{n} = \sum_{i=0}^{n}
\begin{bmatrix}
n \\
i
\end{bmatrix}_{q}
x^{i}a^{n-i} ,
$$
by the quantum binomial theorem for $q-$commuting variables. Thus, since $q$ is a primitive $n$th root of unity, $( x+a )^{n} = x^{n} + a^{n}.$ But also, again in $k_q \langle x, a^{\pm 1} \rangle$,
$$ (x + a)^n = x^n + a^n + \sum_{j=1}^{n-1} P(j, n-j).$$
Thus $P(j,n-j) = 0$ in $k_q \langle x, a^{\pm 1} \rangle$  for $1 \leq j \leq n-1$, since no cancellation can occur between these terms when straightening out the monomials on the right side of the above identity, as the number of $a$'s and $x$'s in a monomial does not change when applying the $q$-commutation identity. Thus (\ref{jammy}) is proved.

\noindent (ii) Let $q$ be a primitive $m$th root of unity in $k$. From (i), the subalgebra $ C := k\langle x, a^{\pm m} \rangle$ of $A(x,a,x^m)$ maps, under the factor by $\langle xa - qax \rangle$, onto the subalgebra $D$ of the quantum plane generated by $x$ and $a^{\pm m}$. But the latter algebra is precisely $k[x, a^{\pm m}]$. However $C$ is commutative because $a^m$ is central in $A(x,a,x^m)$ by Proposition \ref{cent}(iii), so the map from $C$ onto $D$ must be an isomorphism, since every proper factor of the domain $k[x, a^{\pm m}]$ kas GK-dimension strictly less than 2. Since $x$ is $(1,a)-$primitive and $a^m$ is grouplike, $C$ is a right coideal subalgebra of $A(x,a,x^m)$. Since $S(C_0) = S(k[a^{\pm m}]) \subseteq C$, it follows from \cite[Theorem 1.3]{Ma} that $A(x,a,x^m)$ is a faithfully flat right and left $C$-module.

\noindent (iii) The arguments are similar to those for (ii) and are left to the reader.

\noindent (iv) Since $m \geq 4$, Lemma \ref{epi}(ii) can be applied with $n=4$. Coupled with Corollary \ref{basis} it guarantees that $A(x,a,x^m)$ has a factor algebra containing a noncommutative free subalgebra. Lifting this free factor to a necessarily free pre-image in $A(x,a,x^m)$ proves the claim.
\end{proof}

\section{The Hopf algebras $A(x,a,g)$, for $g(x)$ of degree at most 3}\label{anyg}

\subsection{$g(x)$ of degree 2: the quantum Borel.}\label{small}

\begin{proposition}\label{degtwo} Keep the notation of $\S$\ref{nota}, and let $n=2$, so that $g(x) = x^2 + r_1 x$, for $r_1 \in k$. 
\begin{enumerate}
\item[(i)] The Hopf algebra structure of $A(x,a,g)$ is independent of $g$, that is of the parameter $r_1$. Namely, given $r_1 \in k$, set $x':=  x + \frac{r_{1}}{2} ( 1 - a )$. Then $A(x,a,g) = k \langle a^{\pm 1}, x' \rangle$ is isomorphic as a Hopf algebra to the Borel  in $U_{-1}(\mathfrak{sl}(2))$, or equivalently to a localised quantum plane  $k \langle a^{\pm 1}, x' \rangle$ at the parameter $-1$. Moreover, $A_0(x,a,g)$ is isomorphic as an algebra to the quantum plane with parameter $-1$.
\item[(ii)] $A(x,a,g)$ is a noetherian AS-regular domain of Gelfand-Kirillov and global dimension 2.
\item[(iii)] $A(x,a,g)$ is a finite module over its central Hopf subalgebra $k[a^{\pm 2}][g]$, where $g$ is the square of the uninverted $(1,a)$-primitive generator of the quantum Borel.
\end{enumerate}
\end{proposition}
\begin{proof}
(i) Recall the relation $\sigma_1 (x,a,g)$ from (\ref{ij}), 
$$ r_1 P(1,0) + P(1,1) - r_1a^2 = 0.$$
Rewrite this as
\begin{equation}\label{qplane}
a\left( x + \frac{r_{1}}{2} ( 1 - a )  \right) +  \left( x + \frac{r_{1}}{2} ( 1 - a )  \right)  a = 0,
\end{equation}
and define $x':=  x + \frac{r_{1}}{2} ( 1 - a ) \in A(x,a,g)$. Thus $A(x,a,g)$  is the quantum plane $\field_{-1}\langle x',a^{\pm 1}\rangle$ with $a$ grouplike and $x'$ $(1,a)$-primitive, as required. The corresponding statement regarding $A_0(x,a,g)$ is also immediate from (\ref{qplane}).

\noindent(ii) Being an iterated skew polynomial algebra by (i), $A(x,a,g)$ has finite global dimension. Since it is an affine noetherian Hopf algebra satisfying a polynomial identity, it is AS-Gorenstein and GK-Cohen-Macaulay by \cite[$\S$6.2]{BrMa}. Thus, $A(x,a,g)$ is AS-regular. Since it has GK-dimension 2 by virtue of being a finite module over $k[a^{\pm 2}, x'^2]$, its global dimension is also 2 since it is GK-Cohen Macaulay.

\noindent (iii)  It is easy to check that $k\langle a^{\pm 2}, x'^2 \rangle = k \langle a^{\pm 2}, g \rangle$, and that $A(x,a,g)$ is a finite module over this subalgebra. Its structure as a Hopf subalgebra is well-known and easy to check; or one can use Proposition \ref{cent}(i).
\end{proof}

\subsection{$g(x)$ of degree 3: localised down-up algebras and their deformations.} 

Let $g(x) = r_1 x + r_2 x^2 + x^3,$ with $r_1, r_2 \in k$. Recall from (\ref{ij}) in $\S$\ref{nota} that the defining relations in the free algebra $k\langle x,a \rangle$ of the subalgebra $A_0(x,a,g) $ of $A(x,a,g)$  are
\begin{eqnarray*} \sigma_1:& \qquad ax^2 \; =& -xax -x^2a -r_2(ax + xa) - r_1a + r_1 a^3\; .\\
\sigma_2:& \qquad a^2x \; =& -axa - xa^2 - r_2a^2 + r_2a^3  \; . 
\end{eqnarray*}

Compare the above relations with the following:

\begin{definition}[Benkart, Roby \cite{BENKART}]\label{downup}
Let $\alpha, \beta, \gamma \in \mathbb{C}$. The \emph{down-up algebra} $A=A(\alpha, \beta, \gamma)$ is the $\mathbb{C}$-algebra  with generators $d, u$ and relations
\begin{eqnarray*}
d^{2}u &=& \alpha dud + \beta ud^{2} + \gamma d, \\
du^{2} &=& \alpha udu + \beta u^{2}d + \gamma u.
\end{eqnarray*} 
\end{definition}

The relation between the two presentations is encompassed in the following concept, introduced in \cite[Section 3]{BG2}. Here, we slightly weaken the usual requirement that the generators of the free algebra are assigned degree 1, in order to allow for the terms in $a^3$ in the relations for $A_0(x,a,g)$.

\begin{definition}\label{deform} Let $V$ be a $k$-vector space with basis $v_1, v_2, \ldots , v_n$, and assign a grading to the tensor algebra $T(V)$ by setting $\mathrm{degree}( v_i) = d_i$ for some positive integers $d_1, d_2, \ldots , d_n$. Denote the subspace of  $T(V)$ consisting of the homogeneous elements of degree $d$ by $T(V)_d$, and write $T(V)^d$ for $\oplus_{i=o}^d T(V)_i$.  Let $N$ and $t$  be positive integers, and let $w_1, w_2, \ldots , w_t \in T(V)^N$. Let $\pi$ be the canonical projection from $T(V)^N$ onto $T(V)_N$. Write $s_i := \pi(w_i)$ for $i = 1, \ldots , t$, and define $A = T(V)/\langle s_1, \ldots , s_t \rangle$ and $U = T(V)/ \langle w_1, \ldots , w_t \rangle$. Then $U$ is called a \emph{PBW-deformation} of $A$ if $A \cong \mathrm{gr}U$.
\end{definition} 

Note that, in the setting of the above definition, there is clearly a canonical surjection from $A$ onto $\mathrm{gr}U$, since the latter is $T(V)/\mathrm{gr}\langle w_1, \ldots , w_t \rangle$. The point of the definition is to isolate the circumstance in which this map is an isomorphism.

For undefined terminology in the following result, see for example \cite{bz}.

\begin{proposition}\label{three} Retain the notation of $\S$\ref{nota} and of (\ref{free}) in $\S$\ref{PBW}, and let $g(x) = r_1x + r_2 x^2 + x^3.$ 
\begin{enumerate}
\item[(i)] $A_0(x,a,x^3)$ is isomorphic to the downup algebra $A(-1,-1,0)$.
\item[(ii)] $A_0(x,a,x^3)$ is a noetherian domain.
\item[(iii)] $A_0(x,a,x^3)$ is Auslander regular of global dimension 3.
\item[(iv)] $A_0(x,a,x^3)$ has PBW basis $\langle x \rangle \langle \mathcal{L}_3 \rangle \langle a \rangle$; that is, it has basis
$$ \{ x^{\ell}(ax)^i a^j : \ell, i ,j \in \mathbb{Z}_{\geq 0}\}.$$
\item[(v)] $\mathrm{GKdim}(A_0(x,a,x^3)) = 3$, and $A_0(x,a,x^3)$ is GK-Cohen Macaulay.
\item[(vi)] $A_0(x,a,g)$ is a PBW deformation of $A_0(x,a,x^3)$. That is, $A_0(x,a,g)$ is a PBW deformation of the down-up algebra $A(-1,-1,0)$.
\item[(vii)] Statements (ii)-(v) apply verbatim to $A_0 (x,a,g)$ and to $A(x,a,g)$. Moreover, $A(x,a,g)$ is a PBW deformation of $A(x,a,x^3)$, and the PBW basis of $A(x,a,g)$ is $ \{ x^{\ell}(ax)^i a^j : \ell, i  \in \mathbb{Z}_{\geq 0}, j \in \mathbb{Z}\}.$.
\item[(viii)] $A(x,a,g)$ is AS-regular of dimension 3.
\end{enumerate}
\end{proposition}

\begin{proof}
\noindent (i) is clear from the definitions.

\noindent (ii) This is immediate from (i) and \cite[Main Theorem]{kirk}, where it is proved that $A(\alpha,\beta,\gamma)$ is a noetherian domain provided $\beta \neq 0$.

\noindent (iii) Immediate from (i) and \cite[Theorem 4.1 and Lemma 4.2(i)]{kirk}.

\noindent (iv) This is a special case of Theorem \ref{gotcha}. In the light of (i), it is also a special case of \cite[Theorem 3.1]{BENKART}.

\noindent (v) For down-up algebras, these results are obtained in \cite[Corollary 3.2]{BENKART} and \cite[Lemma 4.2(ii)]{kirk}, so apply here in view of (i).

\noindent (vi) For all $g$ of degree 3, $U := A_0(x,a,g)$ has the same PBW basis as described in (iv), by Theorem \ref{gotcha}. Let $V = ka \oplus kx$ and filter $T(V)$ by setting $x$ to have degree $2$ and $a$ to have degree $1$. Then the surjection referred to before the proposition, in this case from $A := A_0(x,a,x^3)$ to $\mathrm{gr}U$,  is an isomorphism, as claimed.

\noindent (vii) We prove these statements for $A_0(x,a,g)$. Extending the conclusions to the localisation $A(x,a,g)$ of $A_0(x,a,g)$ at the central regular Ore set $\{a^{3\ell }: \ell \in \mathbb{Z}_{\geq 0} \}$ is straightforward, and is left to the reader. 

That $A_0(x,a,g)$ is a noetherian domain follows from (ii), (vi) and \cite[Proposition 1.6.6 and Theorem 1.6.9]{robson}. Moreover, the fact (vi) that $A_0(x,a,g)$ has the same PBW basis as $A_0(x,a,x^3)$ ensures that $\mathrm{GKdim}A_0(x,a,g) = 3$, from (v); indeed, this also follows from the following paragraph. 

Denoting the filtration of $A_0(x,a,g)$ defined in the proof of (vi) by $\mathcal{F}$, we know from (vi) that this is a finite dimensional filtration whose associated graded algebra is affine and noetherian by (vi) and (ii). Thus, for any finitely generated $A_0(x,a,g)$-module $M$ given a good $\mathcal{F}$-filtration, 
$$ \mathrm{GKdim}_{A_0(x,a,g)}(M) = \mathrm{GKdim}_{\mathrm{gr}_{\mathcal{F}}A_0(x,a,g)}(\mathrm{gr}_{\mathcal{F}}(M)) $$
by \cite[Proposition 8.6.5]{robson}. 

For a module $M$ over the ring $R$, write $j_R(M)$ for the homological grade of $M$, that is $j_R(M) := \mathrm{inf}\{i : \mathrm{Ext}_R^i(M,R) \neq 0 \}$. For every finitely generated $A_0(x,a,g)$-module $M$,
$$ j_{A_0(x,a,g)}(M) =  j_{\mathrm{gr}_{\mathcal{F}}A_0(x,a,g)}(\mathrm{gr}_{\mathcal{F}}(M) )$$
by \cite[Chapter III, $\S$2.5, Theorem 2]{LHVO}.

Let $M$ be a finitely generated $A_0(x,a,g)$-module. Then, by (v) and the above equalities,
\begin{eqnarray*} j_{A_0(x,a,g)}(M) + \mathrm{GKdim}_{A_0(x,a,g)}(M) &=&  j_{\mathrm{gr}_{\mathcal{F}}A_0(x,a,g)}(\mathrm{gr}_{\mathcal{F}}(M)) +  \mathrm{GKdim}_{\mathrm{gr}_{\mathcal{F}}A_0(x,a,g)}(\mathrm{gr}_{\mathcal{F}}(M))\\
&=&  j_{A_0(x,a,x^3)}(\mathrm{gr}_{\mathcal{F}}(M)) +  \mathrm{GKdim}_{A_0(x,a,x^3)}(\mathrm{gr}_{\mathcal{F}}(M))\\
&=& \mathrm{GKdim}(A_0(x,a,x^3)) \\
&=& \mathrm{GKdim}(A_0(x,a,g)) = 3,
\end{eqnarray*}
so that $A_0(x,a,g)$ is GK-Cohen-Macaulay.

By (vi), (iii) and \cite[Corollary 7.6.18]{robson}, writing $\mathrm{gl.dim}$ for the global dimension of an algebra, 
\begin{equation}\label{upper} \mathrm{gl.dim}A_0(x,a,g) \leq \mathrm{gl.dim}A_0(x,a,x^3) = 3.
\end{equation}

Finally, to see that $ \mathrm{gl.dim}A_0(x,a,g) \geq 3$, apply the Cohen-Macaulay property with $M$ equal to the trivial module $k$.

\noindent (viii) Note that $A(x,a,g)$ is a noetherian Hopf algebra by Theorem \ref{gee}. Thus (viii) follows from (iii), (v), (vii) and \cite[Lemma 6.1]{bz}, which states that an Auslander regular noetherian Hopf algebra is AS-regular.

\end{proof}

The downup algebra $A(-1,-1,0)$ is a finite module over its centre, which is itself affine, by \cite[Corollary 2.0.2 and Lemma 2.0.1]{KULKARNI2001431}, \cite[Theorem 1.3(f)]{Zhao}. The next result generalises these facts to the algebras $A_0(x,a,g)$ and $A(x,a,g)$, for an arbitrary degree 3 polynomial $g(x)$. The proof proceeds via a result of independent interest, also obtained by Kulkarni \cite{kulk} in the down-up case: namely we show that the algebras $A_0(x,a,g)$ and $A(x,a,g)$ are maximal orders when $g(x)$ has degree at most 3\footnote{When $g(x)$ has degree 2, so that $A_0(x,a,g)$ and $A(x,a,g)$ are (localised) quantum planes by Proposition \ref{degtwo}(i), this is immediate from \cite[Corollaire V.2.6]{Maury}}. The definition of the term \emph{maximal order} as applied to a prime noetherian ring $R$ is given for example in \cite[5.1.1]{robson}. In \cite[Theorem 5.3.13]{robson} it is shown that, when such a ring $R$ is a finite module over its centre $Z$, this definition of maximal order coincides with the concept of a \emph{maximal classical $Z$-order}. The argument to show that the algebras are finite over their centres can be abstracted in the following form, which is simply an assembly of well-known results, but which is perhaps worth recording explicitly. The key point, that the maximal order property lifts to a filtered deformation, is due to Chamarie, \cite{cham}.

\begin{theorem}\label{maxo} Let $k$ be a field and let $R$ be an affine $k$-algebra satisfying a polynomial identity. Suppose that $R$ has an $\mathbb{Z}_{\geq 0}$-filtration $\mathcal{F}$ such that $\mathrm{gr}_{\mathcal{F}}(R)$ is a domain and a noetherian maximal order. Then $R$ is a domain and a noetherian maximal order, and is a finite module over its centre $Z(R)$, which is a normal affine domain.   
\end{theorem} 
\begin{proof} That the stated properties of $\mathrm{gr}_{\mathcal{F}}(R)$ all lift to $R$ is guaranteed, respectively, by \cite[Proposition 1.6.6, Theorem 1.6.9 and Theorem 5.1.6]{robson}. But now, since $R$ is a maximal order satisfying a polynomial identity, it is its own trace ring by \cite[Proposition 13.9.8]{robson}. Thus \cite[Proposition 13.9.11(ii)]{robson} implies that $R$ is a finite module over its centre $Z(R)$, with $Z(R)$ an affine $k$-algebra because $R$ is affine, thanks to the Artin-Tate lemma \cite[Lemma 13.9.10(ii)]{robson}.   Finally, the fact that $R$ is a maximal order implies that $Z(R)$ is normal, \cite[Proposition 5.1.10(b)]{robson}.
 \end{proof}

The pay-off in part (iv) below now follows easily by combining the above with Kulkarni's work on the down-up algebra:

\begin{proposition}\label{cube} Keep the hypotheses and notation of Proposition \ref{three}, so in particular $g(x) = r_1 x + r_2 x^2 + x^3$.

\begin{enumerate}
\item[(i)] Let $\lambda$ be a primitive cube root of unity in $k$. The centre of $A_0(x,a,x^3)$ is the subalgebra
 
\begin{eqnarray*} Z(A_0(x,a,x^3)) &\; =\;& k \langle a^3, x^3, (xa)^3 - 3\lambda^2 (xa)^2(ax) + 3 \lambda (xa)(ax)^2 - (ax)^3,\\ 
&\,\;& \qquad (xa)^3 - 3 \lambda (xa)^2(ax) + 3 \lambda^2 (xa)(ax)^2 - (ax)^3, (ax)^2 - x^2 a^2 \rangle;
\end{eqnarray*}

\noindent and the centre of $A(x,a,x^3)$ is $Z(A_0(x,a,x^3))[a^{-3}]$.
\bigskip

\item[(ii)] The centres of  $A_0(x,a,g)$ and of $A(x,a,g)$ contain the subalgebra
$$\field\langle g, a^3, (ax)^2 - x^{2}a^{2} -  r_2 xa^{2} - r_1 a^{2} \rangle.$$

\item[(iii)] Both $A_0(x,a,g)$ and $A(x,a,g)$ are domains and maximal orders.

\item[(iv)]  Both $A_0(x,a,g)$ and $A(x,a,g)$ are finite modules over their centres, which are affine normal domains.
\end{enumerate}
\end{proposition}

\begin{proof}
\noindent (i) By Proposition \ref{three}(i), $A_0(x,a,x^3)$ is isomorphic to the down-up algebra $A(-1,-1,0)$. Let $R := k[X,Y]$ and let $\theta$ be the $k$-algebra automorphism of $R$ given by $\theta (X) = -X - Y$ and $\theta (Y) = X$. Use these ingredients to define the so-called \emph{hyperbolic ring}
$$ H := \langle R, u,w : ur = \theta (r) u, wr = \theta^{-1} (r)y \forall r\in R; \; uw = X, wu = Y \rangle. $$
By \cite[Proposition 3.0.1]{KULKARNI2001431}, $H$ is isomorphic to $A(-1,-1,0)$, so composing these isomorphisms we find that $A_0(x,a,x^3)$ is isomorphic to $H$ via the map $x \mapsto u$, $a \mapsto w$.

Now $\theta^3 = 1$, so by \cite[Corollary 2.0.2]{KULKARNI2001431},
$$ Z(H) = R^{\theta}\langle u^3, w^3 \rangle,$$
where $R^{\theta}$ denotes the fixed subalgebra of $R$ under the action of $\theta$. This fixed ring can be routinely calculated (see for example \cite[Theorem 4.0.3]{KULKARNI2001431}). Transferring the resulting description of $Z(H)$ back along the isomorphism to $A_0(x,a,x^3)$ yields the desired description of $Z(A_0(x,a,x^3))$. Alternatively, one can proceed in a similar way, starting from \cite[Theorem 1.3(f)]{Zhao}.

The second part of (i) follows easily from the first, since $A(x,a,x^3) = A_0(x,a,x^3)\langle a^{-3} \rangle$, with $a^3$ central.

\noindent (ii) By Proposition \ref{cent}(iii), $g, a^3 \in Z(A_0(x,a,g))$. It is straightforward  to check that $axax - x^{2}a^{2} -  r_2 xa^{2} - r_1 a^{2}$ commutes with $a$ and $x$. Since $A(x,a,x^3)$ is the localisation of $A_0(x,a,x^3)$ at the central regular Ore set $\{a^{3\ell} : \ell \geq 0 \}$, the listed elements are central in $A(x,a,x^3)$.

\noindent (iii) In view of Proposition \ref{three}(i), that $A_0(x,a,x^3)$ is a maximal order follows from the corresponding result for the down-up algebra $A(-1,-1,0)$, namely \cite[Theorem 2.6]{kulk}. As explained above, the maximal order property lifts from $A_0(x,a,x^3)$ to its filtered deformation $A_0(x,a,g)$ by \cite[Theorem 5.1.6]{robson}.

It is easy to see from the definition that the localisation of a maximal order at a central regular Ore set is again a maximal order, so the desired conclusion for $A(x,a,g)$ also follows.

\noindent (iv) Given (iii) and Theorem \ref{maxo}, the desired results will follow if $A_0(x,a,g)$ (and hence also its localisation $A(x,a,g)$) satisfy a polynomial identity. This is easy to see: by (ii), the subalgebra 
$$ C :=  k \langle g, a^3, (ax) \rangle $$
of $A_0(x,a,g)$ is commutative, and it is clear from the PBW theorem for $A_0(x,a,g)$, Theorem \ref{PBW}, that $A_0(x,a,g)$ is a finitely generated right or left $C$-module. Therefore $A_0(x,a,g)$ satisfies a polynomial identity by \cite[Corollary 13.1.13(iii)]{robson}, as required. 
\end{proof}

\begin{remark} \rm{Note that the generators listed in (i) above for $Z(A_0(x,a,x^3))$ are permuted by the $k$-algebra involution of $A_0(x,a,x^3)$ which interchanges $a$ and $x$, and that the third listed element of $Z(A_0(x,a,g))$ in (ii) is a lift of the involution-invariant generator $(ax)^2 - a^2 x^2$ of $Z(A_0(x,a,x^3))$. It would be interesting to determine $Z(A_0(x,a,g))$ for  a general polynomial $g(x)$ of degree 3, and in particular to discover whether it contains elements which are lifts of the other listed generators of $Z(A_0(x,a,x^3))$.}
\end{remark}


\section{The Hopf algebra $A(g,f)$}\label{tensor}

\subsection{Definition of $A(g,f)$} \label{combine}Suppose the equation of a plane curve $\mathcal{C}$ in $\field^{2}$ is \emph{decomposable}; that is, the equation can be written in the form $f(y)=g(x)$ with $f(y)= \sum_{i=1}^{m}s_{i}y^{i} \in \field[y]$, $g(x) = \sum_{i=1}^{n} r_{i}x^{i} \in \field[x]$, where $r_i,s_i \in k,$ and $r_n, s_m \in k \setminus \{0\}$.  

We conjecture the existence of a Hopf algebra $A(g,f)$ which contains the coordinate ring $\mathcal{O}(\mathcal{C})$ of the curve $\mathcal{C}$ as a right coideal subalgebra. (When $n$ or $m$ is 1 this is of course trivial, since then $\mathcal{O}(\mathcal{C}) = k[t]$ is itself a Hopf algebra. We prove the conjecture here with the imposition of an additional hypothesis on $f$ and $g$, required to ensure that $k[x]$ and $k[y]$ are quantum homogeneous subspaces of $A(x,a,g)$ and $A(y,b,f)$ respectively. Namely, we shall assume in $\S$\ref{combprop} that
\begin{eqnarray*} \mathbf{(H)} \quad &\,&\textit{each polynomial } g(x) \textit{ and } f(y) \textit{ either has}\\
&\,& \textit{degree at most 5 or is a power of }x \textit{ resp. }y.
\end{eqnarray*}

We begin, though, without assuming $\mathbf{(H)}$. To construct $A(g,f)$, given polynomials $f(y)$ and $g(x)$ as in the first sentence above, first form the Hopf algebras $A(x,a,g)$ and $A(y,b,f)$ as in $\S$\ref{nota}. Then consider their tensor product
$$ T \quad := \quad A(x,a,g)  \otimes_k  A(y,b,f).$$
Thus $T$ is a Hopf $k$-algebra, inheriting the relevant structures from its component Hopf subalgebras $A(x,a,g)$ and $A(y,b,f)$ in view of Theorem \ref{gee}. Moreover, $T$ is affine, $T = k \langle x,y, a^{\pm 1}, b^{\pm 1} \rangle$, where, here and henceforth, we simplify notation by writing $x$ for $x \otimes 1$, $y$ for $1 \otimes y$, etc. Since $a$ and $b$ are grouplike and $x$ and $y$ are skew primitive, by Lemmas \ref{hopf} and \ref{xa}, $T$ is generated by grouplike and skew primitive elements, and is therefore pointed, by \cite[Corollary 5.1.14(a)]{radford}. The elements $f,g, a^n, b^m$ are in the centre of $T$ by Proposition \ref{cent}(iii), so that the right ideal 
$$ I \quad := \quad T(g-f) + T(a^n - b^m) $$
of $T$ is actually two-sided. We can therefore define the $k$-algebra
\begin{equation}\label{hat} A(g,f) \quad := \quad T/I.
\end{equation}
In the theorem below and later, we'll continue with the abuse of notation used earlier, writing $x, a$ and so on for the images of these elements of $T$ in various factor algebras, in situations where we believe confusion is unlikely. For the reader's convenience the relations for $A(g,f)$ are listed in Theorem \ref{summit}(i), even though they are easily read off from (\ref{ij}) and (\ref{hat}).

\begin{theorem}\label{summit}
Keep the notation introduced in the above paragraphs, but don't assume $\bf{(H)}$.
\begin{enumerate}
\item[(i)] $A(g,f)$ is the factor $k$-algebra of the free product $k \langle x,y\rangle \ast k \langle a^{\pm 1}, b^{\pm 1} \rangle$ by the ideal generated by the relations
\begin{eqnarray*} [x,y] = [x,b] = [a,b] = [a,y] = 0;&\;& f(y) = g(x), \quad a^n = b^m,\\
\sum_{i=j}^{n} r_{i} P(j,i-j)_{(a,x)} - r_{j}a^{n},&\,& \; \; (j = 1, \ldots , n-1),\\
\sum_{\ell=p}^{m} r_{\ell} P(p,\ell - p)_{(b,y)} - r_{p}b^{m}, &\,& \; \; (p = 1, \ldots , m-1).
\end{eqnarray*}
\item[(ii)] The $k$-algebra $A(g,f)$ inherits a Hopf algebra structure from $T$. Thus its coproduct $ \Delta $, counit $ \varepsilon $ and antipode $S$ satisfy:
\begin{gather*}
	\Delta(x) 
= 1 \otimes x + x \otimes
a,\qquad
	\Delta(y) 
= 1 \otimes y + y \otimes b,\\
	\Delta (a)=a \otimes a,\quad
	\Delta(b) 
= b \otimes b,\quad
	\varepsilon(x) = 0, \quad
	\varepsilon(y)=0, \quad
	\varepsilon(a)= \varepsilon(b)=1,\\
	S(x) =-  xa^{-1},\quad
	S(y) = - yb^{-1},\quad
	S(a) = a^{-1} \qquad
           S(b) = b^{-1}.
\end{gather*}
\end{enumerate}
\end{theorem}
\begin{proof}
Given the above discussion and the results of $\S$\ref{1}, it is enough to show that $I$ is a Hopf ideal of $T$. This is an easy consequence of the facts that $a^n$ and $b^m$ are grouplike, and $g$ and $f$ are respectively $(1,a^n)$-skew primitive and $(1,b^m)$-skew primitive, by Proposition \ref{cent}(i). To see that $S(I) \subseteq I$ one can either calculate directly or appeal to \cite[Proposition 7.6.3]{radford}.
\end{proof}


\subsection{Properties of $A(g,f)$ under hypothesis \bf{(H)}}\label{combprop} To describe the PBW basis for $A(g,f)$ it is necessary to decorate the notation for the PBW generators of $A(x,a,g)$ introduced at (\ref{free}) in $\S$\ref{PBW}. Namely, for $g(x)$ of degree $n$ and $f(y)$ of degree $m$, define subsets of (respectively) the free semigroups on generators $\{x,a\}$ and $\{y,b\}$,
$$\mathcal{L}_n (a,x) := \{a^ix^j :  i,j >0,\, i + j < n \}, $$
and 
$$\mathcal{L}_m (b,y) := \{b^iy^j :  i,j >0,\, i + j < m \}. $$
As before, let $\langle \mathcal{L}_n(a,x) \rangle$ and $\langle \mathcal{L}_m(b,y) \rangle$ denote the free subsemigroups of $\langle a,x \rangle$ and $\langle b,y \rangle$ generated by these sets. 

In the proof of the next theorem we use some elementary properties of faithful flatness whose proofs we have not been able to locate in the literature, although closely related statements in a commutative setting can be found for example in \cite{Ho}. Namely, let $R$ and $S$ be rings, $I$ an ideal of $R$ and $M$ a left $R$-module. Then
\begin{enumerate}
\item[(1)] If $M$ is a faithfully flat $R$-module, then $M/IM$ is a faithfully flat $R/I$-module.
\item[(2)] If $\theta:R \longrightarrow S$ is a ring homomorphism and $S$ is a a faithfully flat (left) $R$-module, then $\theta$ is injective.
\end{enumerate}

\begin{theorem}\label{coideal} Retain the notations introduced in $\S$\ref{combine}, and assume hypothesis $\mathbf{(H)}$.
\begin{enumerate}
\item[(i)]  The coordinate ring $\mathcal{O}(\mathcal{C})$ of the plane curve $\mathcal{C}$ is a quantum homogeneous space of the Hopf algebra $A(g,f)$.
\item[(ii)] $A(x,a,g)$ and $A(y,b,f)$ are Hopf subalgebras of $A(g,f)$, and $A(g,f)$ is faithfully flat over these subalgebras.
\item[(iii)] Assume that $f$ and $g$ have degrees $m$ and $n$ respectively, with $2 \leq m,n \leq 5$. Let $\{c_{\ell} : \ell \in \mathbb{Z}_{\geq 0} \}$ be a $k$-basis for $\mathcal{O}(\mathcal{C})$.
\begin{enumerate}
\item[(a)] $A(g,f)$ has PBW basis
$$ \{ c_{\ell}\langle\mathcal{L}_n(a,x)\rangle \langle \mathcal{L}_m (b,y)\rangle a^ib^j : \ell \in \mathbb{Z}_{\geq 0}, i \in \mathbb{Z}, 0 \leq j < m \}. $$

\item[(b)] $A(g,f)$ is a free left $\mathcal{O}(\mathcal{C})-$module with basis 

$$ \{ \langle \mathcal{L}_n(a,x)\rangle \langle \mathcal{L}_m (b,y)\rangle a^ib^j :  i \in \mathbb{Z}, 0 \leq j < m \}. $$
\end{enumerate} 
\end{enumerate} 
\end{theorem}
\begin{proof} (i) Under hypothesis {\bf(H)}, the subalgebra $k\langle x, a^{\pm n} \rangle$ of $A(x,a,g)$ is isomorphic to $k[x, a^{\pm n}] $;  this follows from the centrality of $a^{\pm n}$, Proposition \ref{cent}(iii), together with the PBW theorem Corollary \ref{basis} when $g(x)$ has degree at most 5, and by Proposition \ref{xnprop}(ii) when $g(x) = x^n$. The same applies to the subalgebra $k\langle y, b^{\pm m} \rangle = k[y, b^{\pm m}]$ of $A(y,b, f)$. Thus $R := k[x,y,a^{\pm n}, b^{\pm m}]$ is a subalgebra of $T$. Note that $R$ is a quantum homogeneous space of $T$, just as was the case in Proposition \ref{xnprop}(ii) for its two components in their respective Hopf algebras  - that is, it is a right coideal subalgebra of $T$ which contains the inverses of all its grouplike elements, so Masuoka's theorem \cite[Theorem 1.3]{Ma} applies. In particular, $T$ is a faithfully flat left and right $R$-module. 

It follows from (1) above that the algebra $T/(a^n - b^m)T$ is a faithfully flat $R/(a^n - b^m)R$-module. Observe that $R/(a^n - b^m)R$ is the group ring $k[x,y]G$, where $G$ is the group $\langle a^{\pm 1}, b^{\pm 1} : [a,b] = 1, a^n = b^m \rangle$. In particular, $R/(a^n - b^m)R$ is a free $k[x,y]$-module, so that $T/(a^n - b^m)T$ is a faithfully flat $k[x,y]$-module. Hence, by (2) above, $k[x,y]$ embeds in $T/(a^n - b^m)T$. A second application of (1) and (2), this time to the ideal $(g^n - f^m)k[x,y]$ of $k[x,y]$, now implies that $T/(a^n - b^m)T + (g^n - f^m)T$ is a faithfully flat $\mathcal{O}(\mathcal{C})$-module. That is, again appealing to (2), $\mathcal{O}(\mathcal{C})$ embeds in $A(g,f)$ and is a quantum homogeneous space of $A(g,f)$.

\noindent (ii) In a similar way to (i), $k[f, b^{\pm m}]$ is a right coideal subalgebra of $A(y,b, f)$  and hence, again using \cite[Theorem 1.3]{Ma}, $A(y,b,f)$ is a faithfully flat left and right $k[f, b^{\pm m}]$-module. Define 
$$D \; := \; A(x,a,g) \otimes k[f, b^{\pm m}] \subseteq T, $$
so that $D \cong A(x,a,g)[f, b^{\pm m}]$ and $T$ is a faithfully flat left and right $D$-module. Let $J = (f-g)D + (b^{m} - a^n)D$, an ideal of $D$ with $D/J \cong A(x,a,g)$ and $T/JT = A(g,f)$. By (1) above, $T/JT$ is a faithfully flat left and right $A(x,a,g)$-module. In particular, by (2) above, $A(x,a,g)$ embeds in $A(g,f)$.

The argument for $A(y,b,f)$ is exactly similar.  

\noindent (iii) Both parts are similar to (i), but easier: thanks to Corollary \ref{basis} there is an explicit PBW basis for $T$ under the given hypotheses. Thus one can simply retrace the proof of (i), replacing ``faithful flatness'' by ``free over an explicitly stated basis" at each occurrence of the former term. Details are left to the reader.
\end{proof}

\begin{remark} In fact there is less of a gap between parts (i) and (iii) of the above result than at first appears. For, notice that the proof of (i) starts with the fact that $R = k[x,y, a^{\pm n}, b^{\pm m}]$ is a quantum homogeneous space of $T$. In particular $T$ is a faithfully flat left $R$-module, and hence, by \cite[Theorem 2.1]{MaW}, $T$ is a projective generator for $R$. Therefore, by the Quillen-Suslin theorem on projective modules over polynomial algebras, as strengthened by Swan \cite[Corollary 1.4]{Sw} to allow Laurent polynomial generators, $T$ is a free left $R$-module. Finally, freeness is preserved by factoring by $(a^n - b^m)R + (g-f)R$. So the only extra feature in (iii)(b) as compared with (i)  is the explicit description of a free basis.
\end{remark}

Fundamental properties of the Hopf algebras $A(g,f)$ can be read off from our knowledge gained about the algebras $A(x,a,g)$ in $\S\S$\ref{1}-\ref{anyg}, provided hypothesis {\bf(H)} is in play. The following portmanteau result summarises the basic facts.

\begin{theorem}\label{theworks} Retain the notation of $\S$\ref{combine}, so $g(x)$ and $f(y)$ are polynomials of degree $n$ and $m$ respectively. Assume hypothesis {\bf(H)} for parts (ii) - (v). Then $A(g,f)$ satisfies the following properties.
\begin{enumerate}
\item[(i)] $A(g,f)$ is a pointed Hopf algebra, with its grouplikes being the finitely generated abelian group $\langle a^{\pm 1}, b^{\pm 1} : [a,b] = 1, a^n = b^m \rangle$.
\item[(ii)] The following statements are equivalent:
\begin{enumerate}
\item[(a)] $\mathrm{max}\{n,m\} \leq 3;$
\item[(b)] $A(g,f)$ is a finite module over its centre.
\item[(c)] $A(g,f)$ satisfies a polynomial identity;
\item[(d)] $\mathrm{GKdim}(A(g,f)) < \infty$;
\item[(e)] $A(g,f)$ does not contain a noncommutative free subalgebra;
\end{enumerate} 

\item[(iii)] If the equivalent conditions in (ii) hold, then $A(g,f)$ is noetherian.
\item[(iv)] Suppose that $\mathrm{max}\{n,m\} \leq 3$. Then $A(g,f)$ is AS-Gorenstein and GK-Cohen-Macaulay, with
$$ \mathrm{inj.dim}(A(g,f)) = \mathrm{GKdim}(A(g,f)) = n + m -2. $$

\item[(v)]  If $\mathcal{C}$ has a singular point at the origin then $\mathrm{gl.dim}(A(g,f)) = \infty$
\end{enumerate}
\end{theorem}

\begin{proof}(i) It has already been shown in Theorem \ref{summit} that $A(g,f)$ is a Hopf algebra. It is generated by the grouplike elements $a,b$ and the skew primitives $x$ and $y$, and hence is pointed by \cite[Corollary 5.1.14(a)]{radford}. Moreover by \cite[Corollary 5.1.14(b)]{radford}, $\{ a, b \}$ generate its grouplikes $G(A(g,f))$. Hence $G(A(g,f))$ is as stated.

\noindent (ii) $(a)\Longrightarrow (b):$ Suppose $\mathrm{max}\{n,m\} \leq 3$. Then both $A(x,a,g)$ and $A(y,b,f)$ are finite modules over their centres, by Propositions \ref{degtwo}(ii) and \ref{cube}(iv). The same is thus clearly true of the tensor product $T$ of these algebras, and so of its factor algebra $A(g,f)$.

\noindent $(b)\Longrightarrow (c):$ \cite[Corollary 13.1.13(i)]{robson}.

\noindent $(c)\Longrightarrow (d):$ \cite[Corollary 10.7]{krause}.

\noindent $(d)\Longrightarrow (e):$ Suppose that $\mathrm{GKdim}(A(g,f))$ is finite. Then $A(g,f)$ cannot contain a noncomutative free subalgebra, since such an algebra has infinite GK-dimension, \cite[Example 1.2]{krause}.

\noindent $(e)\Longrightarrow (a):$ Suppose that  $A(g,f)$ does not contain a noncommutative free subalgebra. By Theorem \ref{coideal}(ii), the same is true for the subalgebras $A(x,a,g)$ and $A(y,b,f)$. By Corollary \ref{basis} and Proposition \ref{xnprop}(iv), $\mathrm{max}\{n,m\} \leq 3$.

\noindent (iii) Suppose that (ii)(b) holds. Since $A(g,f)$ is an affine $k$-algebra, its centre $Z$ is also affine, by (ii)(b) and the Artin-Tate lemma \cite[Lemma 13.9.10(ii)]{robson}. Thus $Z$ is noetherian by the Hilbert basis theorem, and so $A(g,f)$ is a noetherian algebra since it is a finite $Z$-module.

\noindent (iv) Since $\mathrm{max}\{n,m\} \leq 3$, the Hopf algebra $A(g,f)$ is a finite module over its affine centre by (ii). Indeed, the same is also true for $T = A(x,a,g) \otimes_k A(y,b,f)$. Therefore $T$ and $A(g,f)$ are both AS-Gorenstein and GK-Cohen-Macaulay, by \cite[Theorem 0.2]{WuZh}. Now 
\begin{equation}\label{hype} \mathrm{GKdim}(T) = \mathrm{GKdim}(A(x,a,g)) + \mathrm{GKdim}(A(y,b,f)) = n +m, 
\end{equation}
by \cite[Corollary 10.17]{krause} for the first equality and the second by Propositions \ref{degtwo}(ii) and \ref{three}(v),(vii). 

Next, the central elements $f-g$ and $a^n - b^m$ of the GK-Cohen-Macaulay algebra $T$ are easily seen to form a regular central sequence in $T$, using the PBW theorem, Theorem \ref{coideal}(ii)(a). So, by (\ref{hype}) and \cite[Proposition 2.11 and Theorem 4.8$(i)\Longleftrightarrow (iii)$]{BrMa},
\begin{equation}\mathrm{GKdim}(A(g,f)) = \mathrm{GKdim}(T) - 2 = n + m - 2,
\end{equation}
as required. Finally, since $A(g,f)$ is an AS-Gorenstein GK-Cohen-Macaulay algebra, its injective and Gel'fand-Kirillov dimensions are equal, \cite[Theorem 0.2]{WuZh}.

\noindent (v) Suppose that $\mathcal{C}$ has a singularity at the origin. Thus, letting $\mathfrak{m} = \langle x, y \rangle \lhd \mathcal{O}(\mathcal{C}) \subset A(g,f)$,
\begin{equation}\label{sing}   \mathrm{pr.dim}_{\mathcal{O}(\mathcal{C})}(\mathcal{O}(\mathcal{C})/\mathfrak{m}) = \infty. 
\end{equation}
\noindent Suppose that $\mathrm{gl.dim}(A(g,f)) < \infty$. Then, in particular, $\mathrm{pr.dim}_{A(g,f)}(k_{\mathrm{tr}}) < \infty$, where $k_{\mathrm{tr}}$ denotes the trivial module. By Theorem \ref{coideal}(i) and ({\bf{H}}), $A(g,f)$ is a flat $\mathcal{O}(\mathcal{C})$-module, so the restriction to $\mathcal{O}(\mathcal{C})$ of a projective $A(g,f)$-resolution of $k_{\mathrm{tr}}$ yields a finite flat $\mathcal{O}(\mathcal{C})$-resolution of $k_{\mathrm{tr}}$.

But $\epsilon (x) = \epsilon (y) = 0$, so that $\mathfrak{m} k_{\mathrm{tr}} = 0$. Hence, bearing in mind that flat dimension and weak dimension are the same for modules over a noetherian ring,  
$$\mathrm{pr.dim}_{\mathcal{O}(\mathcal{C})}(\mathcal{O}(\mathcal{C})/\mathfrak{m}) < \infty.$$ 
This contradicts (\ref{sing}), and so $\mathrm{gl.dim}(A(g,f)) = \infty$, as required. 
\end{proof}

Presumably Theorem \ref{theworks}(v) remains true for all singular plane decomposable curves $\mathcal{C}$, but we have been unable so far to prove this.


\section{Examples}\label{examples}

In this section we gather together and discuss some special cases of the construction described in $\S$\ref{tensor}. 

We begin by noting that, given a decomposable plane curve $\mathcal{C}$ with equation $g(x) = f(y)$, we may always use a linear change of coordinates to assume that the curve passes though the origin, so that $f$ and $g$ have constant term 0. We can further assume that both $f$ and $g$ are monic. For, recall the scaling maps $\{\theta_{\lambda}: x \mapsto \lambda x : \lambda \in k \setminus \{0\} \}$ of $k\langle x, a^{\pm 1} \rangle$ introduced in $\S$\ref{leading}, where we wrote $g^{\lambda}$ for $\theta (g)$. The following lemma extends this scaling procedure to the Hopf algebras $A(g,f)$. Its straightforward proof is left to the reader.

\begin{lemma}\label{Hopfscale} Let $\mathcal{C}$ be the decomposable plane curve with equation $f(y) = g(x)$, where $f$ and $g$ are polynomials with constant term 0, respectively of degree $m,n$ with $m,n \geq 2$ and satisfying hypothesis {\bf (H)} of $\S$\ref{combine}. Let $\lambda, \mu \in k \setminus \{0\}$. 
\begin{enumerate}
\item[(i)] $\theta_{\lambda} \otimes \mathrm{Id}_{A(y,b,f)}$ and $\mathrm{Id}_{A(x,a,g)} \otimes \theta_{\mu}$ are commuting automorphisms of $k \langle x, a^{\pm 1} \rangle \otimes k \langle y, b^{\pm 1} \rangle$, whose composition induces an isomorphism of Hopf algebras 
$$ \theta_{\lambda, \mu} : A(x,a,g) \otimes A(y,b,f) \longrightarrow A(x,a,g^{\lambda}) \otimes A(y,b,f^{\mu}). $$
\item[(ii)] The map $\theta_{\lambda,\mu}$ induces an isomorphism of Hopf algebras 
$$ \overline{\theta_{\lambda,\mu}} : A(g,f) \longrightarrow A(g^{\lambda}, f^{\mu}),$$
under which the quantum homogeneous space $k[x,y]/\langle g - f \rangle$ of $A(g,f)$ is mapped to the quantum homogeneous space $k[x,y]/\langle g^{\lambda} - f^{\mu} \rangle$ of $A(g^{\lambda}, f^{\mu})$.
\end{enumerate}

\end{lemma}


\subsection{$A(g,f)$ for degree 2 polynomials}\label{degs2}

Let $\mathcal{C}$ be an arbitrary decomposable plane curve of degree 2 - that is, $\mathcal{C}$ has defining equation $g(x) = f(y)$ with $\mathrm{deg}f = \mathrm{deg}g = 2$. After an application of Lemma \ref{Hopfscale} we can assume without loss of generality that $g$ and $f$ are monic. Thus the equation of $\mathcal{C}$ has the form
\begin{equation}\label{deg2ex}  rx + x^2 \; = \; sy + y^2, 
\end{equation}
where $(r,s) \in k^2$. An easy exercise determines the possibilities for $\mathcal{C}$. Namely, the Jacobian criterion shows that $\mathcal{C}$ is smooth if and only if $r \neq \pm s$; and if $r = \pm s$, $\mathcal{C}$ has a unique singular point. Then, by the linear change of variables 
$$ u = x - y + \frac{1}{2}(r-s), \; v = x + y + \frac{1}{2}(r - s), $$
one sees that the coordinate ring $\mathcal{O}(\mathcal{C})$ is isomorphic to $k[u,v]/\langle uv + C \rangle$, where $C$ is a constant which is non-zero in the smooth case and 0 in the singular case, the latter being the coordinate crossing.

Consider now the corresponding Hopf algebra $A(g,f)$.  As in Proposition \ref{degtwo}(i), proceed by changing the variables $x$ and $y$ in $A(x,a,g)$ and $A(y,b,f)$ to
\begin{equation}\label{change} x' := x + \frac{r}{2}(1-a), \qquad y' := y + \frac{s}{2}(1 - b).
\end{equation}
These elements are respectively $(1,a)-$ and $(1,b)-$primitive, and by Proposition \ref{degtwo}(i) we have 
\begin{eqnarray*} T \; &=& A(x,a,g) \otimes A(x,b,f)\\
&=& k \langle a^{\pm 1}, b^{\pm 1}, x', y' : x'a + ax' = 0, \, y'b + by' = 0, \\
&\,& \qquad \qquad [a,b] = [x',y'] = [a,y'] = [x',b] = 0 \rangle.
\end{eqnarray*}
One calculates that, in $T$,
$$ x'^2 -y'^2 \; = \; f-g + \frac{1}{4}(r^2 - s^2) - \frac{1}{4}(r^2a^2 - s^2b^2). $$
Therefore, in $A(g,f)$, that is $modulo \langle f-g, a^2 - b^2 \rangle$,
\begin{equation} \label{crossing} x'^2 - y'^2 \; \equiv \; \frac{1}{4}(r^2 - s^2)(1 - a^2).
\end{equation}
The outcome is summarised in the next result.

\begin{proposition} \label{deg2story} Let $\mathcal{C}$ be the decomposable degree 2 plane curve embedded in the plane by the equation (\ref{deg2ex}). Then $\mathcal{O}(\mathcal{C})$ is a quantum homogeneous space of the Hopf algebra $A(g,f)$, where $A(g,f)$ has the following properties.
\begin{enumerate}
\item[(i)] Defining $x'$ and $y'$ as in (\ref{change}), $A(g,f)$ has presentation
\begin{eqnarray*}
A(g,f) \; &=& \;  k \langle a^{\pm 1}, b^{\pm 1}, x', y' : x'a + ax' = 0, \, y'b + by' = 0, \\
&\,& \qquad \qquad [a,b] = [x',y'] = [a,y'] = [x',b] = 0, \\
&\,& \qquad \qquad a^2 = b^2, \; \; x'^2 - y'^2 = \frac{1}{4}(r^2 - s^2)(1 - a^2) \rangle.
\end{eqnarray*}
Here, $a$ and $b$ are grouplike, $x'$ is $(1,a)$-primitive and $y'$ is $(1,b)$-primitive.
\item [(ii)] $A(g,f)$ is an affine noetherian pointed Hopf algebra, is a finite module over its centre, and is AS-Gorenstein and GK-Cohen-Macaulay, with injective and Gel'fand-Kirillov dimensions equal to 2. 
\item[(iii)]  $\mathrm{gl.dim}(A(g,f)) < \infty \Longleftrightarrow r \neq \pm s,$ that is, if and only if $\mathcal{C}$ is smooth. In this case, $\mathrm{gl.dim}(A(g,f)) = 2.$
\item[(iv)] Up to an isomorphism of Hopf algebras there are only two possible algebras $A(g,f)$ - the smooth case and the singular case.
\item[(v)] $A(g,f)$ is prime but is not a domain.
\end{enumerate}
\end{proposition}

\begin{proof}(i) This is sketched in the discussion before the proposition. 

\noindent (ii) These are all special cases of Theorem \ref{theworks}(i),(ii), (iv), (v).

\noindent (iii) Suppose first that $r^2 = s^2$. Then the relation (\ref{crossing}) of $A(g,f)$ becomes $(x' - y')(x' + y') = 0$. Note that, by Masuoka's theorem \cite[Theorem 1.3]{Ma}, $A(g,f)$ is faithfully flat over its right coideal subalgebra $k\langle x', y' \rangle$. As in the proof of Theorem \ref{theworks}(v), if the trivial $A(g,f)$-module had a finite projective resolution then the same would be true for the $k\langle x', y' \rangle$-module $k\langle x', y' \rangle/\langle x',y' \rangle$ since the latter is the restriction of the former. But this is manifestly false, since $\langle x',y' \rangle$ defines the singular point of this curve. So $\mathrm{gl.dim}(A(g,f)) = \infty$. 

Conversely, suppose that $r^2 \neq s^2$. Define $R$ to be the commutative subalgebra $k\langle x',y', a^{\pm 2} \rangle$ of $A(g,f)$. Observe that, setting 
$$X := x'-y', \, Y := x'+y', \, Z := \frac{1}{4}(r^2 - s^s)(1 - a^2),$$
$R$ is isomorphic to the localisation of 
$$ k[X,Y,Z]/\langle XY - Z \rangle $$
at the powers of $Z - \frac{1}{4}(r^2 - s^2)$. In particular, (for example by the Jacobian criterion), $R$ is smooth, $\mathrm{gl.dim}(R) = 2$. Now, by the defining relations of $A(g,f)$, the augmentation ideal $R^+ := \langle X,Y,Z \rangle$ of $R$ is normal in $A(g,f)$, so that the right ideal $R^+ A(g,f)$ of $A(g,f)$ is a two sided ideal. Then one easily checks that
\begin{equation}\label{quot} A(g,f)/R^+ A(g,f) \; \cong \; k(\mathbb{Z}_2 \times \mathbb{Z}_2),
\end{equation}
the group algebra of the Klein 4-group $K$, with generators the images of $a$ and $ab^{-1}$. In fact, $A(g,f)$ is a crossed product $R \ast K$, though we don't need this. Rather, it is enough to note that $A(g,f)$ is faithfully flat as a right and left $R$-module, being a quantum homogeneous space using as usual \cite[Theorem 1.3]{Ma}. 

The simple $R$-module $R/R^+$ has a finite $R$-projective resolution $\mathcal{P}$, by smoothness of $R$. Faithful flatness ensures exactness of $- \otimes_{R} A(g,f)$, yielding a finite projective resolution $\mathcal{P} \otimes_R A(g,f)$ of the $A(g,f)$-module $A(g,f)/ R^+ A(g,f)$. However, $A(g,f)/ R^+ A(g,f)$ is semisimple Artinian, by (\ref{quot}), and hence contains the trivial module $k$ of the Hopf algebra $A(g,f)$ as a direct summand. Therefore $\mathrm{pr.dim}_{A(g,f)}( k) < \infty$. By \cite[$\S$2.4]{LoLo}, $\mathrm{gl.dim}(A(g,f)) < \infty$, as required. That the global dimension is equal to 2 follows either from the fact that the injective dimension is 2, as shown in (ii), or from the fact that $\mathrm{pr.dim}_R(R/R^+) = 2$ since $R$ is the coordinate ring of a smooth surface.

\noindent (iv) It is clear from the presentation in (i) how to define a Hopf algebra isomorphism between any two members of the smooth family simply by scaling the generators $x' $ and $y'$. On the other hand, the coefficients disappear from the relations when $r = \pm s$, so the result is clear in this case also.

\noindent (v) It is clear from the defining relations and the PBW theorem that $A(g,f)$ is not a domain, since $a \neq \pm b$, but $(a-b)(a + b) = 0$. The easiest way to see that $A(g,f)$ is prime is to use the smash product description of $A(g,f)$ found in the proof of (iii). Namely, it was shown there that $A(g,f)$ contains a commutative subalgebra $R := k[X,Y,Z]/\langle XY - tZ \rangle$, where $t \in k$ is 0 in the singular case and non-zero in the smooth case. Then $A(g,f)$ is a smash product  $R\# \Gamma$ where $\Gamma = \langle \overline{a}, \overline{b} \rangle$ is a Klein 4-group such that $a: X \leftrightarrow -Y$ and $b: X \leftrightarrow Y$. When $t \neq 0$, $R$ is a domain and $\Gamma$ acts faithfully on its quotient field $Q$. Hence $Q\#\Gamma$ is a simple ring by \cite{Az}, see \cite[Exercise 6, p. 48]{Pas}. When $t = 0$, $R$ is $\Gamma$-prime and primeness of $R\#\Gamma$ follows by passing to the quotient ring $Q(R)\# \Gamma$, where $Q(R) = k(X, Z) \oplus k(Y,Z)$ is $\Gamma$-simple and $\Gamma$ acts faithfully on $Q(R)$. Thus $Q(R)\#\Gamma$ is a simple ring by \cite[Theorem 1.2(d)]{O}. So in all cases $R\# \Gamma$ has a simple quotient ring and hence is a prime ring.
\end{proof}

\subsection{The cusps $y^m = x^n$}\label{cusps} Let $n$ and $m$ be coprime integers, with $m > n \geq 2$. The cusp $y^m = x^n$ was shown to be a quantum homogeneous space in a pointed affine noetherian Hopf $k$-algebra domain,  in \cite[Construction 1.2]{MR2732991}. In the notation introduced by Goodearl and Zhang in \cite{MR2732991}, the Hopf algebra constructed is labelled $B(1,1,n,m,q)$, where $q$ is a primitive $nm$th root of unity in $k$. The algebra $B(1,1,n,m,q)$ is constructed as the skew group algebra of the infinite cyclic group whose coefficient ring is the coordinate ring of the cusp, with the twisting automorphism acting on the generators of the coordinate ring by multiplication by appropriate powers of $q$. In particular, this means that $B(1,1,n,m,q)$ has GK-dimension 2 and is a finite module over its centre. It is also straightforward to see from its construction that $B(1,1,n,m,q)$ is a factor Hopf algebra of a suitable localised quantum 4-space $Q := k_{{\bf q}}[x^{\pm 1}_1,x^{\pm 1}_2,x_3,x_4]$, where $[x_1,x_2] = [x_3,x_4] = 0$, the other pairs of generators $q$-commute, and the factoring relations are 
\begin{equation}\label{factor} x_1^n - x_2^m, \textit{  and  }  x_3^n - x_4^m.
\end{equation}
Now it is also not hard to deduce from Proposition \ref{xnprop}(i) that the \emph{same} localised quantum 4-space $Q$ is a factor Hopf algebra of $T := A(x,a,x^n) \otimes A(y,b,y^m)$. Moreover, the relations used to define $A(x^n,y^m)$ as a factor of $T$ have images in $Q$ which are exactly the elements listed in  (\ref{factor}). The upshot is therefore the following observation, the details of the argument being left to the interested reader.

\begin{lemma} \label{cuspstory} With the above notation, $B(1,1,n,m,q)$ is a factor Hopf algebra of $A(x^n,y^m)$.
\end{lemma}

Note that even in the ``smallest'' case, $(n,m) = (2,3)$, $B(1,1,2,3,q)$ is a \emph{proper} factor of $A(x^2, y^3)$ - here, both algebras are finite over their centres, but $A(x^2, y^3)$ has GK-dimension 3 and $B(1,1,2,3)$ is a domain of GK-dimension 2.   

\subsection{The nodal cubic}\label{nodal} Consider the nodal cubic, $y^2 = x^2 + x^3$. Since $(n,m) = (2,3)$, the results of $\S\S$\ref{combine},\ref{combprop} apply. The presentation of $A(x^2 + x^3, y^2)$ is
\begin{eqnarray*} A(x^2 + x^3,y^2) \; &=& \; k\langle x,y,a^{\pm 1}, b^{\pm 1} : y^2 = x^2 + x^3, \; a^3 = b^2, \\
&\,& \qquad \qquad [x,y] = [x,b] = [a,y] = [a,b] = 0, \\
&\,& \qquad \qquad ax + xa + ax^2 + xax + x^2a = 0, \\
&\,& \qquad \qquad a^2x + axa + xa^2 = a^3 + a^2, \\
&\,& \qquad \qquad yb + by = 0 \rangle.
\end{eqnarray*}
This is precisely the Hopf algebra presented in \cite{KT}, for the parameters $p = q = 0$. As noted in \cite[$\S$2.3]{KM}, the other values of $p$ and $q$ yield isomorphic Hopf algebras. From the results of $\S$\ref{combprop} we see that $A(x^2 + x^3, y^2)$ is an affine noetherian Hopf algebra of GK-dimension 3, which is a finite module over its centre. It is AS-Gorenstein and GK-Cohen Macaulay of injective dimension 3, but has infinite global dimension. The Hopf algebra constructed in \cite{KM} is a quotient Hopf algebra of $A(x^2 + x^3, y^2)$ of GK-dimension 1 which still admits the nodal cubic as a quantum homogeneous space, demonstrating that $A(x^2 + x^3, y^2)$ is \emph{not} minimal with this property.

\subsection{The lemniscate}\label{lem} The Lemniscate of Gerono is usually presented by the equation $ y^2 = x^2 - x^4. $ Applying Lemma \ref{Hopfscale} with $\mu = 1$ and $\lambda$ a primitive $8^{th}$ root of 1 in $k$, we can work with the presentation 
\begin{equation}\label{lem2}y^2 = x^4 + \lambda^2 x^2 
\end{equation}
of the lemniscate. The outcome is as follows, recalling that in $\S$\ref{combprop} we defined $\langle \mathcal{L}_4(a,x) \rangle$ to denote the free subsemigroup of the free semigroup $\langle a,x \rangle$ generated by $\{(ax), (ax^2), (a^2x) \}$.

\begin{lemma} Let the lemniscate $\mathcal{L}$ be presented by the equation (\ref{lem2}), so $f(y) = y^2$ and $g(x) = x^4 + \lambda^2 x^2$.
\begin{enumerate}
\item[(i)] The coordinate ring $\mathcal{O}(\mathcal{L})$ is a quantum homogeneous space in $A(x^4 + \lambda^2 x^2, y^2),$ a Hopf algebra with generators $x, a^{\pm 1}, y , b^{\pm 1}$ and relations
\begin{gather*}
[a,b] = [x,y] = [a,y] = [x,b] = 0,\\
	by + yb = 0, \qquad     a x^3+ xax^{2} + x^2ax + x^3a + \lambda^2 (xa + ax) = 0,\\
 \lambda^2 a^2 + a^2x^2 +x^2a^2 + xaxa + ax^2a + xa^2x + axax - \lambda^2  a^4 = 0,\\
a^3x  + a^2xa+ axa^2  + xa^3 = 0.
\end{gather*}
\item[(ii)] The algebra $A(x^4 + \lambda^2 x^2, y^2)$ has PBW basis
$$ \{ x^r y^{\epsilon_1}\langle \mathcal{L}_4(a,x) \rangle a^s b^{\epsilon_2} : r \in \mathbb{Z}_{\geq 0}, s \in \mathbb{Z}, \epsilon_j \in \{0,1\}, j = 1,2 \}. $$
\item[(iii)] $A(x^4 + \lambda^2 x^2, y^2)$ is not a domain and has infinite Gel'fand-Kirillov and global dimensions.
\end{enumerate}
\end{lemma} 
\begin{proof} (i) Theorem \ref{summit}.

\noindent (ii) The PBW basis is given by Theorem \ref{coideal}(iii)(a), noting that $\mathcal{O}(\mathcal{L}) = k\langle x,y \rangle$ is a free $k[x]$-module on the basis $\{1, y \}$.

\noindent (iii) The group-like element $a^{-2}b$  of $A(x^4 + \lambda^2 x^2, y^2)$ has order 2, so $(a^{-2}b - 1)(a^{-2}b + 1) = 0$ and $A(x^4 + \lambda^2 x^2, y^2)$ is not a domain. By (ii), $A(x^4 + \lambda^2 x^2, y^2)$ contains  the noncommutative free algebra $k\langle \mathcal{L}_4(a,x) \rangle$, and hence has infinite Gel'fand-Kirillov dimension by \cite[Example 1.2]{krause}. Since the lemniscate has a sungularity at the origin, $\mathrm{gl.dim}(A(x^4 + \lambda^2 x^2, y^2)) = \infty$ by Theorem \ref{theworks}(v).
\end{proof}

\section*{Acknowledgements} Some of this research will form part of the PhD thesis of the second author at the University of Glasgow. Her PhD is funded by a Schlumberger Foundation Faculty for the Future Fellowship, for which we are very grateful. The research of the first author was supported in part by Leverhulme Emeritus Fellowship EM-2017-081. 

We thank Uli Kraehmer (Dresden) for his unfailing interest in this work and for many helpful comments. We also very much appreciate input from Stefan Kolb (Newcastle), Manuel Martins (Dresden), Chelsea Walton (Urbana) and James Zhang (Seattle).

\bibliographystyle{plain}

\newpage

\section*{Appendix}

We present here details of the proofs of the cases $t=2$ and $t=3$ of Proposition $2.8$.

\noindent (ii) Let $t=2$ and $n \geq 4$. We use the following identities throughout the proofs:
\begin{eqnarray}
P(r,s) &=&  P(r-1,s-1)xa +  P(r-1,s-1)ax + P(r-2,s)a^{2} + P(r,s-2)x^{2}, \label{a}\\
P(r,s) &=&  xaP(r-1,s-1) +  axP(r-1,s-1) + a^{2}P(r-2,s) + x^{2}P(r,s-2),\label{b}\\
P(r,s) &=& xP(r-1,s-1)a + aP(r-1,s-1)x + aP(r-2,s)a + xP(r,s-2)x \label{c}.
\end{eqnarray}
We resolve the amabiguities in three cases.

\noindent (a) The ambiguity arises from the two routes to resolve the word
$$
a^{2}\omega_{1} = a^{2}(ax^{n-1}) = (a^{3}x^{n-3})x^{2}= \omega_{3}x^{2}
$$
in the free algebra $\field \langle a,x \rangle$ using the relations $\sigma_{1}$ and $\sigma_{3}$. Considering first $\sigma_{1}$, use  (\ref{a}) to write it as
$$
\omega_{1} = ax^{n-1} \rightarrow - (\sum_{i=2}^{n-1}r_{i} \underline{( P(0, i-2)ax} + \underline{P(0,i-2)xa }) + \sum_{i=3}^{n-1}r_{i}\underline{P(1,i-3)x^{2}} +
$$
$$
 \underline{Q(1,n-3)x^{2}}  + \underline{ P(0, n-2)ax} + \underline{P(0,n-2)xa } +  r_{1}\underline{a}) + r_{1}\underline{a^{n}}
$$
$$
= - ( \sum_{i=2}^{n}r_{i} \underline{( P(0, i-2)ax} + \underline{P(0,i-2)xa }) + \sum_{i=3}^{n-1}r_{i}\underline{P(1,i-3)x^{2}} + \underline{Q(1,n-3)x^{2}} + r_{1}\underline{a}) + r_{1}\underline{a^{n}}.
$$
Premultiply this by $a^{2}$, and use Lemmas 2.6 and 2.7 to separate reducible and irreducible words, yielding
\begin{equation}
\tag{$\alpha$}
a^{2}\omega_{1} \rightarrow - ((\sum_{i=2}^{n}r_{i} ( a^{2}P(0, i-2)ax + a^{2}P(0,i-2)xa)  + \sum_{i=3}^{n-1}r_{i} a^{2}P(1,i-3)x^{2} + \underline{a^{2}Q(1,n-3)ax} + r_{1}\underline{a^{3}}) + r_{1}\underline{a^{n+2}}.
\end{equation}
The following words in ($\alpha$) of length $n+2$ are reducible:
$$
-a^{2}x^{n-2}ax \qquad \text{and} \qquad -a^{2}x^{n-2}xa.
$$
Using (\ref{b}), we write $\sigma_{2}$ as
$$
-a^{2}x^{n-2}  \rightarrow  \sum_{i=3}^{n-1}r_{i} (\underline{xaP(1,i-3)}  + \underline{axP(1,i-3)}) + \sum_{i=4}^{n-1}r_{i}\underline{x^{2}P(2,i-4)}  + \sum_{i=2}^{n-1}r_{i}\underline{a^{2}P(0,i-2)} 
$$
$$
+ \underline{xaP(1,n-3)}  + \underline{axP(1,n-3)} + \underline{x^{2}P(2,n-4)} - r_{2}\underline{a^{n}}
$$
$$
=  \sum_{i=3}^{n}r_{i} (\underline{xaP(1,i-3)}  + \underline{axP(1,i-3)}) + \sum_{i=4}^{n}r_{i}\underline{x^{2}P(2,i-4)}  + \sum_{i=2}^{n-1}r_{i}\underline{a^{2}P(0,i-2)} - r_{2}\underline{a^{n}}.
$$
Post multiplying this by $ax$ and using Lemmas 2.6 and 2.7 to separate reducible and irreducible words, yields
\begin{equation}
\tag{$\beta$}
-a^{2}x^{n-2}ax  \rightarrow \sum_{i=3}^{n}r_{i} (\underline{xaP(1,i-3)ax}  + \underline{axP(1,i-3)ax}) + \sum_{i=4}^{n}r_{i}\underline{x^{2}P(2,i-4)ax}  + 
\end{equation}
$$
\sum_{i=2}^{n-1}r_{i}\underline{a^{2}P(0,i-2)ax} - r_{2}\underline{axa^{n}}
$$
 Similarly, post multiplying the relation for $-a^{2}x^{n-2}$ above with $xa$  and using Lemmas 2.6 and 2.7 to separate reducible and irreducible words, yields
\begin{equation}
\tag{$\tau$}
-a^{2}x^{n-2}xa \rightarrow \sum_{i=3}^{n}r_{i} (xaP(1,i-3)xa  + \underline{axP(1,i-3)xa}) + \sum_{i=4}^{n}r_{i}\underline{x^{2}P(2,i-4)xa}  + 
\end{equation}
$$
\sum_{i=2}^{n-1}r_{i}a^{2}P(0,i-2)xa - r_{2}\underline{xa^{n+1}}.
$$

To reduce $xa^{2}x^{n-2}a$, note first that using (\ref{c}), the left hand side of $\sigma_{2}$ is written as
$$
a^{2}x^{n-2} \rightarrow - ( \sum_{i=2}^{n-1} \underline{aP(0,i-2)a} + \sum_{i=3}^{n-1}r_{i}( \underline{xP(1,i-3)a} + \underline{aP(1,i-3)x})  + \sum_{i=4}^{n-1}\underline{xP(2,i-4)x} 
$$
$$
 + \underline{xP(1,n-3)a} + \underline{aQ(1,n-3)x} + \underline{xP(2,n-4)x} + \underline{aP(0,n-2)a}  ) + r_{2}\underline{a^{n}}.
$$
$$
= - ( \sum_{i=2}^{n} \underline{aP(0,i-2)a} + \sum_{i=3}^{n-1}r_{i}( \underline{xP(1,i-3)a} + \underline{aP(1,i-3)x})  + \sum_{i=4}^{n}\underline{xP(2,i-4)x} 
$$
$$
 + \underline{xP(1,n-3)a} + \underline{aQ(1,n-3)x}  ) + r_{2}\underline{a^{n}}.
$$
Pre and post multiplying this by $x$ and $a$ respectively and using Lemmas 2.6 and 2.7 to separate reducible and irreducible words, yields
\begin{equation}
\tag{$\gamma$}
xa^{2}x^{n-2}a \rightarrow - ( \sum_{i=2}^{n} \underline{xaP(0,i-2)a^{2}} + \sum_{i=3}^{n-1}r_{i}( \underline{x^{2}P(1,i-3)a^{2}} + \underline{xaP(1,i-3)xa})  + \sum_{i=4}^{n}\underline{x^{2}P(2,i-4)xa} 
\end{equation}
$$
 + \underline{x^{2}P(1,n-3)a^{2}} + \underline{xaQ(1,n-3)xa}  ) + r_{2}\underline{xa^{n+1}}.
$$
Thus, the reduction process ends here. Substituting ($\beta$), ($\tau$) and ($\gamma$) into ($\alpha$) and simplifying yields
\begin{equation}
\tag{$\chi$}
a^{2}\omega_{1} \rightarrow ( \sum_{i=3}^{n}r_{i} (\underline{xaP(1,i-3)ax} + \underline{axP(1,i-3)ax} + \underline{axP(1,n-3)xa} ) +  \sum_{i=4}^{n}r_{i} \underline{x^{2}P(2,i-4)ax} +
\end{equation}
$$
  r_{1}\underline{a^{n+2}} )  -  (  \sum_{i=2}^{n}r_{i} \underline{xaP(0,i-2)a^{2}} + \sum_{i=3}^{n-1}r_{i} a^{2}P(i,i-3)x^{2} +  \sum_{i=3}^{n}r_{i}\underline{x^{2}P(1,i-3)a^{2}} + 
$$
$$
\underline{a^{2}Q(1,n-3)x^{2}}  + r_{1}\underline{a^{3}}  +   r_{2}\underline{axa^{n}} ).
$$

Turning now to $\omega_{3}x^{2}$, using (\ref{b}) to write the right hand side of $\sigma_{3}$ yields
$$
\omega_{3} \rightarrow - ( \sum_{i=3}^{n-1}r_{i}\underline{a^{2}P(1,i-3)} +  \underline{a^{2}Q(1,n-3)} +  \sum_{i=4}^{n-1}r_{i} (\underline{xaP(2,i-4)} + \underline{axP(2,i-4)}) + 
$$
$$
\underline{xaP(2,n-4)} + \underline{axP(2,n-4)} + \sum_{i=5}^{n-1}r_{i}\underline{x^{2}P(3,i-5)} + \underline{x^{2}P(3,n-5)} ) + r_{3}\underline{a^{n}}
$$
$$
= - ( \sum_{i=3}^{n-1}r_{i}\underline{a^{2}P(1,i-3)} +  \underline{a^{2}Q(1,n-3)} +  \sum_{i=4}^{n}r_{i} (\underline{xaP(2,i-4)} + \underline{axP(2,i-4)}) + 
$$
$$
\sum_{i=5}^{n}r_{i}\underline{x^{2}P(3,i-5)} ) + r_{3}\underline{a^{n}}.
$$
When we post multiply this by $x^{2}$ and use Lemmas 2.6 and 2.7 to separate reducible and irreducible words, this yields
\begin{equation}
\tag{$\alpha'$}
\omega_{3}x^{2} \rightarrow - ( \sum_{i=3}^{n-1}r_{i}a^{2}P(1,i-3)x^{2} +  \underline{a^{2}Q(1,n-3)x^{2}} +  \sum_{i=4}^{n}r_{i} (xaP(2,i-4)x^{2} + axP(2,i-4)x^{2}) + 
\end{equation}
$$
\sum_{i=5}^{n}r_{i}x^{2}P(3,i-5)x^{2} ) + r_{3}\underline{x^{2}a^{n}}.
$$
We get the following reducible words of length $n+2$ from ($\alpha'$):
$$
xaa^{2}x^{n-2}, \quad axa^{2}x^{n-2}, \quad x^{2}a^{3}x^{n-3}
$$
from $ xaP(2,n-4)x^{2} $, $axP(2,n-4)x^{2}$ and $x^{2}P(3,n-5)x^{2}$ respectively. Using (\ref{a}), the right hand side of $\sigma_{2}$ is written as 
$$
a^{2}x^{n-2} \rightarrow - ( \sum_{i=2}^{n-1}r_{i}\underline{P(0,i-2)a^{2}} + \sum_{i=3}^{n-1}r_{i} ( \underline{P(1,i-3)ax} + \underline{P(1,i-3)xa} )  + \sum_{i=4}^{n-1}r_{i}\underline{P(2,i-4)x^{2}}
$$
$$
 +  \underline{Q(2,n-4)x^{2}} + \underline{P(0,n-2)a^{2}} + \underline{P(1,n-3)ax} + \underline{P(1,n-3)xa} ) +  r_{2}\underline{a^{n}}
$$
$$
= - ( \sum_{i=2}^{n}r_{i}\underline{P(0,i-2)a^{2}} + \sum_{i=3}^{n}r_{i} ( \underline{P(1,i-3)ax} + \underline{P(1,i-3)xa} )  + \sum_{i=4}^{n-1}r_{i}\underline{P(2,i-4)x^{2}} 
$$
$$
+  \underline{Q(2,n-4)x^{2}}  ) +  r_{2}\underline{a^{n}}.
$$
Pre multiplying this by $xa$ and $ax$ and using Lemmas 2.6 and 2.7 to separate reducible and irreducible words, yields
\begin{equation}
\tag{$\beta'$}
-xaa^{2}x^{n-2} \rightarrow ( \sum_{i=2}^{n}r_{i}\underline{xaP(0,i-2)a^{2}} + \sum_{i=3}^{n}r_{i} ( \underline{xaP(1,i-3)ax} + xaP(1,i-3)xa )  + 
\end{equation}
$$
\sum_{i=4}^{n-1}r_{i}xaP(2,i-4)x^{2} +  \underline{xaQ(2,n-4)x^{2}}  ) - r_{2}\underline{xa^{n+1}}
$$
and
\begin{equation}
\tag{$\gamma'$}
-axa^{2}x^{n-2} \rightarrow ( \sum_{i=2}^{n}r_{i}axP(0,i-2)a^{2} + \sum_{i=3}^{n}r_{i} ( \underline{axP(1,i-3)ax} + \underline{axP(1,i-3)xa} )  + 
\end{equation}
$$
\sum_{i=4}^{n-1}r_{i}\underline{axP(2,i-4)x^{2}} +  \underline{axQ(2,n-4)x^{2}}  ) -  r_{2}\underline{axa^{n}}
$$
respevtively.
Recall from ($\gamma$) above that the reducible word $xa^{2}x^{n-2}a$ is given by,
\begin{equation}
\tag{$\eta'$}
xa^{2}x^{n-2}a \rightarrow - (  \sum_{i=3}^{n-1}r_{i}( \underline{x^{2}P(1,i-3)a^{2}} + \underline{xaP(1,i-3)xa})  + \sum_{i=4}^{n}\underline{x^{2}P(2,i-4)xa} +
\end{equation}
$$
\sum_{i=2}^{n} \underline{xaP(0,i-2)a^{2}} + \underline{x^{2}P(1,n-3)a^{2}} + \underline{xaQ(1,n-3)xa}  ) + r_{2}\underline{xa^{n+1}}.
$$

Turning now to the reducible word $ax^{n-1}a^{2}$ in ($\gamma'$), using (\ref{a}), the right hand side of $\sigma_{1}$ is written as
$$
ax^{n-1} \rightarrow - ( \sum_{i=2}^{n-1}r_{i}(  \underline{axP(0,i-2)} + \underline{xaP(0,i-2)})  + \underline{xaP(0,n-2)} + \sum_{i=3}^{n-1}r_{i}\underline{x^{2}P(1,i-3)} + 
$$
$$
\underline{x^{2}P(1,n-3)} +  r_{1}\underline{a}) + r_{1}\underline{a^{n}}
$$
so that
\begin{equation}
\tag{$\tau'$}
ax^{n-1}a^{2} \rightarrow - (\sum_{i=2}^{n-1}r_{i}(  \underline{axP(0,i-2)a^{2}} + \underline{xaP(0,i-2)a^{2}})  + \underline{xaP(0,n-2)a^{2}} + 
\end{equation}
$$
\sum_{i=3}^{n}r_{i}\underline{x^{2}P(1,i-3)a^{2}} + r_{1}\underline{a^{3}}) + r_{1}\underline{a^{n+2}}.
$$

Using (\ref{a}), the right hand side of $\sigma_{3}$ becomes,
$$
a^{3}x^{n-3} \rightarrow - ( \sum_{i=4}^{n-1}r_{i}\underline{P(1,i-3)a^{2}} + \sum_{i=4}^{n-1}r_{i} ( \underline{P(2,i-4)ax} + \underline{P(2,i-4)xa}  ) + \sum_{i=5}^{n-1}r_{i}\underline{P(3,i-5)x^{2}}
$$
$$
  + \underline{Q(3,n-5)x^{2}} + \underline{P(1,n-3)a^{2}}  + \underline{P(2,n-4)ax} + \underline{P(2,n-4)xa} +  r_{3}\underline{a^{3}} ) + r_{3}\underline{a^{n}}
$$
$$
= - ( \sum_{i=4}^{n}r_{i}\underline{P(1,i-3)a^{2}} + \sum_{i=4}^{n}r_{i} ( \underline{P(2,i-4)ax} + \underline{P(2,i-4)xa}  ) + \sum_{i=5}^{n-1}r_{i}\underline{P(3,i-5)x^{2}}
$$
$$
  + \underline{Q(3,n-5)x^{2}} +  r_{3}\underline{a^{3}} ) + r_{3}\underline{a^{n}}.
$$
Premultiplying this by $x^{2}$ and using Lemmas 2.6 and 2.7 to separate reducible and irreducible words, yields
\begin{equation}
\tag{$\chi'$}
-x^{2}a^{3}x^{n-3} \rightarrow ( \sum_{i=4}^{n}r_{i}\underline{x^{2}P(1,i-3)a^{2}} + \sum_{i=4}^{n}r_{i} ( \underline{x^{2}P(2,i-4)ax} + \underline{x^{2}P(2,i-4)xa}  ) + 
\end{equation}
$$
\sum_{i=5}^{n-1}r_{i}\underline{x^{2}P(3,i-5)x^{2}}  + \underline{x^{2}Q(3,n-5)x^{2}} + r_{3}\underline{x^{2}a^{3}} ) +  r_{3}\underline{x^{2}a^{n}}.
$$
Thus, the reduction process ends here. Substituting ($\beta'$), ($\gamma'$), ($\eta'$, ($\tau'$) and ($\chi'$) into ($\alpha'$)  and simplifying yields
\begin{equation}
\tag{$\alpha'$}
\omega_{3}x^{2} \rightarrow ( \sum_{i=3}^{n}r_{i} (\underline{xaP(1,i-3)ax} + \underline{axP(1,i-3)ax} + \underline{axP(1,n-3)xa} ) +  \sum_{i=4}^{n}r_{i} \underline{x^{2}P(2,i-4)ax} +
\end{equation}
$$
  r_{1}\underline{a^{n+2}} )  -  (  \sum_{i=2}^{n}r_{i} \underline{xaP(0,i-2)a^{2}} + \sum_{i=3}^{n-1}r_{i} a^{2}P(i,i-3)x^{2} +  \sum_{i=3}^{n}r_{i}\underline{x^{2}P(1,i-3)a^{2}} + 
$$
$$
 \underline{a^{2}Q(1,n-3)x^{2}}  + r_{1}\underline{a^{3}}  +   r_{2}\underline{axa^{n}} ).
$$
Comparing ($\alpha$) and ($\alpha'$), we conclude that the overlap ambiguity $\{\omega_{1}, \omega_{3} \}$ is resolvable.\\

\noindent (b) Let $j=n-3$, and consider the overlap ambiguity $\{\omega_{n-3}, \omega_{n-1} \}$. We may assume without loss of generallity that $n \geq 5$ since we have dealt with $1=4-3$ in Proposition 2.8(i). Using (\ref{a}), we write the right hand side of $\sigma_{n-3}$ as
$$
\omega_{n-3} \rightarrow - ( \sum_{i=n-3}^{n-1}r_{i}  \underline{P(n-5,i-(n-3))a^{2}}  +  \sum_{i=n-2}^{n-1}r_{i} ( \underline{P(n-4,i-(n-2))xa} + \underline{P(n-4,i-(n-2))ax})    
$$
$$
+  r_{n-1}\underline{P(n-3,0)x^{2}} + \underline{Q(n-3,1)x^{2}} + \underline{P(n-5,3)a^{2}} + \underline{P(n-4,2)xa} + \underline{P(n-4,2)ax} ) + r_{n-3}\underline{a^{n}}
$$
$$
= - ( \sum_{i=n-3}^{n}r_{i}  \underline{P(n-5,i-(n-3))a^{2}}  +  \sum_{i=n-2}^{n}r_{i} ( \underline{P(n-4,i-(n-2))xa} + \underline{P(n-4,i-(n-2))ax})    
$$
$$
+  r_{n-1}\underline{P(n-3,0)x^{2}} + \underline{Q(n-3,1)x^{2}}  ) + r_{n-3}\underline{a^{n}}
$$

Thus, premultiplying this by $a^{2}$ and using Lemmas 2.6 and 2.7 to separate reducible and irreducible words, yields
\begin{equation}
\tag{$I$}
a^{2}\omega_{n-3} \rightarrow - ( \sum_{i=n-3}^{n}r_{i}  a^{2}P(n-5,i-(n-3))a^{2}  +  \sum_{i=n-2}^{n}r_{i} ( a^{2}P(n-4,i-(n-2))xa + 
\end{equation}
$$
a^{2}P(n-4,i-(n-2))ax) +   r_{n-1}a^{2}P(n-3,0)x^{2} + \underline{a^{2}Q(n-3,1)x^{2}} ) + r_{n-3}\underline{a^{n+2}} ,
$$
The reducible words of length $n+2$ above are :
$$
a^{n-3}x^{3}a^{2}, \qquad a^{n-2}x^{2}ax, \qquad a^{n-2}x^{2}xa,
$$
from $a^{2}P(n-5,3)a^{2}$, $a^{2}P(n-4,2)ax$ and $a^{2}P(n-4,2)xa$ respectively. 

Using (\ref{b}), we write the right hand side of $\sigma_{n-3}$ as 
$$
a^{n-3}x^{3} \rightarrow - ( \sum_{i=n-3}^{n-1}r_{i}\underline{a^{2}P(n-5,i-(n-3))} + \sum_{i=n-2}^{n-1}r_{i} (\underline{axP(n-4,i-(n-2))} + \underline{xaP(n-4,i-(n-2))} )   
$$
$$
+ r_{n-1}\underline{x^{2}P(n-3,0)} +  \underline{a^{2}Q(n-5,3)} + \underline{axP(n-4,2)} + \underline{xaP(n-4,2)} + \underline{x^{2}P(n-3,1)} ) + r_{n-3}\underline{a^{n}}
$$
$$
= - ( \sum_{i=n-3}^{n-1}r_{i}\underline{a^{2}P(n-5,i-(n-3))}  + \sum_{i=n-2}^{n}r_{i} (\underline{axP(n-4,i-(n-2))} + \underline{xaP(n-4,i-(n-2))} )
$$
$$
+ \sum_{i=n-1}^{n}r_{i}  \underline{x^{2}P(n-3,i-(n-1))} +  \underline{a^{2}Q(n-5,3)} ) + r_{n-3}\underline{a^{n}}.
$$
Post multiplying this by $a^{2}$ and using Lemmas 2.6 and 2.7 to separate reducible and irreducible words, yields
\begin{equation}
\tag{$II$}
- a^{n-3}x^{3}a^{2} \rightarrow ( \sum_{i=n-3}^{n-1}r_{i}\underline{a^{2}P(n-5,i-(n-3))a^{2}} + \sum_{i=n-2}^{n}r_{i} (\underline{axP(n-4,i-(n-2))a^{2}} + 
\end{equation}
$$
\underline{xaP(n-4,i-(n-2))a^{2}} ) +   \sum_{i=n-1}^{n}r_{i}\underline{x^{2}P(n-3,i - (n-1))a^{2}} +  \underline{a^{2}Q(n-5,3)a^{2}}  ) -  r_{n-3}\underline{a^{n+2}}.
$$

Similarly, using (\ref{b}), the right hand side of $\sigma_{n-2}$ becomes 
$$
a^{n-2}x^{2} \rightarrow - ( \sum_{i=n-2}^{n-1}r_{i} \underline{a^{2}P(n-4,i-(n-2))} + r_{n-1} (\underline{xaP(n-3,i-(n-1))} + \underline{axP(n-3,i-(n-1))}) 
$$
$$
+ \underline{xaP(n-3,1)} + \underline{axP(n-3,1)} + \underline{x^{2}P(n-2,0)} + \underline{a^{2}Q(n-4,2)} ) + r_{n-2}\underline{a^{n}}.
$$
Thus, post multiplying this by $ax$ and $xa$ and using Lemmas 2.6 and 2.7 to separate reducible and irreducible words, yield
\begin{equation}
\tag{$III$}
-a^{n-2}x^{2}ax \rightarrow (  \sum_{i=n-2}^{n-1}r_{i} \underline{a^{2}P(n-4,i-(n-2))ax} + \sum_{i=n-1}^{n}r_{i} (\underline{xaP(n-3,i-(n-1))ax}  
\end{equation}
$$
+ \underline{axP(n-3,i-(n-1))ax}) + x^{2}P(n-2,0)ax + \underline{a^{2}Q(n-4,2ax)} ) - r_{n-2}\underline{axa^{n}}.
$$
and
\begin{equation}
\tag{$IV$}
-a^{n-2}x^{2}xa \rightarrow (\sum_{i=n-2}^{n-1}r_{i} \underline{a^{2}P(n-4,i-(n-2))xa} +  \sum_{i=n-1}^{n}r_{i} (xaP(n-3,i-(n-1))xa 
\end{equation}
$$
 + \underline{axP(n-3,i-(n-1))xa}) +  \underline{x^{2}P(n-2,0)xa} + \underline{a^{2}Q(n-4,2)xa} ) - r_{n-2}\underline{xa^{n+1}}
$$
respectively.
Turning now to the reducible word $x^{2}a^{n-1}x$, using (\ref{a}), the right hand side of $\sigma_{n-1}$ becomes
$$
a^{n-1}x  \rightarrow - \left( \underline{P(n-2,0)xa} + \underline{P(n-3,1)a^{2}} + r_{n-1}\underline{a^{n-1}}  \right) + r_{n-1}\underline{a^{n}}.
$$
Premultiplying this by $x^{2}$ and using Lemmas 2.6 and 2.7 to separate reducible and irreducible words, yields
\begin{equation}
\tag{$V$}
x^{2}a^{n-1}x  \rightarrow - \left( \underline{x^{2}P(n-2,0)xa} + \underline{x^{2}P(n-3,1)a^{2}} + r_{n-1}\underline{x^{2}a^{n-1}}  \right) + r_{n-1}\underline{x^{2}a^{n}}.
\end{equation}

Returning to (IV), the reducible word $xa^{n-2}x^{2}a \in xaP(n-3,1)xa$ is reduced as follows. We use (\ref{c}) to write the right hand side of $\sigma_{n-2}$ as 
$$
a^{n-2}x^{2} \rightarrow - ( \sum_{i=n-2}^{n-1}r_{i}  \underline{aP(n-4,i-(n-2))a}  +  r_{n-1}(\underline{aP(n-3,0)x} + \underline{x P(n-3,0)a} ) 
$$
$$
+ \underline{xP(n-2,0)x} + \underline{aQ(n-3,1)x}   +  \underline{aP(n-4,2)a} + \underline{x P(n-3,1)a}  ) + r_{n-2}\underline{a^{n}}.
$$
$$
= - ( \sum_{i=n-2}^{n}r_{i}  \underline{aP(n-4,i-(n-2))a}  +  r_{n-1}\underline{aP(n-3,0)x} + \sum_{i=n-1}^{n}r_{i} \underline{x P(n-3,i-(n-1))a}  
$$
$$
+ \underline{xP(n-2,0)x} + \underline{aQ(n-3,1)x}   ) + r_{n-2}\underline{a^{n}}.
$$
Therefore, pre and post multiplying this by $x$ and $a$ respectively and using Lemmas 2.6 and 2.7 to separate reducible and irreducible words, yields
\begin{equation}
\tag{$VI$}
xa^{n-2}x^{2}a \rightarrow - (  \sum_{i=n-2}^{n}r_{i}  \underline{xaP(n-4,i-(n-2))a^{2}}  +  r_{n-1}\underline{xaP(n-3,0)xa} + 
\end{equation}
$$
\sum_{i=n-1}^{n}r_{i} \underline{x^{2} P(n-3,i-(n-1))a^{2}}   + \underline{x^{2}P(n-2,0)xa} + \underline{xaQ(n-3,1)xa}     ) + r_{n-2}\underline{xa^{n+1}}.
$$
Thus, the reduction process stops here and we get the following after assemblying (I), (II), (III), (IV), (V), (VI) and simplifying yields:
\begin{equation}
\tag{$\Gamma$}
a^{2}\omega_{n-3} \rightarrow (  \sum_{i=n-1}^{n}r_{i}( \underline{xaP(n-3,i-(n-1))ax} + \underline{axP(n-3,i-(n-1))ax} + \underline{axP(n-3,i-(n-1))xa}) 
\end{equation}
$$
+ \sum_{i=n-2}^{n}r_{i}\underline{axP(n-4,i-(n-2))a^{2}} + r_{n-1}\underline{x^{2}a^{n}}) - ( r_{n-2}\underline{axa^{n}} + r_{n-1}a^{2}P(n-3,0)x^{2} +  
$$
$$
\underline{a^{2}Q(n-3,1)x^{2}} +  \underline{x^{2}P(n-2,0)xa}  + \sum_{i=n-1}^{n}r_{i} \underline{x^{2}P(n-3,i-(n-1))a^{2}}  ).
$$
\\

 Now, consider the alternate grouping of the overlap ambiguity, namely $\omega_{n-1}x^{2}=(a^{n-1}x)x^{2}$. Using (\ref{b}), the words on the right hand side of $\sigma_{n-1}$ are written as
$$
\omega_{n-1} \rightarrow - \left( \underline{a^{2}Q(n-3,1)} + \underline{axP(n-2,0)} + \underline{xaP(n-2,0)}  + r_{n-1}\underline{a^{n-1}} \right) + r_{n-1}\underline{a^{n}}.
$$
Postmultiplying this by $x^{2}$ and using Lemmas 2.6 and 2.7 to separate reducible and irreducible words, yields
\begin{equation}
\tag{$a$}
\omega_{n-1}x^{2} \rightarrow  - \left( \underline{a^{2}Q(n-3,1)x^{2}} + axP(n-2,0)x^{2} + xaP(n-2,0)x^{2}  + r_{n-1}a^{n-1}x^{2} \right) + r_{n-1}\underline{x^{2}a^{n}}.
\end{equation}
The reducible words of length $n+2$ above are
$$
axa^{n-2}x^{2} \qquad \text{and} \qquad xaa^{n-2}x^{2}.
$$
Using (\ref{a}), the right hand side of $\sigma_{n-2}$ becomes
$$
a^{n-2}x^{2} \rightarrow - ( \sum_{i=n-1}^{n}r_{i} (  \underline{P(n-4,i-(n-2))a^{2}} + \underline{P(n-3,i-(n-1))xa})
$$
$$
 + \underline{P(n-3,i-(n-1))ax} +  r_{n-2}\underline{a^{n-2}} ) + r_{n-2}\underline{a^{n}}.
$$
Premultiplying this by $ax$ and $xa$ gives
\begin{equation}
\tag{$b$}
-axa^{n-2}x^{2} \rightarrow  ( \sum_{i=n-1}^{n}r_{i} (  \underline{axP(n-4,i-(n-2))a^{2}} + \underline{axP(n-3,i-(n-1))xa} + 
\end{equation}
$$
\underline{axP(n-3,i-(n-1))ax} ) +  r_{n-2}\underline{axa^{n-2}} ) -  r_{n-2}\underline{axa^{n}}
$$
and
\begin{equation}
\tag{$c$}
- xaa^{n-2}x^{2} \rightarrow ( \sum_{i=n-1}^{n}r_{i} (  \underline{xaP(n-4,i-(n-2))a^{2}} + xaP(n-3,i-(n-1))xa + 
\end{equation}
$$
\underline{xaP(n-3,i-(n-1))ax} ) +  r_{n-2}\underline{xa^{n-1}} ) + r_{n-2}\underline{xa^{n+1}}
$$
respectively.
The reducible word $xa^{n-2}x^{2}a$ is given by $(VI)$. Thus, the reduction process ends here and assemblying (a), (b), (c) and (VI), we obtain
\begin{equation}
\tag{$\Gamma'$}
\omega_{n-1}x^{2}  \rightarrow (  \sum_{i=n-1}^{n}r_{i}( \underline{xaP(n-3,i-(n-1))ax} + \underline{axP(n-3,i-(n-1))ax} + \underline{axP(n-3,i-(n-1))xa}) 
\end{equation}
$$
+ \sum_{i=n-2}^{n}r_{i}\underline{axP(n-4,i-(n-2))a^{2}} + r_{n-1}\underline{x^{2}a^{n}}) - ( r_{n-2}\underline{axa^{n}} + r_{n-1}a^{2}P(n-3,0)x^{2} +  
$$
$$
\underline{a^{2}Q(n-3,1)x^{2}} +  \underline{x^{2}P(n-2,0)xa} + \sum_{i=n-1}^{n}r_{i} \underline{x^{2}P(n-3,i-(n-1))a^{2}}  ).
$$
Therefore, comparing ($\Gamma$) and ($\Gamma'$), we conclude that the overlap ambiguity $\{ \omega_{n-3}, \omega_{n-1} \}$ is resolvable.
\\

\noindent (c) Suppose now that $ 1 < j < n-3$, so that $n \geq 6$. Consider the overlap ambiguity $\{ \omega_{j}, \omega_{j+2} \}$. We use (\ref{a}) to write the right hand side of $\sigma_{j}$ as
$$
\omega_{j}  \rightarrow - (  \sum_{i=j}^{n-1} r_{i} \underline{P(j-2,i-j)a^{2}} +  \sum_{i=j+1}^{n-1} r_{i} \left( \underline{P(j-1, i-j-1)ax} + \underline{P(j-1,i-j-1)xa}    \right)  +
$$
$$
  \sum_{i=j+2}^{n-1}\underline{P(j,i-j-2)x^{2}} \underline{Q(j,n-j-2)x^{2}} + \underline{P(j-2,n-j)a^{2}} + \underline{P(j-1, n-j-1)ax} 
$$
$$
+ \underline{P(j-1,n-j-1)xa}  ) + r_{j}\underline{a^{n}}
$$
$$
=  - (  \sum_{i=j}^{n} r_{i} \underline{P(j-2,i-j)a^{2}} +  \sum_{i=j+1}^{n} r_{i} \left( \underline{P(j-1, i-j-1)ax} + \underline{P(j-1,i-j-1)xa}    \right) +
$$
$$
 \sum_{i=j+2}^{n-1}\underline{P(j,i-j-2)x^{2}} + \underline{Q(j,n-j-2)x^{2}} ) + r_{j}\underline{a^{n}}
$$
Premultiplying this by $a^{2}$ and using Lemmas 2.6 and 2.7 to separate reducible and irreducible words yields, 
\begin{equation}
\tag{$\mu$}
a^{2}\omega_{j}  \rightarrow - ( \sum_{i=j}^{n} r_{i} \underline{a^{2}P(j-2,i-j)a^{2}} + \sum_{i=j+1}^{n} r_{i} \left( a^{2}P(j-1, i-j-1)ax + a^{2}P(j-1,i-j-1)xa  \right) 
\end{equation}
$$
+ \sum_{i=j+2}^{n-1}a^{2}P(j,i-j-2)x^{2} +  \underline{a^{2}Q(j,n-j-2)x^{2}}  ) + r_{j}\underline{a^{n+2}}
$$

We have the following reducible words of length $n+2$ from $(\mu)$:
\begin{equation}
\tag{$M$}
a^{j+1}x^{n-j-1}ax, \; a^{j+1}x^{n-j-1}xa, \; a^{j}x^{n-j}a^{2}
\end{equation}
from $a^{2}P(j-1, n-j-1)ax$, $a^{2}P(j-1,n-j-1)xa$ and $a^{2}P(j-2,n-j)a^{2}$ respectively. The word $a^{j+1}x^{n-j-1}xa$ is $(a^{j+1}x^{n-j-1})xa$ and $a(a^{j}x^{n-j})a$. So it involves an overlap ambiguity, but for $t=1$ and this overlap ambiguity has been resolved in (3) and (4). Thus, we using the route $(a^{j+1}x^{n-j-1})xa$ will lead to the same result. 

We treat the reducible words in ($M$) in turn, starting with $a^{j+1}x^{n-j-1}ax$ and $a^{j+1}x^{n-j-1}xa$. Expand $\sigma_{j+1}$ using $(\ref{b})$ as
$$
a^{j+1}x^{n-j-1}  \rightarrow - (\sum_{i=j+1}^{n-1}\underline{a^{2}P(j-1,i-j-1)} + \underline{a^{2}Q(j-1,n-j-1)} +
$$
$$
  \sum_{i=j+2}^{n-1} r_{i} \left( \underline{axP(j, i-j-2)} + \underline{xaP(j,i-j-2)}  \right) + \underline{axP(j, n-j-2)} + \underline{xaP(j,n-j-2)} +
$$
$$
 \sum_{i=j+3}^{n-1}\underline{x^{2}P(j+1,i-j-3)} +\underline{x^{2}P(j+1,n-j-3)}  ) + r_{j+1}\underline{a^{n}}.
$$
$$
= - (\sum_{i=j+1}^{n-1}\underline{a^{2}P(j-1,i-j-1)} + \underline{a^{2}Q(j-1,n-j-1)} +
$$
$$
  \sum_{i=j+2}^{n} r_{i} \left( \underline{axP(j, i-j-2)} + \underline{xaP(j,i-j-2)}  \right) + \sum_{i=j+3}^{n}\underline{x^{2}P(j+1,i-j-3)}  ) + r_{j+1}\underline{a^{n}}.
$$
Thus, post multiplying this by $ax$ and $xa$ and using Lemmas 2.6 and 2.7 to separate reducible and irreducible words yields, 
\begin{equation}
\tag{$\omega$}
- a^{j+1}x^{n-j-1}ax  \rightarrow  (\sum_{i=j+1}^{n-1}\underline{a^{2}P(j-1,i-j-1)ax} + \underline{a^{2}Q(j-1,n-j-1)ax} +
\end{equation}
$$
 \sum_{i=j+2}^{n} r_{i} \left( \underline{axP(j, i-j-2)ax} + \underline{xaP(j,i-j-2)ax}  \right) + \sum_{i=j+3}^{n}\underline{x^{2}P(j+1,i-j-3)ax}  ) - r_{j+1}\underline{axa^{n}}
$$
and
\begin{equation}
\tag{$\zeta$}
- a^{j+1}x^{n-j-1}xa  \rightarrow   (\sum_{i=j+1}^{n-1}\underline{a^{2}P(j-1,i-j-1)xa} + \underline{a^{2}Q(j-1,n-j-1)xa} +
\end{equation}
$$
 \sum_{i=j+2}^{n} r_{i} \left( \underline{axP(j, i-j-2)xa} + \underline{xaP(j,i-j-2)xa}  \right) + \sum_{i=j+3}^{n}\underline{x^{2}P(j+1,i-j-3)xa}  ) - r_{j+1}\underline{xa^{n+1}}
$$
respectively.
Turning now to the reduction of $xa^{j+1}x^{n-j-1}a$ in ($\zeta$), use (\ref{c}) to expand $\sigma_{j+1}$ as
$$
a^{j+1}x^{n-j-1}  \rightarrow - ( \sum_{i=j+1}^{n-1} r_{i} \underline{aP(j-1,i-j-1)a} + \underline{aP(j-1,n-j-1)a} + \sum_{i=j+2}^{n-1} r_{i} \underline{aP(j,i-j-2)x} +
$$
$$
  \underline{aQ(j,n-j-2)x}  + \sum_{i=j+2}^{n-1} r_{i}\underline{xP(j,i-j-2)a} + \underline{xP(j,n-j-2)a} + \sum_{i=j+3}^{n} r_{i}\underline{xP(j+1,i-j-3)x} 
$$
$$
+ \underline{xP(j+1,n-j-3)x}  ) + r_{j+1}\underline{a^{n}}
$$
$$
= - ( \sum_{i=j+1}^{n} r_{i} \underline{aP(j-1,i-j-1)a} + \sum_{i=j+2}^{n-1} r_{i} \underline{aP(j,i-j-2)x} + \underline{aQ(j,n-j-2)x} 
$$
$$
   + \sum_{i=j+2}^{n} r_{i}\underline{xP(j,i-j-2)a} + \sum_{i=j+3}^{n} r_{i}\underline{xP(j+1,i-j-3)x} ) + r_{j+1}\underline{a^{n}}.
$$
Pre and post multiplying this by $x$ and $a$ respectively and using Lemmas 2.6 and 2.7 to separate reducible and irreducible words yields, 
\begin{equation}
\tag{$N$}
xa^{j+1}x^{n-j-1}a  \rightarrow - ( \sum_{i=j+1}^{n} r_{i} \underline{xaP(j-1,i-j-1)a^{2}}  + \sum_{i=j+2}^{n-1} r_{i} \underline{xaP(j,i-j-2)xa} +
\end{equation}
$$
  \underline{xaQ(j,n-j-2)xa}  + \sum_{i=j+2}^{n} r_{i}\underline{x^{2}P(j,i-j-2)a^{2}} + \sum_{i=j+3}^{n} r_{i}\underline{x^{2}P(j+1,i-j-3)xa} ) + r_{j+1}\underline{xa^{n+1}}.
$$
Turning now to the reducible word $a^{j}x^{n-j}a^{2}$, the third reducible term in ($M$), use (\ref{b}) to expand $\sigma_{j}$  as
$$
a^{j}x^{n-j}  \rightarrow - (\sum_{i=j}^{n-1} r_{i}\underline{a^{2}P(j-2,i-j)} + \underline{a^{2}Q(j-2,n-j)} + \sum_{i=j+1}^{n-1} r_{i} ( \underline{axP(j-1,i-j-1)}  +
$$
$$
   \underline{xaP(j-1,i-j-1)}) +  \underline{axP(j-1,n-j-1)} +  \underline{xaP(j-1,n-j-1)} + \sum_{i=j+2}^{n-1} r_{i}\underline{x^{2}P(j,i-j-2)}
$$
$$
+ \underline{x^{2}P(j,n-j-2)} )  + r_{j}\underline{a^{n}}
$$
$$
= - (\sum_{i=j}^{n-1} r_{i}\underline{a^{2}P(j-2,i-j)} + \underline{a^{2}Q(j-2,n-j)} + \sum_{i=j+1}^{n} r_{i} ( \underline{axP(j-1,i-j-1)}  
$$
$$
 +  \underline{xaP(j-1,i-j-1)})  + \sum_{i=j+2}^{n} r_{i}\underline{x^{2}P(j,i-j-2)} )  + r_{j}\underline{a^{n}}.
$$
Thus, post multiplying this by $a^{2}$ and using  Lemmas 2.6 and 2.7 to separate reducible and irreducible words yields, 
\begin{equation}
\tag{$P$}
- a^{j}x^{n-j}a^{2}  \rightarrow  (\sum_{i=j}^{n-1} r_{i}\underline{a^{2}P(j-2,i-j)a^{2}} + \underline{a^{2}Q(j-2,n-j)a^{2}} + \sum_{i=j+1}^{n} r_{i} ( \underline{axP(j-1,i-j-1)a^{2}}
\end{equation}
$$
 +  \underline{xaP(j-1,i-j-1)a^{2}})  + \sum_{i=j+2}^{n} r_{i}\underline{x^{2}P(j,i-j-2)a^{2}} )  - r_{j}\underline{a^{n+2}} .
$$
Hence, the reduction process ends here and assemblying $(\mu)$, ($M$), ($\omega$), ($\zeta$), ($N$), and ($P$), we obtain

\begin{equation}
\tag{$Q$}
a^{2}\omega_{j} \rightarrow (\sum_{i=j+1}^{n} r_{i}\underline{axP(j-1,i-j-1)a^{2}} + \sum_{i=j+2}^{n} r_{i} (\underline{xaP(j,i-j-2)ax} + \underline{axP(j,i-j-2)ax}   
\end{equation}
$$
+  \underline{axP(j,i-j-2)xa}) +  \sum_{i=j+3}^{n} r_{i} \underline{x^{2}P(j+1,i-j-3)ax} )   - (  \sum_{i=j+2}^{n-1} r_{i} a^{2}P(j,i-j-3)x^{2} + 
$$
$$
\underline{a^{2}Q(j,n-j-2)x^{2}} +  r_{j+1}  \underline{axa^{n}})
$$

Turning now to the other half of the overlap ambiguity, use (\ref{b}) to expand $\sigma_{j+2}$ as
$$
\omega_{j+2} \rightarrow - ( \sum_{i=j+2}^{n-1} r_{i}\underline{a^{2}P(j,i-j-2)} +  \underline{a^{2}Q(j,n-j-2)} + \sum_{i=j+3}^{n-1} r_{i}( \underline{axP(j+1,i-j-3)} +   
$$
$$
  \underline{xaP(j+1,i-j-3)}) + \underline{axP(j+1,n-j-3)} +   \underline{xaP(j+1,n-j-3)} +  \sum_{i=j+4}^{n-1} r_{i}\underline{x^{2}P(j+2,i-j-4)} 
$$
$$
+ \underline{x^{2}P(j+2,n-j-4)}  ) + \underline{r_{j+2}a^{n}}
$$
$$
= - (\sum_{i=j+2}^{n-1} r_{i}\underline{a^{2}P(j,i-j-2)} +  \underline{a^{2}Q(j,n-j-2)} +  \sum_{i=j+3}^{n} r_{i}\left( \underline{axP(j+1,i-j-3)} +   \underline{xaP(j+1,i-j-3)}\right) 
$$
$$
+  \sum_{i=j+4}^{n} r_{i}\underline{x^{2}P(j+2,i-j-4)}  ) + \underline{r_{j+2}a^{n}}.
$$
Thus, post multiplying this by $x^{2}$ and using  Lemmas 2.6 and 2.7 to separate reducible and irreducible words yields, 
\begin{equation}
\tag{$R$}
\omega_{j+2}x^{2} \rightarrow - (\sum_{i=j+2}^{n-1} r_{i}a^{2}P(j,i-j-2)x^{2} +  \underline{a^{2}Q(j,n-j-2)x^{2}}  + \sum_{i=j+3}^{n} r_{i}(axP(j+1,i-j-3)x^{2} +   
\end{equation}
$$
 xaP(j+1,i-j-3)x^{2}) + \sum_{i=j+4}^{n} r_{i}x^{2}P(j+2,i-j-4)x^{2}    ) + r_{j+2}\underline{a^{n}x^{2}}.
$$
The reducible words in $(R)$ of length $n+2$ are 
$$
xaa^{j+1}x^{n-j-1}, \; axa^{j+1}x^{n-j-1}, \;  x^{2}a^{j+2}x^{n-j-2}
$$
which belong to $xaP(j+1,n-j-3)x^{2}$, $axP(j+1,n-j-3)x^{2} $ and $ x^{2}P(j+2,n-j-4)x^{2}$ respectively. The word $xaa^{j+1}x^{n-j-1}$ is $xa(a^{j+1}x^{n-j-1})$ and $x(a^{j+2}x^{n-j-2})x$. So it involves an overlap ambiguity, but for $t=1$ and this overlap ambiguity has been resolved in (3) and (4). 

To deal with the first two reducible words, first, use (\ref{a}) to expand $\sigma_{j+1}$  as
$$
a^{j+1}x^{n-j-1}  \rightarrow - ( \sum_{i=j+1}^{n-1} r_{i} \underline{P(j-1,i-j-1)a^{2}}   + \underline{P(j-1,n-j-1)a^{2}} + \sum_{i=j+2}^{n-1} r_{i} ( \underline{P(j, i-j-2)ax} + 
$$
$$
\underline{P(j,i-j-2)xa} ) +  \underline{P(j, n-j-2)ax} +  \underline{P(j,n-j-2)xa} + \sum_{i=j+3}^{n-1}\underline{P(j+1,i-j-3)x^{2}}
$$
$$
  \underline{Q(j+1,n-j-3)x^{2}} ) + r_{j+1}\underline{a^{n}}
$$
$$
= - ( \sum_{i=j+1}^{n} r_{i} \underline{P(j-1,i-j-1)a^{2}} + \sum_{i=j+2}^{n} r_{i} \left( \underline{P(j, n-j-2)ax} + \underline{P(j,n-j-2)xa}  \right) +
$$
$$
\sum_{i=j+3}^{n-1}\underline{P(j+1,i-j-3)x^{2}} + \underline{Q(j+1,n-j-3)x^{2}} ) + r_{j+1}\underline{a^{n}}
$$

Thus, premultiplying this $xa$ and $ax$ and using  Lemmas 2.6 and 2.7 to separate reducible and irreducible words yield,
\begin{equation}
\tag{$S$}
-xaa^{j+1}x^{n-j-1} \rightarrow (\sum_{i=j+1}^{n} r_{i} \underline{xaP(j-1,i-j-1)a^{2}} + \sum_{i=j+2}^{n} r_{i} \left( \underline{xaP(j, n-j-2)ax} + xaP(j,n-j-2)xa  \right) 
\end{equation}
$$
+ \sum_{i=j+3}^{n-1} xaP(j+1,i-j-3)x^{2} + \underline{xaQ(j+1,n-j-3)x^{2}} ) -  r_{j+1}\underline{xa^{n+1}}
$$
and
\begin{equation}
\tag{$T$}
-axa^{j+1}x^{n-j-1} \rightarrow (\sum_{i=j+1}^{n} r_{i} \underline{axP(j-1,i-j-1)a^{2}} + \sum_{i=j+2}^{n} r_{i} \left( \underline{axP(j, n-j-2)ax} + \underline{axP(j,n-j-2)xa}  \right) 
\end{equation}
$$
+ \sum_{i=j+3}^{n-1}\underline{axP(j+1,i-j-3)x^{2}} + \underline{axQ(j+1,n-j-3)x^{2}} ) -  r_{j+1}\underline{axa^{n}}
$$
respectively. The reduction of $xa^{j+1}x^{n-j-1}a$  to irreducible words is given in $(N)$.
To reduce $x^{2}a^{j+2}x^{n-j-2}$, first, use (\ref{a}) to expand $\sigma_{j+2}$ as
$$
a^{j+2}x^{n-j-2} \rightarrow - ( \sum_{i=j+2}^{n-1} r_{i}\underline{P(j,i-j-2)a^{2}} + \underline{P(j,n-j-2)a^{2}} +  \sum_{i=j+3}^{n} r_{i} ( \underline{P(j+1,i-j-3)ax} +   
$$
$$
\underline{P(j+1,i-j-3)xa} )  + \underline{P(j+1,n-j-3)ax} +   \underline{P(j+1,n-j-3)xa} +  \sum_{i=j+4}^{n-1} r_{i}\underline{P(j+2,i-j-4)x^{2}} 
$$
$$
+  \underline{Q(j,n-j-4)x^{2}}   ) + \underline{r_{j+2}a^{n}}
$$
$$
=  - ( \sum_{i=j+2}^{n} r_{i}\underline{P(j,i-j-2)a^{2}} +  \sum_{i=j+3}^{n} r_{i}\left( \underline{P(j+1,i-j-3)ax} +   \underline{P(j+1,i-j-3)xa}\right) +
$$
$$
   \sum_{i=j+4}^{n-1} r_{i}\underline{P(j+2,i-j-4)x^{2}} +  \underline{Q(j,n-j-4)x^{2}}   ) + r_{j+2}\underline{a^{n}}.
$$
Thus, premultiplying this by $x^{2}$ and using  Lemmas 2.6 and 2.7 to separate reducible and irreducible words yields,
\begin{equation}
\tag{$U$}
-x^{2}a^{j+2}x^{n-j-2} \rightarrow ( \sum_{i=j+2}^{n} r_{i}\underline{x^{2}P(j,i-j-2)a^{2}} +  \sum_{i=j+3}^{n} r_{i}( \underline{x^{2}P(j+1,i-j-3)ax} +   
\end{equation}
$$
 \underline{x^{2}P(j+1,i-j-3)xa}) + \sum_{i=j+4}^{n-1} r_{i}\underline{x^{2}P(j+2,i-j-4)x^{2}} +  \underline{x^{2}Q(j,n-j-4)x^{2}}   ) + r_{j+2}\underline{x^{2}a^{n}}.
$$ 
Thus, the reduction process ends here and when we assemble $(R)$, $(S)$, $(T)$, $(N)$ and $(U)$, we obtain
\begin{equation}
\tag{$V$}
\omega_{j+2}x^{2}\rightarrow   (\sum_{i=j+1}^{n} r_{i}\underline{axP(j-1,i-j-1)a^{2}} + \sum_{i=j+2}^{n} r_{i} (\underline{xaP(j,i-j-2)ax} + \underline{axP(j,i-j-2)ax}  +   
\end{equation}
$$
\underline{axP(j,i-j-2)xa}) + \sum_{i=j+3}^{n} r_{i}\underline{x^{2}P(j+1,i-j-3)ax} )   - (  \sum_{i=j+2}^{n-1} r_{i} a^{2}P(j,i-j-2)x^{2} +
$$
$$
\underline{a^{2}Q(j,n-j-2)x^{2}} +   r_{j+1}  \underline{axa^{n}}).
$$
Comparing $(Q)$ and $(V)$, we conclude that the overlap ambiguity $\{ \omega_{j}, \omega_{j+2} \}$ is resolvable for all $j$ with $1  <  j < n-3$. Thus, Proposition 2.8(ii) follows from cases (a), (b) and (c).\\

\noindent (iii) We get an ambiguity from the two routes to resolve the word
$$
a^{3}\omega_{j} = a^{3}(a^{j}x^{n-j}) = (a^{j+3}x^{n-j-3})x^{3}=\omega_{j+3}x^{3}
$$
in the free algebra $\field \langle a, x \rangle$ uring the relations $\sigma_{j}$ and $\sigma_{j+3}$ for $ 1 \leq j \leq n-4$. Throughout the proofs, we use the following identities:
\begin{equation}
\label{d}
P(r,s) = P(r-3,s)a^{3} + P(r-2,s-1)(a^{2}x + axa +xa^{2}) + P(r-1,s-2)(ax^{2} + xax + x^{2}a) + P(r,s-3)x^{3}
\end{equation}
\begin{equation}
\label{e}
P(r,s) = a^{3}P(r-3,s) + (a^{2}x + axa +xa^{2}) P(r-2,s-1)+ (ax^{2} + xax + x^{2}a)P(r-1,s-2) + x^{3}P(r,s-3)
\end{equation}
\begin{equation}
\label{f}
P(r,s) = \left( x^{2}P(r-1,s-2) + axP(r-2,s-1) + xaP(r-2,s-1) + a^{2}P(r-3,s) \right)a + 
\end{equation}
$$
\left( x^{2}P(r,s-3) + axP(r-1,s-2) + xaP(r-1,s-2) + a^{2}P(r-2,s-1) \right)x
$$
\begin{equation}
\label{g}
P(r,s) = a\left( P(r-1,s-2)x^{2} + P(r-2,s-1)ax + P(r-2,s-1)xa + P(r-3,s)a^{2} \right) + 
\end{equation}
$$
x\left( P(r,s-3)x^{2} + P(r-1,s-2)ax + P(r-1,s-2)xa + P(r-2,s-1)a^{2} \right).
$$
We resolve the amabiguities in four cases.

\noindent (a) Let $j=1$ and consider the overlap ambiguity $\{ \omega_{1}, \omega_{4} \}$. Use (\ref{d}) to expand $\sigma_{1}$ as
$$
\omega_{1} \rightarrow - (  \sum_{i=3}^{n-1}r_{i}\underline{P(0,i-3)( x^{2}a + xax + ax^{2} )} + \underline{P(0,n-3)( x^{2}a + xax + ax^{2} )} 
$$
$$
+ \underline{Q(1,n-4)x^{3}} +  \sum_{i=4}^{n-1}r_{i}\underline{P(1,i-4)x^{3}} + r_{2}\underline{P(1,1)} + r_{1}\underline{a}) + r_{1}\underline{a^{n}}
$$
$$
= - (  \sum_{i=3}^{n}r_{i}\underline{P(0,i-3)( x^{2}a + xax + ax^{2} )} + \sum_{i=4}^{n-1}r_{i}\underline{P(1,i-4)x^{3}} + \underline{Q(1,n-4)x^{3}} + r_{2}\underline{P(1,1)} + r_{1}\underline{a}) + r_{1}\underline{a^{n}}.
$$
Thus, premultiplying this by $a^{3}$ and using Lemmas 2.6 and 2.7 to separate reducible and irredicible words, yields
\begin{equation}
\label{3t1}
a^{3} \omega_{1} \rightarrow - (  \sum_{i=3}^{n}r_{i}a^{3} P(0,i-3)( x^{2}a + xax + ax^{2} ) + \sum_{i=4}^{n-1}r_{i}\underline{a^{3} P(1,i-4)x^{3}} + 
\end{equation}
$$
\underline{a^{3}Q(1,n-4)x^{3}} + r_{2}\underline{a^{3} P(1,1)} + r_{1}\underline{a^{4}}) + r_{1}\underline{a^{n+3}}.
$$
The following words in $(\ref{3t1})$ of length $n+3$ are reducible:
$$
a^{3}x^{n-1}a, \; a^{3}x^{n-3}ax^{2}, \; a^{3}x^{n-2}ax.
$$ 
We first reduce $a^{3}x^{n-1}a$ as follows. Use (\ref{e}) to expand $\sigma_{3}$ as
$$
a^{3}x^{n-3} \rightarrow - ( \sum_{i=3}^{n-1}r_{i}  \underline{a^{3}P(0,i-3)} + \sum_{i=4}^{n-1}r_{i}\underline{(xa^{2} + axa + a^{2}x) P(1,i-4)} +  
$$
$$
\underline{(xa^{2} + axa + a^{2}x) P(1,n-4)}  + \sum_{i=5}^{n-1}r_{i}\underline{(x^{2}a + xax + ax^{2})P(2,i-5)} +
$$
$$
 \underline{(x^{2}a + xax + ax^{2})P(2,n-5)} +  \sum_{i=6}^{n-1}r_{i}\underline{x^{3}P(3,i-6)} ) + \underline{x^{3}P(3,n-6)} ) + r_{3}\underline{a^{n}}
$$
$$
= - ( \sum_{i=3}^{n-1}r_{i}  \underline{a^{3}P(0,i-3)} + \sum_{i=4}^{n}r_{i}\underline{(xa^{2} + axa + a^{2}x) P(1,i-4)} +  
$$
$$
\sum_{i=5}^{n}r_{i}\underline{(x^{2}a + xax + ax^{2})P(2,i-5)} +  \sum_{i=6}^{n}r_{i}\underline{x^{3}P(3,i-6)} ) + r_{3}\underline{a^{n}}
$$
Post multiplying this by $x^{2}a$, $xax$ and $ax^{2}$, and using Lemmas 2.6 and 2.7 to separate reducible and irredicible words, yields
\begin{equation}
\label{3t2}
- a^{3}x^{n-1}a \rightarrow ( \sum_{i=3}^{n-1}r_{i}  a^{3}P(0,i-3)x^{2}a + \sum_{i=4}^{n}r_{i} (xa^{2} + axa + a^{2}x) P(1,i-4)x^{2}a + 
\end{equation}
$$
\sum_{i=5}^{n}r_{i} \underline{(x^{2}a + xax + ax^{2})P(2,i-5)x^{2}a} +  \sum_{i=6}^{n}r_{i}\underline{x^{3}P(3,i-6)x^{2}a} ) - r_{3}\underline{x^{2}a^{n+1}}, 
$$
\begin{equation}
\label{3t3}
- a^{3}x^{n-2}ax \rightarrow  ( \sum_{i=3}^{n-1}r_{i}  \underline{a^{3}P(0,i-3)xax} + \sum_{i=4}^{n}r_{i} (xa^{2} + axa + a^{2}x) P(1,i-4)xax +   
\end{equation}
$$
\sum_{i=5}^{n}r_{i}\underline{(x^{2}a + xax + ax^{2})P(2,i-5)xax} +  \sum_{i=6}^{n}r_{i}\underline{x^{3}P(3,i-6)xax} ) - r_{3}\underline{xaxa^{n}}
$$
and
\begin{equation}
\label{3t4}
- a^{3}x^{n-3}ax^{2} \rightarrow ( \sum_{i=3}^{n-1}r_{i}  \underline{a^{3}P(0,i-3)ax^{2}} + \sum_{i=4}^{n}r_{i}\underline{(xa^{2} + axa + a^{2}x) P(1,i-4)ax^{2}} +  
\end{equation}
$$
\sum_{i=5}^{n}r_{i}\underline{(x^{2}a + xax + ax^{2})P(2,i-5)ax^{2}} +  \sum_{i=6}^{n}r_{i}\underline{x^{3}P(3,i-6)ax^{2}} ) -  r_{3}\underline{ax^{2}a^{n}} 
$$
respectively. The reducible words in (\ref{3t2})-(\ref{3t4}) of length $n+3$ are as follows:
\begin{enumerate}
\item $axa^{2}x^{n-2}a \;  \in \; axa P(1,n-4) x^{2}a $,
\item $x^{2}a^{3}x^{n-3}a \;  \in \;  x^{2}a P(2,n-5) x^{2}a$,
\item $xa^{3}x^{n-3}xa \; \in \; xa^{2} P(1,n-4) x^{2}a$,
\item $ xa^{3}x^{n-3}ax \; \in \; xa^{2} P(1,n-4) xax $.
\end{enumerate}
Turning now to reduce $axa^{2}x^{n-2}a$, use (\ref{g}) to expand $\sigma_{2}$ as
$$
a^{2}x^{n-2} \rightarrow - (  \sum_{i=3}^{n-1}r_{i} ( \underline{aP(0,i-3)ax} + \underline{aP(0,i-3)xa} + \underline{xP(0,i-3)a^{2}} ) + \underline{aP(0,n-3)ax} + 
$$
$$
 \underline{aP(0,n-3)xa} + \underline{xP(0,n-3)a^{2}} + \sum_{i=4}^{n-1}r_{i} \underline{a P(1,i-4)x^{2}} + \underline{aQ(1,n-4)x^{2}} + 
$$
$$
\sum_{i=4}^{n-1}r_{i} ( \underline{xP(1,i-4)ax} + \underline{xP(1,i-4)xa} ) + \underline{xP(1,n-4)ax} + \underline{xP(1,n-4)xa} + 
$$
$$
\sum_{i=5}^{n-1}r_{i}\underline{xP(2,i-5)x^{2}} + \underline{xP(2,n-5)x^{2}} +  r_{2}\underline{a^{2}} ) + r_{2}\underline{a^{n}}
$$
$$
= - (\sum_{i=3}^{n}r_{i} ( \underline{aP(0,i-3)ax} + \underline{aP(0,i-3)xa} + \underline{xP(0,i-3)a^{2}} ) +  \sum_{i=4}^{n-1}r_{i} \underline{a P(1,i-4)x^{2}} + \underline{aQ(1,n-4)x^{2}}
$$
$$
  + \sum_{i=4}^{n}r_{i} ( \underline{xP(1,i-4)ax} + \underline{xP(1,i-4)xa} ) + \sum_{i=5}^{n}r_{i}\underline{xP(2,i-5)x^{2}} +    r_{2}\underline{a^{2}} ) + r_{2}\underline{a^{n}}.
$$
Pre and postmultiply this by $ax$ and $a$ respectively and use Lemmas 2.6 and 2.7 to separate reducible and irredicible words to get
\begin{equation}
\label{3t5}
axa^{2}x^{n-2}a  \rightarrow  -  ( \sum_{i=3}^{n}r_{i} ( \underline{axaP(0,i-3)axa} + \underline{axaP(0,i-3)xa^{2}} + ax^{2}P(0,i-3)a^{3} ) + 
\end{equation}
$$
 \sum_{i=4}^{n-1}r_{i} \underline{axa P(1,i-4)x^{2}a} + \underline{axaQ(1,n-4)x^{2}a} + \sum_{i=4}^{n}r_{i} ( \underline{ax^{2}P(1,i-4)axa} + 
$$
$$
\underline{ax^{2}P(1,i-4)xa^{2}} ) + \sum_{i=5}^{n}r_{i}\underline{ax^{2}P(2,i-5)x^{2}a} +    r_{2}\underline{axa^{3}} ) + r_{2}\underline{axa^{n+1}}.
$$
The only reducible word above is $ax^{n-1}a^{3} \in ax^{2}P(0,n-3)a^{3}$. To reduce $ax^{n-1}a^{3}$, use (\ref{e}) to expand $\sigma_{1}$ as
$$
ax^{n-1} \rightarrow - ( \sum_{i=4}^{n-1}r_{i}\underline{x^{3} P(1,i-4)} + \underline{x^{3} P(1,n-4)} + \sum_{i=3}^{n-1}r_{i}\underline{(x^{2}a + xax + ax^{2})P(0,i-3)} 
$$
$$
+ r_{2}\underline{P(1,1)} + r_{1}\underline{a} ) + r_{1}\underline{a^{n}}
$$
$$
= - ( \sum_{i=4}^{n}r_{i}\underline{x^{3} P(1,i-4)} +  \sum_{i=3}^{n-1}r_{i}\underline{(x^{2}a + xax + ax^{2})P(0,i-3)} + r_{2}\underline{P(1,1)} + r_{1}\underline{a} ) + r_{1}\underline{a^{n}}.
$$
Post multiplying this by $a^{3}$ and using Lemmas 2.6 and 2.7 to separate reducible and irredicible words, yields
\begin{equation}
\label{3t6}
- ax^{n-1}a^{3} \rightarrow ( \sum_{i=4}^{n}r_{i}\underline{x^{3} P(1,i-4)a^{3}} +  \sum_{i=3}^{n-1}r_{i}\underline{(x^{2}a + xax + ax^{2})P(0,i-3)a^{3}} + r_{2}\underline{P(1,1)a^{3}} + r_{1}\underline{a^{4}} ) 
\end{equation}
$$
-  r_{1}\underline{a^{n+3}} .
$$
Turning now to the reduction of $xa^{3}x^{n-3}ax$ and $xa^{3}x^{n-3}xa$, use (\ref{f}) to expand $\sigma_{3}$ as
$$
a^{3}x^{n-3} \rightarrow - ( \sum_{i=3}^{n-1}r_{i} \underline{a^{2}P(0,i-3)a} + \underline{a^{2}P(0,n-3)a} + \sum_{i=4}^{n-1}r_{i} \underline{a ^{2}P(1,i-4)x} +  \underline{a^{2}Q(1,n-4)x}  
$$
$$
+ \sum_{i=4}^{n-1}r_{i} ( \underline{ax P(1,i-4)a} + \underline{xa P(1,i-4)a} ) + \underline{ax P(1,n-4)a} + \underline{xa P(1,n-4)a} +
$$
$$
 \sum_{i=5}^{n-1}r_{i} ( \underline{x^{2} P(2,i-5)a} + \underline{xa P(2,i-5)x}  + \underline{ax P(2,i-5)x} ) +  \underline{x^{2} P(2,n-5)a} +
$$
$$
 \underline{xa P(2,n-5)x}  + \underline{ax P(2,n-5)x} + \sum_{i=6}^{n-1}r_{i}\underline{x^{2}P(3,i-6)x} +   \underline{x^{2}P(3,n-6)x}) + r_{3}\underline{a^{n}}
$$
$$
= - ( \sum_{i=3}^{n}r_{i} \underline{a^{2}P(0,i-3)a} + \sum_{i=4}^{n-1}r_{i} \underline{a ^{2}P(1,i-4)x} + \underline{a^{2}Q(1,n-4)x}   + \sum_{i=4}^{n}r_{i} ( \underline{ax P(1,i-4)a} 
$$
$$
+ \underline{xa P(1,i-4)a} ) +  \sum_{i=5}^{n}r_{i} ( \underline{x^{2} P(2,i-5)a} + \underline{xa P(2,i-5)x}  + \underline{ax P(2,i-5)x} ) +
$$
$$
  \sum_{i=6}^{n}r_{i}\underline{x^{2}P(3,i-6)x}) + r_{3}\underline{a^{n}}.
$$
Thus, premultiplying this by $x$ and post multiplying by $ax$ and $xa$ and using Lemmas 2.6 and 2.7 to separate reducible and irredicible words, yield
\begin{equation}
\label{3t7}
xa^{3}x^{n-3}ax \rightarrow - ( \sum_{i=3}^{n}r_{i} \underline{xa^{2}P(0,i-3)a^{2}x} + \sum_{i=4}^{n-1}r_{i} \underline{xa ^{2}P(1,i-4)xax} + \underline{xa^{2}Q(1,n-4)xax}   + 
\end{equation}
$$
 \sum_{i=4}^{n}r_{i} ( \underline{xax P(1,i-4)a^{2}x} + \underline{x^{2}a P(1,i-4)a^{2}x} ) +  \sum_{i=5}^{n}r_{i} ( \underline{x^{3} P(2,i-5)a^{2}x} + \underline{x^{2}a P(2,i-5)xax}  
$$
$$
+ \underline{xax P(2,i-5)xax} ) +  \sum_{i=6}^{n}r_{i}\underline{x^{3}P(3,i-6)xax}) + r_{3}\underline{xaxa^{n}}.
$$
and
\begin{equation}
\label{3t8}
xa^{3}x^{n-3}xa \rightarrow - ( \sum_{i=3}^{n}r_{i} \underline{xa^{2}P(0,i-3)axa} + \sum_{i=4}^{n-1}r_{i} \underline{xa ^{2}P(1,i-4)x^{2}a} + \underline{xa^{2}Q(1,n-4)x^{2}a}   +
\end{equation}
$$
 \sum_{i=4}^{n}r_{i} ( \underline{xax P(1,i-4)axa}+ \underline{x^{2}a P(1,i-4)axa} ) +  \sum_{i=5}^{n}r_{i} ( \underline{x^{3} P(2,i-5)axa} + x^{2}a P(2,i-5)x^{2}a 
$$
$$
+ \underline{xax P(2,i-5)x^{2}a} ) +  \sum_{i=6}^{n}r_{i}\underline{x^{3}P(3,i-6)x^{2}a}) + r_{3}\underline{x^{2}a^{n+1}}
$$
respectively. All the words in (\ref{3t7}) and (\ref{3t8}) are irreducible except $x^{2}a^{3}x^{n-3}a \;  \in \; x^{2}a P(2,n-5)x^{2}a$. But $x^{2}a^{3}x^{n-3}a$ appears with opposite sign to the same reducible word in (\ref{3t2}) so they cancel out. Thus, the reduction process ends here and substituting (\ref{3t2}), (\ref{3t3}), (\ref{3t4}), (\ref{3t5}), (\ref{3t6}), (\ref{3t7}), (\ref{3t8}) into (\ref{3t1}) and simplifying gives
\begin{equation}
\label{3t'}
a^{3}\omega_{1} \rightarrow ( r_{2}\underline{xa^{4}} + r_{2}\underline{axa^{n+1}} + \sum_{i=3}^{n-1}r_{i} \underline{( x^{2}a + xax )P(0,i-3)a^{3}} + \sum_{i=4}^{n}r_{i}( \underline{x^{3}P(1,i-4)a^{3}} + 
\end{equation}
$$
\underline{a^{2}xP(1,i-4)(ax^{2} + x^{2}a + xax )} + \underline{axaP(1,i-4)(ax^{2} + xax )} + \underline{xa^{2}P(1,i-4)ax^{2}})  + 
$$
$$
 \sum_{i=5}^{n}r_{i} (  \underline{ ( ax^{2}+ x^{2}a + xax )P(2,i-5)ax^{2}} +  \underline{ax^{2}P(2,i-5)xax}) + \sum_{i=6}^{n}r_{i} \underline{x^{3}P(3,i-6)ax^{2}}
$$
$$
 - ( r_{2}\underline{a^{3}P(1,1)} + r_{3}\underline{ax^{2}a^{n}}+  \sum_{i=3}^{n}r_{i} (  \underline{xa^{2}P(0,i-3)(a^{2}x + axa)} +  \underline{axaP(0,i-3)(axa + xa^{2})}) +  
$$
$$
\sum_{i=4}^{n}r_{i} (  \underline{(xax + ax^{2} + x^{2}a)P(1,i-4)axa} + \underline{(xax + x^{2}a )P(1,i-4)a^{2}x} + \underline{ax^{2}P(1,i-4)xa^{2}} )  
$$
$$
\sum_{i=4}^{n-1}r_{i}a^{3}P(1,i-4)x^{3}  + \underline{a^{3}Q(1,n-4)x^{3}}    + \sum_{i=5}^{n}r_{i}\underline{x^{3}P(2,i-5)(a^{2}x + axa)}.
$$

Moving now to $\omega_{4}x^{3}$, use (\ref{e}) to write $\sigma_{4}$ as
$$
\omega_{4} \rightarrow - ( \sum_{i=4}^{n-1}r_{i}\underline{a^{3}P(1,i-4)} + \underline{a^{3}Q(1,n-4)}  +  \sum_{i=5}^{n-1}r_{i} \underline{( a^{2}x + axa + xa^{2})P(2,i-5)}  +
$$
$$
\underline{( a^{2}x + axa + xa^{2})P(2,n-5)}   + \sum_{i=6}^{n-1}r_{i}\underline{(ax^{2} + xax + x^{2}a) P(3,i-6)} +
$$
$$
\underline{(ax^{2} + xax + x^{2}a) P(3,n-6)}  + \sum_{i=7}^{n-1}r_{i} \underline{x^{3}P(4,i-7)} + \underline{x^{3}P(4,n-7)}  ) + r_{4}\underline{a^{n}}
$$
$$
= - ( \sum_{i=4}^{n-1}r_{i}\underline{a^{3}P(1,i-4)} + \underline{a^{3}Q(1,n-4)} +  \sum_{i=5}^{n}r_{i} \underline{( a^{2}x + axa + xa^{2})P(2,i-5)}  
$$
$$
+ \sum_{i=6}^{n}r_{i}\underline{(ax^{2} + xax + x^{2}a) P(3,i-6)} + \sum_{i=7}^{n}r_{i} \underline{x^{3}P(4,i-7)}    ) + r_{4}\underline{a^{n}}.
$$
Post multiplying this by $x^{3}$ and using Lemmas 2.6 and 2.7 to separate reducible and irreducible words, yields 
\begin{equation}
\label{3t9}
\omega_{4}x^{3} \rightarrow - ( \sum_{i=4}^{n-1}r_{i}a^{3}P(1,i-4)x^{3} + \underline{a^{3}Q(1,n-4)x^{3}} +  \sum_{i=5}^{n}r_{i} ( a^{2}x + axa + xa^{2})P(2,i-5)x^{3}   
\end{equation}
$$
+ \sum_{i=6}^{n}r_{i}(ax^{2} + xax + x^{2}a) P(3,i-6)x^{3} + \sum_{i=7}^{n}r_{i} x^{3}P(4,i-7)x^{3}    ) + r_{4}\underline{x^{3}a^{n}}.
$$
We get the following reducible words of length $n+3$ from (\ref{3t9}): 
\begin{enumerate}
\item $xa^{2}(a^{2}x^{n-2}) \; \in \; xa^{2}P(2,n-5)x^{3}$,
\item $axa(a^{2}x^{n-2}) \; \in \; axaP(2,n-5)x^{3}$,
\item $a^{2}x(a^{2}x^{n-2}) \; \in \; a^{2}xP(2,n-5)x^{3}$,
\item $x^{2}a(a^{3}x^{n-3}) \; \in \; x^{2}aP(3,n-6)x^{3}$,
\item $ax^{2}(a^{3}x^{n-3}) \; \in \; ax^{2}P(3,n-6)x^{3}$,
\item $xax(a^{3}x^{n-3}) \; \in \; xaxP(3,n-6)x^{3}$,
\item $x^{3}(a^{4}x^{n-4}) \; \in \; x^{3}P(4,n-7)x^{3}$.
\end{enumerate}

Use (\ref{d}) to expand $\sigma_{2}$ as
$$
a^{2}x^{n-2} \rightarrow - ( \sum_{i=3}^{n-1}r_{i} \underline{ P(0,i-3)(xa^{2}+ axa + a^{2}x) }  + \underline{ P(0,n-3)(xa^{2}+ axa + a^{2}x) } + 
$$
$$
\sum_{i=4}^{n-1}r_{i}  \underline{P(1,i-4)(ax^{2} + xax + x^{2}a)} + \underline{P(1,n-4)(ax^{2} + xax + x^{2}a)} + 
$$
$$
\sum_{i=5}^{n-1}r_{i} \underline{P(2,i-5)x^{3}} + \underline{Q(2,n-5)x^{3}} + r_{2}\underline{a^{2}} ) + r_{2}\underline{a^{n}}.
$$
$$
= - ( \sum_{i=3}^{n}r_{i} \underline{ P(0,i-3)(xa^{2}+ axa + a^{2}x) }  + \sum_{i=4}^{n}r_{i}  \underline{P(1,i-4)(ax^{2} + xax + x^{2}a)} + 
$$
$$
\sum_{i=5}^{n-1}r_{i} \underline{P(2,i-5)x^{3}} + \underline{Q(2,n-5)x^{3}} + r_{2}\underline{a^{2}} ) + r_{2}\underline{a^{n}}.
$$
Thus, premultiplying this by $xa^{2}$, $axa$ and $a^{2}x$, and using Lemmas 2.6 and 2.7 to separate reducible and irreducible words, yield
\begin{equation}
\label{3t10}
- xa^{2}(a^{2}x^{n-2}) \rightarrow  ( \sum_{i=3}^{n}r_{i} xa^{2}P(0,i-3)(xa^{2}+ axa + a^{2}x)   + \sum_{i=4}^{n}r_{i}   xa^{2}P(1,i-4)(ax^{2} + xax + x^{2}a) +
\end{equation}
$$
\sum_{i=5}^{n-1}r_{i}  xa^{2}P(2,i-5)x^{3} + \underline{ xa^{2}Q(2,n-5)x^{3}} + r_{2}\underline{xa^{4}} ) - r_{2}\underline{xa^{n+2}},
$$
\begin{equation}
\label{3t11}
-axa(a^{2}x^{n-2}) \rightarrow ( \sum_{i=3}^{n}r_{i} \underline{axaP(0,i-3)(xa^{2}+ axa + a^{2}x) }  + \sum_{i=4}^{n}r_{i}  axaP(1,i-4)(ax^{2} + xax + x^{2}a) + 
\end{equation}
$$
\sum_{i=5}^{n-1}r_{i} \underline{axaP(2,i-5)x^{3}} + \underline{axaQ(2,n-5)x^{3}} + r_{2}\underline{axa^{3}} ) - r_{2}\underline{axa^{n+1}},
$$
and
\begin{equation}
\label{3t12}
-a^{2}x(a^{2}x^{n-2}) \rightarrow  ( \sum_{i=3}^{n}r_{i} a^{2}xP(0,i-3)(xa^{2}+ axa + a^{2}x)  + \sum_{i=4}^{n}r_{i}  \underline{a^{2}xP(1,i-4)(ax^{2} + xax + x^{2}a)} + 
\end{equation}
$$
\sum_{i=5}^{n-1}r_{i} \underline{a^{2}xP(2,i-5)x^{3}} + \underline{a^{2}xQ(2,n-5)x^{3}} + r_{2}\underline{a^{2}xa^{2}} ) - r_{2}\underline{a^{2}xa^{n}}
$$
respectively.
Also, use (\ref{d}) to expand $\sigma_{3}$ as
$$
a^{3}x^{n-3} \rightarrow - ( \sum_{i=3}^{n-1}r_{i}\underline{P(0,i-3)a^{3}} + \underline{P(0,n-3)a^{3}} + \sum_{i=4}^{n-1}r_{i} \underline{P(1,i-4)(a^{2}x + axa + xa^{2})} + 
$$
$$
\underline{P(1,n-4)(a^{2}x + axa + xa^{2})} + \sum_{i=5}^{n-1}r_{i} \underline{P(2,i-5)(ax^{2}  + xax + x^{2}a)}  +
$$
$$
\underline{P(2,n-5)(ax^{2}  + xax + x^{2}a)}  + \sum_{i=6}^{n-1}r_{i} \underline{P(3,i-6)x^{3}} + \underline{Q(3,n-6)x^{3}}  ) + r_{3}\underline{a^{n}}
$$
$$
= - ( \sum_{i=3}^{n}r_{i}\underline{P(0,i-3)a^{3}} + \sum_{i=4}^{n}r_{i} \underline{P(1,i-4)(a^{2}x + axa + xa^{2})} + 
$$
$$
\sum_{i=5}^{n}r_{i} \underline{P(2,i-5)(ax^{2}  + xax + x^{2}a)}  + \sum_{i=6}^{n-1}r_{i} \underline{P(3,i-6)x^{3}} + \underline{Q(3,n-6)x^{3}}  ) + r_{3}\underline{a^{n}}.
$$
Premultiplying this by $x^{2}a$, $ax^{2}$ and $xax$ and using Lemmas 2.6 and 2.7 to separate reducible and irreducible words, yield
\begin{equation}
\label{3t13}
-x^{2}a(a^{3}x^{n-3}) \rightarrow ( \sum_{i=3}^{n}r_{i}\underline{x^{2}aP(0,i-3)a^{3}} + \sum_{i=4}^{n}r_{i} \underline{x^{2}aP(1,i-4)(a^{2}x + axa + xa^{2})} + 
\end{equation}
$$
 \sum_{i=5}^{n}r_{i} x^{2}aP(2,i-5)(ax^{2}  + xax + x^{2}a)  + \sum_{i=6}^{n-1}r_{i} \underline{x^{2}aP(3,i-6)x^{3}} + \underline{x^{2}aQ(3,n-6)x^{3}}  ) - r_{3}\underline{x^{2}a^{n+1}},
$$
\begin{equation}
\label{3t14}
-ax^{2}(a^{3}x^{n-3}) \rightarrow  ( \sum_{i=3}^{n}r_{i}ax^{2}P(0,i-3)a^{3} + \sum_{i=4}^{n}r_{i} \underline{ax^{2}P(1,i-4)(a^{2}x + axa + xa^{2})} + 
\end{equation}
$$
+ \sum_{i=5}^{n}r_{i} \underline{ax^{2}P(2,i-5)(ax^{2}  + xax + x^{2}a)}  + \sum_{i=6}^{n-1}r_{i} \underline{ax^{2}P(3,i-6)x^{3}} +  \underline{ax^{2}Q(3,n-6)x^{3}}  ) - r_{3}\underline{ax^{2}a^{n}},
$$
and
\begin{equation}
\label{3t15}
-xax(a^{3}x^{n-3}) \rightarrow ( \sum_{i=3}^{n}r_{i}\underline{xaxP(0,i-3)a^{3}} + \sum_{i=4}^{n}r_{i} \underline{xaxP(1,i-4)(a^{2}x + axa + xa^{2})} + 
\end{equation}
$$
+  \sum_{i=5}^{n}r_{i} \underline{xaxP(2,i-5)(ax^{2}  + xax + x^{2}a)}  + \sum_{i=6}^{n-1}r_{i} \underline{xaxP(3,i-6)x^{3}} + \underline{xaxQ(3,n-6)x^{3}}  ) - r_{3}\underline{xaxa^{n}}
$$
respectively.
Moreover, use (\ref{d}) to expand $\sigma_{4}$ as
$$
a^{4}x^{n-4} \rightarrow - ( \sum_{i=4}^{n-1}r_{i}\underline{P(1,i-4)a^{3}} + \underline{P(1,n-4)a^{3}} +  \sum_{i=5}^{n-1}r_{i} \underline{P(2,i-5)(a^{2}x + axa + xa^{2})} + 
$$
$$
\underline{P(2,n-5)(a^{2}x + axa + xa^{2})} + \sum_{i=6}^{n-1}r_{i} \underline{P(3,i-6)(ax^{2}  + xax + x^{2}a)} 
$$
$$
 \underline{P(3,n-6)(ax^{2}  + xax + x^{2}a)}  + \sum_{i=7}^{n-1}r_{i} \underline{P(4,i-7)x^{3}} + \underline{Q(4,n-7)x^{3}}  ) + r_{4}\underline{a^{n}}
$$
$$
= - ( \sum_{i=4}^{n}r_{i}\underline{P(1,i-4)a^{3}} + \sum_{i=5}^{n}r_{i} \underline{P(2,i-5)(a^{2}x + axa + xa^{2})} + 
$$
$$
\sum_{i=6}^{n}r_{i} \underline{P(3,i-6)(ax^{2}  + xax + x^{2}a)}  + \sum_{i=7}^{n-1}r_{i} \underline{P(4,i-7)x^{3}} + \underline{Q(4,n-7)x^{3}}  ) + r_{4}\underline{a^{n}}.
$$
Premultiplying this by $x^{3}$ and using Lemmas 2.6 and 2.7 to separate reducible and irreducible words, yields
\begin{equation}
\label{3t16}
-x^{3}(a^{4}x^{n-4}) \rightarrow ( \sum_{i=4}^{n}r_{i}\underline{x^{3}P(1,i-4)a^{3}} + \sum_{i=5}^{n}r_{i} \underline{x^{3}P(2,i-5)(a^{2}x + axa + xa^{2})} + 
\end{equation}
$$
\sum_{i=6}^{n}r_{i} \underline{x^{3}P(3,i-6)(ax^{2}  + xax + x^{2}a)}  + \sum_{i=7}^{n-1}r_{i} \underline{x^{3}P(4,i-7)x^{3}} + \underline{x^{3}Q(4,n-7)x^{3}}  ) - r_{4}\underline{x^{3}a^{n}}.
$$
We get the following reducible words of length $n+3$ from (\ref{3t10})-(\ref{3t16}):
\begin{enumerate}
\item $xa^{3}x^{n-2}a \; \in \; xa^{2}P(1,n-4)x^{2}a$,
\item $xa^{3}x^{n-3}ax \; \in \; xa^{2}P(1,n-4)xax$,
\item $axa^{2}x^{n-2}a \; \in \; axaP(1,n-4)x^{2}a$,
\item $ a^{2}x^{n-1}a^{2} \; \in \; a^{2}xa^{2}x^{n-2} $,
\item $ a^{2}x^{n-2}a^{2}x \; \in \; a^{2}xa^{2}x^{n-2} $,
\item $ a^{2}x^{n-2}axa \; \in \; a^{2}xa^{2}x^{n-2} $,
\item $ x^{2}a^{3}x^{n-3}a \; \in \; x^{2}a P(2,n-5)x^{2}a $,
\item $ ax^{n-1}a^{3} \; \in \; ax^{2}a^{3}x^{n-3} $,
\item $ xa^{2}x^{n-2}a^{2} \; \in \; xa^{2}a^{2}x^{n-2} $.
\end{enumerate}
Recall from (\ref{3t8}) that 
$$
xa^{3}x^{n-3}xa \rightarrow - ( \sum_{i=3}^{n}r_{i} \underline{xa^{2}P(0,i-3)axa} + \sum_{i=4}^{n-1}r_{i} \underline{xa ^{2}P(1,i-4)x^{2}a} + \underline{xa^{2}Q(1,n-4)x^{2}a}   +
$$
$$
 \sum_{i=4}^{n}r_{i} ( \underline{xax P(1,i-4)axa}+ \underline{x^{2}a P(1,i-4)axa} ) +  \sum_{i=5}^{n}r_{i} ( \underline{x^{3} P(2,i-5)axa} + x^{2}a P(2,i-5)x^{2}a 
$$
$$
+ \underline{xax P(2,i-5)x^{2}a} ) +  \sum_{i=6}^{n}r_{i}\underline{x^{3}P(3,i-6)x^{2}a}) + r_{3}\underline{x^{2}a^{n+1}}.
$$
All the words in (\ref{3t8}) except $ -x^{2}a^{3}x^{n-3}a \; \in  -x^{2}a P(2,n-5)x^{2}a $. Thus this term cancels out with the (vii) term in the list above beacuse they appear with opposite signs. 

Similarly, recall from (\ref{3t7}) that 
$$
xa^{3}x^{n-3}ax \rightarrow - ( \sum_{i=3}^{n}r_{i} \underline{xa^{2}P(0,i-3)a^{2}x} + \sum_{i=4}^{n-1}r_{i} \underline{xa ^{2}P(1,i-4)xax} + \underline{xa^{2}Q(1,n-4)xax}   + 
$$
$$
 \sum_{i=4}^{n}r_{i} ( \underline{xax P(1,i-4)a^{2}x} + \underline{x^{2}a P(1,i-4)a^{2}x} ) +  \sum_{i=5}^{n}r_{i} ( \underline{x^{3} P(2,i-5)a^{2}x} + \underline{x^{2}a P(2,i-5)xax}  
$$
$$
+ \underline{xax P(2,i-5)xax} ) +  \sum_{i=6}^{n}r_{i}\underline{x^{3}P(3,i-6)xax}) + r_{3}\underline{xaxa^{n}}.
$$

Furthermore, use (\ref{e}) to expand $\sigma_{2}$ as
$$
a^{2}x^{n-2}\rightarrow  - ( \sum_{i=3}^{n-1}r_{i}\underline{a^{2}xP(0,i-3)} + \sum_{i=3}^{n-1}r_{i}\underline{(xa^{2}+ axa) P(0,i-3)}  + \underline{(xa^{2}+ axa) P(0,n-3)} +
$$
$$
 \sum_{i=4}^{n-1}r_{i} \underline{(ax^{2} + xax + x^{2}a)P(1,i-4)} + \underline{(ax^{2} + xax + x^{2}a)P(1,n-4)} + 
$$
$$
\sum_{i=5}^{n-1}r_{i} \underline{x^{3}P(2,i-5)}  + \underline{x^{3}P(2,n-5)} + r_{2}\underline{a^{2}} ) + r_{2}\underline{a^{n}}
$$
$$
=  - ( \sum_{i=3}^{n-1}r_{i}\underline{a^{2}xP(0,i-3)} + \sum_{i=3}^{n}r_{i}\underline{(xa^{2}+ axa) P(0,i-3)}  + \sum_{i=4}^{n}r_{i} \underline{(ax^{2} + xax + x^{2}a)P(1,i-4)} 
$$
$$
+ \sum_{i=5}^{n}r_{i} \underline{x^{3}P(2,i-5)}  + r_{2}\underline{a^{2}} ) + r_{2}\underline{a^{n}}.
$$
Post multiply this by $a^{2}x$, $xa^{2}$ and $axa$ and use Lemmas 2.6 and 2.7 to separate reducible and irreducible words to get
\begin{equation}
\label{3t17}
a^{2}x^{n-2}a^{2}x  \rightarrow  - ( \sum_{i=3}^{n-1}r_{i}\underline{a^{2}xP(0,i-3)a^{2}x} + \sum_{i=3}^{n}r_{i}\underline{(xa^{2}+ axa) P(0,i-3)a^{2}x}  + 
\end{equation}
$$
\sum_{i=4}^{n}r_{i} \underline{(ax^{2} + xax + x^{2}a)P(1,i-4)a^{2}x}  + \sum_{i=5}^{n}r_{i} \underline{x^{3}P(2,i-5)a^{2}x}   + r_{2}\underline{a^{4}x} ) + r_{2}\underline{a^{2}xa^{n}},
$$

\begin{equation}
\label{3t18}
a^{2}x^{n-2}xa^{2} \rightarrow  - ( \sum_{i=3}^{n-1}r_{i}\underline{a^{2}xP(0,i-3)xa^{2}} + \sum_{i=3}^{n}r_{i}(xa^{2}+ axa) P(0,i-3)xa^{2}  + 
\end{equation}
$$
\sum_{i=4}^{n}r_{i} \underline{(ax^{2} + xax + x^{2}a)P(1,i-4)xa^{2}} +  \underline{x^{3}P(2,n-5)xa^{2}} + xa^{2}x^{n-2}a^{2} + \underline{axax^{n-2}a^{2}})
$$
and
\begin{equation}
\label{3t19}
a^{2}x^{n-2}axa \rightarrow  - ( \sum_{i=3}^{n-1}r_{i}\underline{a^{2}xP(0,i-3)axa} + \sum_{i=3}^{n}r_{i}\underline{(xa^{2}+ axa) P(0,i-3)axa}  + 
\end{equation}
$$
\sum_{i=4}^{n}r_{i} \underline{(ax^{2} + xax + x^{2}a)P(1,i-4)axa} +  \sum_{i=5}^{n}r_{i} \underline{x^{3}P(2,i-5)axa}  + r_{2}\underline{a^{3}xa} ) + r_{2}\underline{axa^{n+1}}
$$
respectively.
The reducible word  $(- xa^{2}x^{n-2}a^{2})$ in (\ref{3t18}) cancels out with $xa^{2}x^{n-2}a^{2}$ which is the (ix) term in the list of reducible words above.

Recall from (\ref{3t5}) that
$$
axa^{2}x^{n-2}a  \rightarrow  -  ( \sum_{i=3}^{n}r_{i} ( \underline{axaP(0,i-3)axa} + \underline{axaP(0,i-3)xa^{2}} + ax^{2}P(0,i-3)a^{3} ) + 
$$
$$
 \sum_{i=4}^{n-1}r_{i} \underline{axa P(1,i-4)x^{2}a} + \underline{axaQ(1,n-4)x^{2}a} + \sum_{i=4}^{n}r_{i} ( \underline{ax^{2}P(1,i-4)axa} + 
$$
$$
\underline{ax^{2}P(1,i-4)xa^{2}} ) + \sum_{i=5}^{n}r_{i}\underline{ax^{2}P(2,i-5)x^{2}a} +    r_{2}\underline{axa^{3}} ) + r_{2}\underline{axa^{n+1}}.
$$
All the words in (\ref{3t5}) are irreducible except $-ax^{2}P(0,i-3)a^{3}$ which can be written as $-ax^{n-1}a^{3}$. But this cancels out with $ax^{n-1}a^{3}$ which is the term (h) in the list of reducible words above.

Thus, the reduction process ends here and when we substitute (\ref{3t5}), (\ref{3t7}), (\ref{3t8}), (\ref{3t10}), (\ref{3t11}), (\ref{3t12}), (\ref{3t13}), (\ref{3t14}), (\ref{3t15}), (\ref{3t16}), (\ref{3t17}), (\ref{3t18}) and (\ref{3t19}) into (\ref{3t9}), we get
\begin{equation}
\label{3t''}
\omega_{4}x^{3} \rightarrow  ( r_{2}\underline{xa^{4}} + r_{2}\underline{axa^{n+1}} + \sum_{i=3}^{n-1}r_{i} \underline{( x^{2}a + xax )P(0,i-3)a^{3}} + \sum_{i=4}^{n}r_{i}( \underline{x^{3}P(1,i-4)a^{3}} + 
\end{equation}
$$
\underline{a^{2}xP(1,i-4)(ax^{2} + x^{2}a + xax )} + \underline{axaP(1,i-4)(ax^{2} + xax )} + \underline{xa^{2}P(1,i-4)ax^{2}})  + 
$$
$$
 \sum_{i=5}^{n}r_{i} (  \underline{ ( ax^{2}+ x^{2}a + xax )P(2,i-5)ax^{2}} +  \underline{ax^{2}P(2,i-5)xax}) + \sum_{i=6}^{n}r_{i} \underline{x^{3}P(3,i-6)ax^{2}}
$$
$$
 - ( r_{2}\underline{a^{3}P(1,1)} + r_{3}\underline{ax^{2}a^{n}}+  \sum_{i=3}^{n}r_{i} (  \underline{xa^{2}P(0,i-3)(a^{2}x + axa)} +  \underline{axaP(0,i-3)(axa + xa^{2})}) +  
$$
$$
\sum_{i=4}^{n}r_{i} (  \underline{(xax + ax^{2} + x^{2}a)P(1,i-4)axa} + \underline{(xax + x^{2}a )P(1,i-4)a^{2}x} + \underline{ax^{2}P(1,i-4)xa^{2}} )  
$$
$$
\sum_{i=4}^{n-1}r_{i}a^{3}P(1,i-4)x^{3}  + \underline{a^{3}Q(1,n-4)x^{3}}    + \sum_{i=5}^{n}r_{i}\underline{x^{3}P(2,i-5)(a^{2}x + axa)}.
$$

Comparing (\ref{3t'}) and (\ref{3t''}), we conclude that the overlap ambiguity $\{ \omega_{1}, \omega_{4} \}$ is resolvable.

\noindent (b) Let $j=2$ and consider the overlap ambiguity $\{ \omega_{2}, \omega_{5} \}$. Use (\ref{d}) to expand $\sigma_{2}$ as
$$
 \omega_{2} \rightarrow - ( \sum_{i=3}^{n-1}r_{i}\underline{P(0,i-3)(a^{2}x + axa + xa^{2})} + \underline{P(0,n-3)(a^{2}x + axa + xa^{2})}
$$
$$
 \sum_{i=4}^{n-1}r_{i} \underline{P(1,i-4) (ax^{2} + xax + x^{2}a)}  +  \underline{P(1,n-4) (ax^{2} + xax + x^{2}a)}  +  
$$
$$
\sum_{i=5}^{n-1}r_{i}\underline{P(2,i-5)x^{3}} + \underline{Q(2,n-5)x^{3}} + r_{2}\underline{a^{2}} ) + r_{2}\underline{a^{n}}
$$
$$
 = - ( \sum_{i=3}^{n}r_{i}\underline{P(0,i-3)(a^{2}x + axa + xa^{2})} +  \sum_{i=4}^{n}r_{i} \underline{P(1,i-4) (ax^{2} + xax + x^{2}a)}  +   
$$
$$
\sum_{i=5}^{n-1}r_{i}\underline{P(2,i-5)x^{3}} + \underline{Q(2,n-5)x^{3}} + r_{2}\underline{a^{2}} ) + r_{2}\underline{a^{n}}.
$$
Premultiplying this by $a^{3}$ and using Lemmas 2.6 and 2.7 to separate reducible and irreducible words, yields
\begin{equation}
\label{3j1}
a^{3} \omega_{2}  \rightarrow - ( \sum_{i=3}^{n}r_{i}a^{3}P(0,i-3)(a^{2}x + axa + xa^{2}) +  \sum_{i=4}^{n}r_{i} a^{3}P(1,i-4) (ax^{2} + xax + x^{2}a)  +   
\end{equation}
$$
 \sum_{i=5}^{n-1}r_{i}a^{3}P(2,i-5)x^{3} + \underline{a^{3}Q(2,n-5)x^{3}} + r_{2}\underline{a^{5}} ) + r_{2}\underline{a^{n+3}}.
$$
The following words of length $n+3$ in (\ref{3j1}) are reducible:
\begin{enumerate}
\item $(a^{4}x^{n-4})x^{2}a \in a^{3}P(1,n-4)x^{2}a$
\item $(a^{4}x^{n-4})xax \in a^{3}P(1,n-4)xax$
\item $(a^{4}x^{n-4})ax^{2} \in a^{3}P(1,n-4)ax^{2}$
\item $a^{3}x^{n-3}axa $
\item $a^{3}x^{n-3}a^{2}x$
\item $a^{3}x^{n-2}a^{2}$.
\end{enumerate}
Using (\ref{e}), we expand $\sigma_{4}$ as
$$
 a^{4}x^{n-4}  \rightarrow - ( \sum_{i=4}^{n-1}r_{i} \underline{a^{3}P(1,i-4)} +  \underline{a^{3}Q(1,n-4)} + \sum_{i=5}^{n-1}r_{i} \underline{(a^{2}x + axa + xa^{2})P(2,i-5)}  +     
$$
$$
\underline{(a^{2}x + axa + xa^{2})P(2,n-5)}  +   \sum_{i=6}^{n-1}r_{i}  \underline{(ax^{2} + xax + x^{2}a)P(3,i-6)} + \underline{(ax^{2} + xax + x^{2}a)P(3,n-6)} 
$$
$$
+ \sum_{i=7}^{n-1}r_{i} \underline{x^{3}P(4,i-7)} +  \underline{x^{3}P(4,n-7)} ) + r_{4}\underline{a^{n}}
$$
$$
= - ( \sum_{i=4}^{n-1}r_{i} \underline{a^{3}P(1,i-4)} +  \underline{a^{3}Q(1,n-4)} + \sum_{i=5}^{n}r_{i} \underline{(a^{2}x + axa + xa^{2})P(2,i-5)}  +     
$$
$$
   \sum_{i=6}^{n}r_{i}  \underline{(ax^{2} + xax + x^{2}a)P(3,i-6)} +  \sum_{i=7}^{n-1}r_{i} \underline{x^{3}P(4,i-7)} ) + r_{4}\underline{a^{n}}.
$$
Thus, post multiplying this by $x^{2}a$, $xax$ and $ax^{2}$, and using Lemmas 2.6 and 2.7 to separate reducible and irreducible words, yield
\begin{equation}
\label{3j2}
- (a^{4}x^{n-4})x^{2}a  \rightarrow  ( \sum_{i=4}^{n-1}r_{i} a^{3}P(1,i-4)x^{2}a + \sum_{i=5}^{n}r_{i} (a^{2}x + axa + xa^{2})P(2,i-5)x^{2}a  +     
\end{equation}
$$
  \sum_{i=6}^{n}r_{i}  \underline{(ax^{2} + xax + x^{2}a)P(3,i-6)x^{2}a} +  \sum_{i=7}^{n}r_{i} \underline{x^{3}P(4,i-7)x^{2}a}  +  \underline{a^{3}Q(1,n-4)x^{2}a}  ) - r_{4}\underline{x^{2}a^{n+1}},
$$

\begin{equation}
\label{3j3}
- (a^{4}x^{n-4})xax  \rightarrow ( \sum_{i=4}^{n-1}r_{i} \underline{a^{3}P(1,i-4)xax} + \sum_{i=5}^{n}r_{i} (a^{2}x + axa + xa^{2})P(2,i-5)xax  + 
\end{equation}
$$
  \sum_{i=6}^{n}r_{i}  \underline{(ax^{2} + xax + x^{2}a)P(3,i-6)xax} +  \sum_{i=7}^{n}r_{i} \underline{x^{3}P(4,i-7)xax}  +  \underline{a^{3}Q(1,n-4)xax}  ) - r_{4}\underline{xaxa^{n}},
$$
and
\begin{equation}
\label{3j4}
- (a^{4}x^{n-4})ax^{2}  \rightarrow ( \sum_{i=4}^{n-1}r_{i} \underline{a^{3}P(1,i-4)ax^{2}} + \sum_{i=5}^{n}r_{i} \underline{(a^{2}x + axa + xa^{2})P(2,i-5)ax^{2}}  +    
\end{equation}
$$
  \sum_{i=6}^{n}r_{i}  \underline{(ax^{2} + xax + x^{2}a)P(3,i-6)ax^{2}} +  \sum_{i=7}^{n}r_{i} \underline{x^{3}P(4,i-7)ax^{2}}  +  \underline{a^{3}Q(1,n-4)ax^{2}}  ) - r_{4}\underline{ax^{2}a^{n}}
$$
respectively.
Similarly, use  (\ref{e}) to expand $\sigma_{3}$ as
$$
 a^{3}x^{n-3}  \rightarrow - ( \sum_{i=3}^{n-1}r_{i} \underline{a^{3}P(0,i-3)}  + \sum_{i=4}^{n-1}r_{i}\underline{(a^{2}x + axa + xa^{2})P(1,i-4)} +\underline{(a^{2}x + axa + xa^{2})P(1,n-4)} 
$$
$$
+ \sum_{i=5}^{n-1}r_{i}(\underline{(ax^{2} + xax + x^{2}a)P(2,i-5)} ) + \underline{(ax^{2} + xax + x^{2}a)P(2,n-5)} + \sum_{i=6}^{n}r_{i}  \underline{x^{3}P(3,i-6)} )    + r_{3}\underline{a^{n}}
$$
$$
= - ( \sum_{i=3}^{n-1}r_{i} \underline{a^{3}P(0,i-3)}  + \sum_{i=4}^{n}r_{i}\underline{(a^{2}x + axa + xa^{2})P(1,i-4)} 
$$
$$
+ \sum_{i=5}^{n}r_{i}(\underline{(ax^{2} + xax + x^{2}a)P(2,i-5)} ) + \sum_{i=6}^{n}r_{i}  \underline{x^{3}P(3,i-6)} )    + r_{3}\underline{a^{n}}.
$$
Post multiply this by $axa$, $a^{2}x$ and $xa^{2}$,  and use Lemmas 2.6 and 2.7 to separate reducible and irreducible words to get
\begin{equation}
\label{3j5}
-a^{3}x^{n-3}axa  \rightarrow (\sum_{i=3}^{n-1}r_{i} \underline{a^{3}P(0,i-3)axa}  + \sum_{i=4}^{n}r_{i}\underline{(a^{2}x + axa + xa^{2})P(1,i-4)axa} 
\end{equation}
$$
 + \sum_{i=5}^{n}r_{i}(\underline{(ax^{2} + xax + x^{2}a)P(2,i-5)axa} ) + \sum_{i=6}^{n}r_{i}  \underline{x^{3}P(3,i-6)axa} )  -  r_{3}\underline{axa^{n+1}},
$$

\begin{equation}
\label{3j6}
-a^{3}x^{n-3}a^{2}x  \rightarrow (\sum_{i=3}^{n-1}r_{i} \underline{a^{3}P(0,i-3)a^{2}x}  + \sum_{i=4}^{n}r_{i}\underline{(a^{2}x + axa + xa^{2})P(1,i-4)a^{2}x} 
\end{equation}
$$
+ \sum_{i=5}^{n}r_{i}(\underline{(ax^{2} + xax + x^{2}a)P(2,i-5)a^{2}x} ) + \sum_{i=6}^{n}r_{i}  \underline{x^{3}P(3,i-6)a^{2}x} ) - r_{3}\underline{a^{2}xa^{n}}
$$
and
\begin{equation}
\label{3j7}
- a^{3}x^{n-3}xa^{2}  \rightarrow (\sum_{i=3}^{n-1}r_{i} \underline{a^{3}P(0,i-3)xa^{2}}  + \sum_{i=4}^{n}r_{i}(a^{2}x + axa + xa^{2})P(1,i-4)xa^{2}
\end{equation}
$$
+ \sum_{i=5}^{n}r_{i}(\underline{(ax^{2} + xax + x^{2}a)P(2,i-5)xa^{2}} ) + \sum_{i=6}^{n}r_{i}  \underline{x^{3}P(3,i-6)xa^{2}} ) -  r_{3}\underline{xa^{n+2}}.
$$
The following words of length $n+3$ from (\ref{3j2}), (\ref{3j3}), (\ref{3j4}), (\ref{3j5}), (\ref{3j6}) and (\ref{3j7}) are reducible:
\begin{enumerate}
\item $x^{2}a^{4}x^{n-4}a \in x^{2}aP(3,n-6)x^{2}a$
\item $x(a^{4}x^{n-4})xa \in xa^{2}P(2,n-5)x^{2}a $
\item $axa^{3}x^{n-3}a \in axaP(2,n-5)x^{2}a$
\item $xa^{4}x^{n-4}ax \in xa^{2}P(2,n-5)xax $
\item $xa^{3}x^{n-3}a^{2} \in xa^{2}P(1,n-4)x a^{2}$.
\end{enumerate}
Use (\ref{f}) to expand $\sigma_{4}$ as
$$
 a^{4}x^{n-4}   \rightarrow - ( \sum_{i=4}^{n-1}r_{i} \underline{a^{2}P(1,i-4)a} + \underline{a^{2}P(1,n-4)a} + \sum_{i=5}^{n-1}r_{i} \underline{a^{2}P(2,i-5)x}  + \underline{a^{2}Q(2,n-5)x} +
$$
$$
 \sum_{i=5}^{n-1}r_{i} ( \underline{axP(2,i-5)a} + \underline{xaP(2,i-5)a} ) + \underline{axP(2,n-5)a} + \underline{xaP(2,n-5)a} +
$$
$$
 \sum_{i=6}^{n-1}r_{i} ( \underline{x^{2}P(3,i-6)a} + \underline{axP(3,i-6)x} + \underline{xaP(3,i-6)x}  )  + \underline{x^{2}P(3,n-6)a} + 
$$
$$
\underline{axP(3,n-6)x} + \underline{xaP(3,n-6)x} + \sum_{i=7}^{n-1}r_{i} \underline{x^{2}P(4,i-7)x} +  \underline{x^{2}P(4,n-7)x}  ) + r_{4}\underline{a^{n}}
$$
$$
 = - ( \sum_{i=4}^{n}r_{i} \underline{a^{2}P(1,i-4)a} + \sum_{i=5}^{n-1}r_{i} \underline{a^{2}P(2,i-5)x}+ \sum_{i=5}^{n}r_{i} ( \underline{axP(2,i-5)a} + \underline{xaP(2,i-5)a} ) 
$$
$$
 + \sum_{i=6}^{n}r_{i} ( \underline{x^{2}P(3,i-6)a} + \underline{axP(3,i-6)x} + \underline{xaP(3,i-6)x}  )  +  \sum_{i=7}^{n}r_{i} \underline{x^{2}P(4,i-7)x} 
$$
$$
+ \underline{a^{2}Q(2,n-5)x}   ) + r_{4}\underline{a^{n}}.
$$
Thus, premultiplying this by $x$ and post multiplying this by $xa$ and $ax$, and using Lemmas 2.6 and 2.7 to separate reducible and irreducible words, yield
\begin{equation}
\label{3j8}
xa^{4}x^{n-4}xa  \rightarrow - ( \sum_{i=4}^{n}r_{i} \underline{xa^{2}P(1,i-4)axa} + \sum_{i=5}^{n-1}r_{i} \underline{xa^{2}P(2,i-5)x^{2}a}+ \underline{xa^{2}Q(2,n-5)x^{2}a} 
\end{equation}
$$
+ \sum_{i=5}^{n}r_{i} ( \underline{xaxP(2,i-5)axa} + \underline{x^{2}aP(2,i-5)axa} )  + \sum_{i=6}^{n}r_{i} ( \underline{x^{3}P(3,i-6)axa} + 
$$
$$
\underline{xaxP(3,i-6)x^{2}a} + x^{2}aP(3,i-6)x^{2}a  )  +  \sum_{i=7}^{n}r_{i} \underline{x^{3}P(4,i-7)x^{2}a}   ) + r_{4}\underline{x^{2}a^{n+1}}
$$
and
\begin{equation}
\label{3j9}
xa^{4}x^{n-4}ax  \rightarrow - ( \sum_{i=4}^{n}r_{i} \underline{xa^{2}P(1,i-4)a^{2}x} + \sum_{i=5}^{n-1}r_{i} \underline{xa^{2}P(2,i-5)xax}+ \underline{xa^{2}Q(2,n-5)xax} 
\end{equation}
$$
+ \sum_{i=5}^{n}r_{i} ( \underline{xaxP(2,i-5)a^{2}x} + \underline{x^{2}aP(2,i-5)a^{2}x} )  + \sum_{i=6}^{n}r_{i} ( \underline{x^{3}P(3,i-6)a^{2}x} + 
$$
$$
\underline{xaxP(3,i-6)xax} + \underline{x^{2}aP(3,i-6)xax}  )  +  \sum_{i=7}^{n}r_{i} \underline{x^{3}P(4,i-7)xax} ) + r_{4}\underline{xaxa^{n}}
$$
respectively.
The reducible word $-x^{2}a^{4}x^{n-4}a$ in (\ref{3j8}) has an opposite sign with $x^{2}a^{4}x^{n-4}a$, the term (i) in the list above so they cancel out.

Again, use (\ref{g}) to expand $\sigma_{3}$ as
$$
a^{3}x^{n-3}  \rightarrow - ( \sum_{i=3}^{n-1}r_{i} \underline{aP(0,i-3)a^{2}} + \underline{aP(0,n-3)a^{2}}+  \sum_{i=4}^{n-1}r_{i} ( \underline{aP(1,i-4)ax} + \underline{aP(1,i-4)xa} $$
$$
+ \underline{xP(1,i-4)a^{2}}) +  \underline{aP(1,n-4)ax} + \underline{aP(1,n-4)xa}+ \underline{xP(1,n-4)a^{2}} +
$$
$$
 \sum_{i=5}^{n-1}r_{i}\underline{aP(2,i-5)x^{2}} + \underline{aQ(2,n-5)x^{2}} +  \sum_{i=5}^{n-1}r_{i} ( \underline{xP(2,i-5)ax} +  \underline{xP(2,i-5)xa}  )  +
$$
$$
\underline{xP(2,n-5)ax} +  \underline{xP(2,n-5)xa} +  \sum_{i=6}^{n-1}r_{i} \underline{xP(3,i-6)x^{2}} + \underline{xP(3,n-6)x^{2}} ) + r_{3}\underline{a^{n}}
$$
$$
= - ( \sum_{i=3}^{n}r_{i} \underline{aP(0,i-3)a^{2}} + \sum_{i=4}^{n}r_{i} ( \underline{aP(1,i-4)ax} + \underline{aP(1,i-4)xa} + \underline{xP(1,i-4)a^{2}}) +  
$$
$$
 \sum_{i=5}^{n-1}r_{i}\underline{aP(2,i-5)x^{2}} +  \sum_{i=5}^{n}r_{i} ( \underline{xP(2,i-5)ax} +  \underline{xP(2,i-5)xa}  )  +  \sum_{i=6}^{n}r_{i} \underline{xP(3,i-6)x^{2}} 
$$
$$
+ \underline{aQ(2,n-5)x^{2}}   ) + r_{3}\underline{a^{n}}.
$$
Pre and post multiplying this by $ax$ and $a$ respectively, and using Lemmas 2.6 and 2.7 to separate reducible and irreducible words, yields
\begin{equation}
\label{3j10}
axa^{3}x^{n-3}a  \rightarrow - (  \sum_{i=3}^{n}r_{i} \underline{axaP(0,i-3)a^{3}} + \sum_{i=4}^{n}r_{i} ( \underline{axaP(1,i-4)axa} + \underline{axaP(1,i-4)xa^{2}} 
\end{equation}
$$
+ \underline{ax^{2}P(1,i-4)a^{3}}) + \sum_{i=5}^{n-1}r_{i}\underline{axaP(2,i-5)x^{2}a} +  \sum_{i=5}^{n}r_{i} ( \underline{ax^{2}P(2,i-5)axa} +  
$$
$$
 \underline{ax^{2}P(2,i-5)xa^{2}}  )  +  \sum_{i=6}^{n}r_{i} \underline{ax^{2}P(3,i-6)x^{2}a} + \underline{axaQ(2,n-5)x^{2}a}   ) + r_{3}\underline{axa^{n+1}} .
$$
Turning now to the reduction of $xa^{3}x^{n-3}a^{2}$, use (\ref{f}) to expand $\sigma_{3}$ as
$$
a^{3}x^{n-3}  \rightarrow - (  \sum_{i=3}^{n-1}r_{i} \underline{a^{2}P(0,i-3)a} + \underline{a^{2}P(0,n-3)a}  + \sum_{i=4}^{n-1}r_{i}\underline{a^{2}P(1,i-4)x} +  \underline{a^{2}Q(1,n-4)x} + 
$$
$$
 \sum_{i=4}^{n-1}r_{i} ( \underline{axP(1,i-4)a} + \underline{xaP(1,i-4)a} ) + \underline{axP(1,n-4)a} + \underline{xaP(1,n-4)a} 
$$
$$
+  \sum_{i=5}^{n-1}r_{i}  ( \underline{x^{2}P(2,i-5)a} + \underline{axP(2,i-5)x} +  \underline{xaP(2,i-5)x} )    +
$$
$$
\underline{x^{2}P(2,n-5)a} + \underline{axP(2,n-5)x} +  \underline{xaP(2,n-5)x} )+ r_{3}\underline{a^{n}}
$$
$$
= - (  \sum_{i=3}^{n}r_{i} \underline{a^{2}P(0,i-3)a} + \sum_{i=4}^{n-1}r_{i}\underline{a^{2}P(1,i-4)x} +  \sum_{i=4}^{n}r_{i} ( \underline{axP(1,i-4)a} + \underline{xaP(1,i-4)a} ) 
$$
$$
+  \sum_{i=5}^{n}r_{i}  ( \underline{x^{2}P(2,i-5)a} + \underline{axP(2,i-5)x} +  \underline{xaP(2,i-5)x} )    +  \underline{a^{2}Q(1,n-4)x}   ) + r_{3}\underline{a^{n}}
$$
Pre and post multiplying this by $x$ and $a^{2}$ respectively, and using Lemmas 2.6 and 2.7 to separate reducible and irreducible words, yields
\begin{equation}
\label{3j11}
xa^{3}x^{n-3}a^{2}  \rightarrow - (  \sum_{i=3}^{n}r_{i} \underline{xa^{2}P(0,i-3)a^{3}} + \sum_{i=4}^{n-1}r_{i}\underline{xa^{2}P(1,i-4)xa^{2}} +  \sum_{i=4}^{n}r_{i} ( \underline{xaxP(1,i-4)a^{3}} 
\end{equation}
$$
+ \underline{x^{2}aP(1,i-4)a^{3}} )  +  \sum_{i=5}^{n}r_{i}  ( \underline{x^{3}P(2,i-5)a^{3}} + \underline{xaxP(2,i-5)xa^{2}} +  \underline{x^{2}aP(2,i-5)xa^{2}}    
$$
$$
+  \underline{xa^{2}Q(1,n-4)xa^{2}}   ) + r_{3}\underline{xa^{n+2}} .
$$
Thus, the reduction process ends here and when we substitute (\ref{3j2}), (\ref{3j3}), (\ref{3j4}), (\ref{3j5}), (\ref{3j6}), (\ref{3j7}), (\ref{3j8}), (\ref{3j9}), (\ref{3j10}) and (\ref{3j11}) into (\ref{3j1}), we get
\begin{equation}
\label{3j12}
a^{3}\omega_{2}  \rightarrow ( r_{2} \underline{a^{n+3}} + \sum_{i=4}^{n}r_{i} ( \underline{a^{2}xP(1,i-4)( a^{2}x + axa + xa^{2}} ) + \underline{axa P(1,i-4) a^{2}x} )+  
\end{equation}
$$
\sum_{i=5}^{n}r_{i} (  \underline{a^{2}xP(2,i-5)(ax^{2} + xax + x^{2}a} )  +  \underline{axaP(2,i-5)(ax^{2}+xax})  + \underline{xa^{2}P(2,i-5)ax^{2}} 
$$
$$
+ \underline{ax^{2}P(2,i-5)a^{2}x} ) + \sum_{i=6}^{n}r_{i} ( \underline{ax^{2}P(3,i-6)xax} + \underline{(ax^{2} + xax + x^{2}a)P(3,i-6)ax^{2}} ) + 
$$
$$
\sum_{i=7}^{n}r_{i} \underline{x^{3}P(4,i-7)ax^{2}}   - (  r_{2}\underline{a^{5}} + r_{3}\underline{a^{2}xa^{n}} + r_{4}\underline{ax^{2}a^{n}} + \sum_{i=3}^{n}r_{i} (\underline{axa + xa^{2}) P(0,i-3)a^{3}} +  
$$
$$
\sum_{i=4}^{n}r_{i} \underline{(ax^{2} + xax + x^{2}a )P(1,i-4)a^{3}} +   \sum_{i=5}^{n}r_{i} \underline{x^{3}P(2,i-5)a^{3}} 
$$
$$
+  \sum_{i=5}^{n-1}r_{i}   a^{3}P(2,i-5)x^{3}   +  \underline{a^{3}Q(2,n-5)x^{3}} ) .
$$

Turning now to the reduction of $\omega_{5}x^{3}$, use (\ref{e}) to expand $\sigma_{5}$ as
$$
\omega_{5} \rightarrow - ( \sum_{i=5}^{n-1}r_{i} \underline{a^{3}P(2,i-5)} + \underline{a^{3}Q(2,n-5)}  +  \sum_{i=6}^{n-1}r_{i} \underline{(a^{2}x + axa + xa^{2} )P(3,i-6)}   +  
$$
$$
\underline{(a^{2}x + axa + xa^{2} )P(3,n-6)}   +  \sum_{i=7}^{n-1}r_{i} \underline{(ax^{2} + xax + xa^{2} )P(4,i-7)} +
$$
$$
\underline{(ax^{2} + xax + xa^{2} )P(4,n-7)} + \sum_{i=8}^{n-1}r_{i}\underline{x^{3}P(5,i-8)} + \underline{x^{3}P(5,n-8)}   ) + r_{5}\underline{a^{n}}
$$
$$
= - ( \sum_{i=5}^{n-1}r_{i} \underline{a^{3}P(2,i-5)} + \underline{a^{3}Q(2,n-5)}  +  \sum_{i=6}^{n}r_{i} \underline{(a^{2}x + axa + xa^{2} )P(3,i-6)}   +  
$$
$$
 \sum_{i=7}^{n}r_{i} \underline{(ax^{2} + xax + xa^{2} )P(4,i-7)} +  \sum_{i=8}^{n}r_{i}\underline{x^{3}P(5,i-8)}    ) + r_{5}\underline{a^{n}}.
$$
Post multiplying this by $x^{3}$ and using Lemmas 2.6 and 2.7 to separate reducible and irreducible words, yields
\begin{equation}
\label{3j13}
\omega_{5}x^{3} \rightarrow - ( \sum_{i=5}^{n-1}r_{i} a^{3}P(2,i-5)x^{3} + \underline{a^{3}Q(2,n-5)x^{3}}  +  \sum_{i=6}^{n}r_{i} (a^{2}x + axa + xa^{2} )P(3,i-6)x^{3}   +
\end{equation}
$$
\sum_{i=7}^{n}r_{i} (ax^{2} + xax + xa^{2} )P(4,i-7)x^{3} + \sum_{i=8}^{n}r_{i}x^{3}P(5,i-8)x^{3}   ) + r_{5}\underline{a^{n+3}}.
$$
The following words in (\ref{3j13}) of length $n+3$ are reducible:
\begin{enumerate}
\item $x^{3}a^{5}x^{n-5} \in x^{3}P(5,n-8)x^{3} $
\item $x^{2}a(a^{4}x^{n-4}) \in x^{2}aP(4,n-7)x^{3}$
\item $xaxa^{4}x^{n-4} \in  xaxP(4,n-7)x^{3} $
\item $ax^{2}a^{4}x^{n-4} \in ax^{2}P(4,n-7)x^{3}$
\item $xa^{2}(a^{3}x^{n-3}) \in xa^{2}P(3,n-6)x^{3}$
\item $a^{2}xa^{3}x^{n-3} \in a^{2}xP(3,n-6)x^{3}$
\item $ axa(a^{3}x^{n-3}) \in axaP(3,n-6)x^{3}$.
\end{enumerate}
In order to reduce (b), (c) and (d) in the list above, use (\ref{d}) to expand $\sigma_{4}$ as
$$
a^{4}x^{n-4} \rightarrow - (  \sum_{i=4}^{n-1}r_{i}\underline{P(1,i-4)a^{3}} + \underline{P(1,n-4)a^{3}}  +  \sum_{i=5}^{n-1}r_{i} \underline{P(2,i-5)(a^{2}x + axa + xa^{2} )}  
$$
$$
+ \underline{P(2,n-5)(a^{2}x + axa + xa^{2} )} + \sum_{i=6}^{n-1}r_{i} \underline{P(3,i-6)(ax^{2} + xax + x^{2}a)} + 
$$
$$
\underline{P(3,n-6)(ax^{2} + xax + x^{2}a)} + \sum_{i=7}^{n-1}r_{i}\underline{P(4,i-7)x^{3}} + \underline{Q(4,n-7)x^{3}} ) + r_{4}\underline{a^{n}}
$$
$$
= - (  \sum_{i=4}^{n}r_{i}\underline{P(1,i-4)a^{3}} +  \sum_{i=5}^{n}r_{i} \underline{P(2,i-5)(a^{2}x + axa + xa^{2} )} +  
$$
$$
 \sum_{i=6}^{n}r_{i} \underline{P(3,i-6)(ax^{2} + xax + x^{2}a)} + \sum_{i=7}^{n-1}r_{i}\underline{P(4,i-7)x^{3}} + \underline{Q(4,n-7)x^{3}} ) + r_{4}\underline{a^{n}}.
$$
Premultiply this by $x^{2}a$, $xax$ and $ax^{2}$, and using Lemmas 2.6 and 2.7 to separate reducible and irreducible words, yield
\begin{equation}
\label{3j14}
- x^{2}a(a^{4}x^{n-4}) \rightarrow (  \sum_{i=4}^{n}r_{i}\underline{x^{2}aP(1,i-4)a^{3}} +  \sum_{i=5}^{n}r_{i} \underline{x^{2}aP(2,i-5)(a^{2}x + axa + xa^{2} )} 
\end{equation}
$$
+ \sum_{i=6}^{n}r_{i} x^{2}aP(3,i-6)(ax^{2} + xax + x^{2}a) + \sum_{i=7}^{n-1}r_{i}\underline{x^{2}aP(4,i-7)x^{3}} +
$$
$$
 \underline{x^{2}aQ(4,n-7)x^{3}} ) - r_{4}\underline{x^{2}a^{n+1}}, 
$$

\begin{equation}
\label{3j15}
- xax(a^{4}x^{n-4}) \rightarrow ( \sum_{i=4}^{n}r_{i}\underline{xaxP(1,i-4)a^{3}} +  \sum_{i=5}^{n}r_{i} \underline{xaxP(2,i-5)(a^{2}x + axa + xa^{2} )}  
\end{equation}
$$
+  \sum_{i=6}^{n}r_{i} \underline{xaxP(3,i-6)(ax^{2} + xax + x^{2}a)} + \sum_{i=7}^{n-1}r_{i}\underline{xaxP(4,i-7)x^{3}} 
$$
$$
+ \underline{xaxQ(4,n-7)x^{3}} ) -  r_{4}\underline{xaxa^{n}},
$$
and
\begin{equation}
\label{3j16}
- ax^{2}(a^{4}x^{n-4}) \rightarrow ( \sum_{i=4}^{n}r_{i}\underline{ax^{2}P(1,i-4)a^{3}} +  \sum_{i=5}^{n}r_{i} \underline{ax^{2}P(2,i-5)(a^{2}x + axa + xa^{2} )}    
\end{equation}
$$
 + \sum_{i=6}^{n}r_{i} \underline{ax^{2}P(3,i-6)(ax^{2} + xax + x^{2}a)} + \sum_{i=7}^{n-1}r_{i}\underline{ax^{2}P(4,i-7)x^{3}} 
$$
$$
+ \underline{ax^{2}Q(4,n-7)x^{3}} ) -  r_{4}\underline{ax^{2}a^{n}}
$$
respectively. Similarly, use (\ref{d}) to expand $\sigma_{5}$ as
$$
a^{5}x^{n-5} \rightarrow  - ( \sum_{i=5}^{n-1}r_{i} \underline{P(2,i-5)a^{3}} + \underline{P(2,n-5)a^{3}} + \sum_{i=6}^{n-1}r_{i}  \underline{P(3,i-6)(a^{2}x + axa + xa^{2})}    
$$
$$
+ \underline{P(3,n-6)(a^{2}x + axa + xa^{2})}   + \sum_{i=7}^{n-1}r_{i}  \underline{P(4,i-7)(ax^{2} + xax + ax^{2})} + 
$$
$$
\underline{P(4,n-7)(ax^{2} + xax + ax^{2})} +  \sum_{i=8}^{n-1}r_{i}\underline{P(5,i-8)x^{3}} + \underline{Q(5,n-8)x^{3}}  ) + r_{5}\underline{a^{n}}
$$
$$
=  - ( \sum_{i=5}^{n}r_{i} \underline{P(2,i-5)a^{3}} +  \sum_{i=6}^{n}r_{i}  \underline{P(3,i-6)(a^{2}x + axa + xa^{2})} +    
$$
$$
\sum_{i=7}^{n}r_{i}  \underline{P(4,i-7)(ax^{2} + xax + ax^{2})} +  \sum_{i=8}^{n-1}r_{i}\underline{P(5,i-8)x^{3}} + \underline{Q(5,n-8)x^{3}}  ) + r_{5}\underline{a^{n}}.
$$
Premultiply this by $x^{3}$ and use Lemmas 2.6 and 2.7 to separate reducible and irreducible words to get
\begin{equation}
\label{3j17}
- x^{3}a^{5}x^{n-5} \rightarrow  ( \sum_{i=5}^{n}r_{i} \underline{x^{3}P(2,i-5)a^{3}} +  \sum_{i=6}^{n}r_{i}  \underline{x^{3}P(3,i-6)(a^{2}x + axa + xa^{2})} +    
\end{equation}
$$
\sum_{i=7}^{n}r_{i}  \underline{x^{3}P(4,i-7)(ax^{2} + xax + ax^{2})} +  \sum_{i=8}^{n-1}r_{i}\underline{x^{3}P(5,i-8)x^{3}} 
$$
$$
+ \underline{x^{3}Q(5,n-8)x^{3}}  ) - r_{5}\underline{x^{3}a^{n}}.
$$
Similarly, expand $\sigma_{3}$ using (\ref{d}) to get
$$
a^{3}x^{n-3} \rightarrow - ( \sum_{i=3}^{n-1}r_{i}  \underline{P(0,i-3)a^{3}} + \underline{P(0,n-3)a^{3}} +  \sum_{i=4}^{n-1}r_{i}\underline{P(1,i-4)(a^{2}x + axa + xa^{2} )}   +  
$$
$$
\underline{P(1,n-4)(a^{2}x + axa + xa^{2} )}   +   \sum_{i=5}^{n-1}r_{i} \underline{P(2,i-5)(ax^{2} + xax + x^{2}a)} + \underline{P(2,n-5)(ax^{2} + xax + x^{2}a)} 
$$
$$
+ \sum_{i=6}^{n-1}r_{i} \underline{P(3,i-6)x^{3}} +  \underline{Q(3,n-6)x^{3}}  ) + r_{3}\underline{a^{n}}  
$$
$$
= - ( \sum_{i=3}^{n}r_{i}  \underline{P(0,i-3)a^{3}} + \sum_{i=4}^{n}r_{i}\underline{P(1,i-4)(a^{2}x + axa + xa^{2} )}   +  
$$
$$
\sum_{i=5}^{n}r_{i} \underline{P(2,i-5)(ax^{2} + xax + x^{2}a)} + \sum_{i=6}^{n-1}r_{i} \underline{P(3,i-6)x^{3}} +  \underline{Q(3,n-6)x^{3}}  ) + r_{3}\underline{a^{n}}  .
$$
Thus, premultiplying this by $xa^{2}$, $a^{2}x$ and $axa$, and using Lemmas 2.6 and 2.7 to separate reducible and irreducible words, yield
\begin{equation}
\label{3j18}
-xa^{2}(a^{3}x^{n-3}) \rightarrow (  \sum_{i=3}^{n}r_{i}  \underline{xa^{2}P(0,i-3)a^{3}} + \sum_{i=4}^{n}r_{i}\underline{xa^{2}P(1,i-4)(a^{2}x + axa + xa^{2} )}   +    
\end{equation}
$$
\sum_{i=5}^{n}r_{i} xa^{2}P(2,i-5)(ax^{2} + xax + x^{2}a) + \sum_{i=6}^{n-1}r_{i} \underline{xa^{2}P(3,i-6)x^{3}} 
$$
$$
+  \underline{xa^{2}Q(3,n-6)x^{3}}  ) - r_{3}\underline{xa^{n+2}} ,
$$

\begin{equation}
\label{3j19}
-a^{2}x(a^{3}x^{n-3}) \rightarrow (  \sum_{i=3}^{n}r_{i}  a^{2}xP(0,i-3)a^{3} + \sum_{i=4}^{n}r_{i}\underline{a^{2}xP(1,i-4)(a^{2}x + axa + xa^{2} )}   +  
\end{equation}
$$
\sum_{i=5}^{n}r_{i} \underline{a^{2}xP(2,i-5)(ax^{2} + xax + x^{2}a)} + \sum_{i=6}^{n-1}r_{i} \underline{a^{2}xP(3,i-6)x^{3}} 
$$
$$
+  \underline{a^{2}xQ(3,n-6)x^{3}}  ) - r_{3}\underline{a^{2}xa^{n}} 
$$
and
\begin{equation}
\label{3j20}
-axa(a^{3}x^{n-3}) \rightarrow (  \sum_{i=3}^{n}r_{i}  \underline{axaP(0,i-3)a^{3}} + \sum_{i=4}^{n}r_{i}\underline{axaP(1,i-4)(a^{2}x + axa + xa^{2} )}   + 
\end{equation}
$$
\sum_{i=5}^{n}r_{i} axaP(2,i-5)(ax^{2} + xax + x^{2}a) + \sum_{i=6}^{n-1}r_{i} \underline{axaP(3,i-6)x^{3}} 
$$
$$
+  \underline{axaQ(3,n-6)x^{3}}  ) -  r_{3}\underline{axa^{n+1}}  
$$
respectively.
The following words in (\ref{3j14})-(\ref{3j20}) of length $n+3$ are irreducible:
\begin{enumerate}
\item $x^{2}a^{4}x^{n-4}a \in x^{2}aP(3,n-6)x^{2}a$
\item $x(a^{4}x^{n-4})xa \in xa^{2}P(2,n-5)x^{2}a$
\item $xa^{4}x^{n-4}ax \in  xa^{2}P(2,n-5)xax$
\item $xa^{3}x^{n-3}a^{2} \in xa^{2}P(1,n-4)xa^{2}  $
\item $a^{2}x^{n-2}a^{3}$
\item $axa^{3}x^{n-3}a \in axaP(2,n-5)x^{2}a$
\end{enumerate}
Recall (\ref{3j8}), (\ref{3j9}), (\ref{3j10}) and (\ref{3j11}) as follows:
$$
xa^{4}x^{n-4}xa  \rightarrow - ( \sum_{i=4}^{n}r_{i} \underline{xa^{2}P(1,i-4)axa} + \sum_{i=5}^{n-1}r_{i} \underline{xa^{2}P(2,i-5)x^{2}a}+ \underline{xa^{2}Q(2,n-5)x^{2}a} 
$$
$$
+ \sum_{i=5}^{n}r_{i} ( \underline{xaxP(2,i-5)axa} + \underline{x^{2}aP(2,i-5)axa} )  + \sum_{i=6}^{n}r_{i} ( \underline{x^{3}P(3,i-6)axa} + 
$$
$$
\underline{xaxP(3,i-6)x^{2}a} + x^{2}aP(3,i-6)x^{2}a  )  +  \sum_{i=7}^{n}r_{i} \underline{x^{3}P(4,i-7)x^{2}a}   ) + r_{4}\underline{x^{2}a^{n+1}},
$$

$$
xa^{4}x^{n-4}ax  \rightarrow - ( \sum_{i=4}^{n}r_{i} \underline{xa^{2}P(1,i-4)a^{2}x} + \sum_{i=5}^{n-1}r_{i} \underline{xa^{2}P(2,i-5)xax}+ \underline{xa^{2}Q(2,n-5)xax} 
$$
$$
+ \sum_{i=5}^{n}r_{i} ( \underline{xaxP(2,i-5)a^{2}x} + \underline{x^{2}aP(2,i-5)a^{2}x} )  + \sum_{i=6}^{n}r_{i} ( \underline{x^{3}P(3,i-6)a^{2}x} + 
$$
$$
\underline{xaxP(3,i-6)xax} + \underline{x^{2}aP(3,i-6)xax}  )  +  \sum_{i=7}^{n}r_{i} \underline{x^{3}P(4,i-7)xax} ) + r_{4}\underline{xaxa^{n}},
$$

$$
axa^{3}x^{n-3}a  \rightarrow - (  \sum_{i=3}^{n}r_{i} \underline{axaP(0,i-3)a^{3}} + \sum_{i=4}^{n}r_{i} ( \underline{axaP(1,i-4)axa} + \underline{axaP(1,i-4)xa^{2}} 
$$
$$
+ \underline{ax^{2}P(1,i-4)a^{3}}) + \sum_{i=5}^{n-1}r_{i}\underline{axaP(2,i-5)x^{2}a} +  \sum_{i=5}^{n}r_{i} ( \underline{ax^{2}P(2,i-5)axa} +  
$$
$$
 \underline{ax^{2}P(2,i-5)xa^{2}}  )  +  \sum_{i=6}^{n}r_{i} \underline{ax^{2}P(3,i-6)x^{2}a} + \underline{axaQ(2,n-5)x^{2}a}   ) + r_{3}\underline{axa^{n+1}},
$$
and
$$
xa^{3}x^{n-3}a^{2}  \rightarrow - (  \sum_{i=3}^{n}r_{i} \underline{xa^{2}P(0,i-3)a^{3}} + \sum_{i=4}^{n-1}r_{i}\underline{xa^{2}P(1,i-4)xa^{2}} +  \sum_{i=4}^{n}r_{i} ( \underline{xaxP(1,i-4)a^{3}} 
$$
$$
+ \underline{x^{2}aP(1,i-4)a^{3}} )  +  \sum_{i=5}^{n}r_{i}  ( \underline{x^{3}P(2,i-5)a^{3}} + \underline{xaxP(2,i-5)xa^{2}} +  \underline{x^{2}aP(2,i-5)xa^{2}}    
$$
$$
+  \underline{xa^{2}Q(1,n-4)xa^{2}}   ) + r_{3}\underline{xa^{n+2}} .
$$

 In order to reduce $a^{2}x^{n-2}a^{3} $, use (\ref{e}) to expand $\sigma_{2}$ as
$$
a^{2}x^{n-2} \rightarrow - ( \sum_{i=3}^{n-1}r_{i} \underline{(axa + xa^{2})P(0,i-3)} + \underline{(axa + xa^{2})P(0,n-3)} + \sum_{i=3}^{n-1}r_{i} \underline{a^{2}x P(0,i-3)}    
$$
$$
+ \sum_{i=4}^{n-1}r_{i} \underline{(ax^{2} + xax + x^{2}a)P(1,i-4)} + \underline{(ax^{2} + xax + x^{2}a)P(1,n-4)} + 
$$
$$
\sum_{i=5}^{n-1}r_{i} \underline{x^{3}P(2,i-5)} + \underline{x^{3}P(2,n-5)} + r_{2}\underline{a^{2}} ) + r_{2}\underline{a^{n}}
$$
$$
= - ( \sum_{i=3}^{n}r_{i} \underline{(axa + xa^{2})P(0,i-3)} +  \sum_{i=3}^{n-1}r_{i} \underline{a^{2}x P(0,i-3)} +   
$$
$$
\sum_{i=4}^{n}r_{i} \underline{(ax^{2} + xax + x^{2}a)P(1,i-4)} +  \sum_{i=5}^{n}r_{i} \underline{x^{3}P(2,i-5)} + r_{2}\underline{a^{2}} ) + r_{2}\underline{a^{n}}.
$$
Post multiply this by $a^{3}$ and use Lemmas 2.6 and 2.7 to separate reducible and irreducible words to get
\begin{equation}
\label{3j21}
a^{2}x^{n-2}a^{3} \rightarrow - ( \sum_{i=3}^{n}r_{i} \underline{(axa + xa^{2})P(0,i-3)a^{3}} +  \sum_{i=3}^{n-1}r_{i} \underline{a^{2}x P(0,i-3)a^{3}} +  
\end{equation}
$$
\sum_{i=4}^{n}r_{i} \underline{(ax^{2} + xax + x^{2}a)P(1,i-4)a^{3}} +  \sum_{i=5}^{n}r_{i} \underline{x^{3}P(2,i-5)a^{3}} + r_{2}\underline{a^{5}} ) + r_{2}\underline{a^{n+3}}.
$$
The reducible word $-x^{2}a^{4}x^{n-4}a$ in (\ref{3j8}) has an opposite sign with $x^{2}a^{4}x^{n-4}a$, the term (i) in the list above so they cancel out. Hence, the reduction process ends here and when we substitute (\ref{3j14})-(\ref{3j21}) and (\ref{3j8})-(\ref{3j11}) into (\ref{3j3}), yielding
\begin{equation}
\label{3j22}
\omega_{5}x^{3}  \rightarrow  ( r_{2} \underline{a^{n+3}} + \sum_{i=4}^{n}r_{i} ( \underline{a^{2}xP(1,i-4)( a^{2}x + axa + xa^{2}} ) + \underline{axa P(1,i-4) a^{2}x} )+  
\end{equation}
$$
\sum_{i=5}^{n}r_{i} (  \underline{a^{2}xP(2,i-5)(ax^{2} + xax + x^{2}a} )  +  \underline{axaP(2,i-5)(ax^{2}+xax})  + \underline{xa^{2}P(2,i-5)ax^{2}} 
$$
$$
+ \underline{ax^{2}P(2,i-5)a^{2}x} ) + \sum_{i=6}^{n}r_{i} ( \underline{ax^{2}P(3,i-6)xax} + \underline{(ax^{2} + xax + x^{2}a)P(3,i-6)ax^{2}} ) + 
$$
$$
\sum_{i=7}^{n}r_{i} \underline{x^{3}P(4,i-7)ax^{2}}   - (  r_{2}\underline{a^{5}} + r_{3}\underline{a^{2}xa^{n}} + r_{4}\underline{ax^{2}a^{n}} + \sum_{i=3}^{n}r_{i} (\underline{axa + xa^{2}) P(0,i-3)a^{3}} +  
$$
$$
\sum_{i=4}^{n}r_{i} \underline{(ax^{2} + xax + x^{2}a )P(1,i-4)a^{3}} +   \sum_{i=5}^{n}r_{i} \underline{x^{3}P(2,i-5)a^{3}} 
$$
$$
+  \sum_{i=5}^{n-1}r_{i}   a^{3}P(2,i-5)x^{3}   +  \underline{a^{3}Q(2,n-5)x^{3}} ) .
$$
Comparing (\ref{3j12}) and (\ref{3j22}), we conclude that the overlap ambiguity $\{ \omega_{2}, \omega_{5} \}$ is resolvable.

\noindent (c) Let $3 \leq j < n-4$ and consider the overlap ambiguity $\{\omega_{j}, \omega_{j+3}\}$. Using (\ref{d}), expand $\sigma_{j}$ to get
$$
\omega_{j} \rightarrow - ( \sum_{i=j}^{n-1}r_{i} \underline{P(j-3,i-j)a^{3}} + \underline{P(j-3,n-j)a^{3}} + \sum_{i=j+1}^{n-1}r_{i} \underline{P(j-2,i-j-1)(a^{2}x + axa + xa^{2})}      
$$
$$
+ \underline{P(j-2,n-j-1)(a^{2}x + axa + xa^{2})}   +  \sum_{i=j+2}^{n-1}r_{i}\underline{P(j-1,i-j-2)(ax^{2} + xax + x^{2}a)} +
$$
$$
\underline{P(j-1,n-j-2)(ax^{2} + xax + x^{2}a)} + \sum_{i=j+3}^{n-1}r_{i} \underline{P(j,i-j-3)x^{3}} +  \underline{Q(j,n-j-3)x^{3}} ) + r_{j}\underline{a^{n}}
$$
$$
= - ( \sum_{i=j}^{n}r_{i} \underline{P(j-3,i-j)a^{3}} + \sum_{i=j+1}^{n}r_{i} \underline{P(j-2,i-j-1)(a^{2}x + axa + xa^{2})} +      
$$
$$
\sum_{i=j+2}^{n}r_{i}\underline{P(j-1,i-j-2)(ax^{2} + xax + x^{2}a)} + \sum_{i=j+3}^{n-1}r_{i} \underline{P(j,i-j-3)x^{3}} +  \underline{Q(j,n-j-3)x^{3}} ) + r_{j}\underline{a^{n}}.
$$
Premultiplying this by $a^{3}$ and using Lemmas 2.6 and 2.7 to separate reducible and irreducible words, yields
\begin{equation}
\label{j1}
a^{3}\omega_{j} \rightarrow - ( \sum_{i=j}^{n}r_{i} a^{3}P(j-3,i-j)a^{3} + \sum_{i=j+1}^{n}r_{i} a^{3}P(j-2,i-j-1)(a^{2}x + axa + xa^{2})    
\end{equation}
$$
+ \sum_{i=j+2}^{n}r_{i}a^{3}P(j-1,i-j-2)(ax^{2} + xax + x^{2}a) + \sum_{i=j+3}^{n-1}r_{i} a^{3}P(j,i-j-3)x^{3} +  
$$
$$
\underline{a^{3}Q(j,n-j-3)x^{3}} ) + r_{j}\underline{a^{n+3}}.
$$
The following words in (\ref{j1}) of length $n+3$ are reducible:
\begin{enumerate}
\item $a^{j}x^{n-j}a^{3} \in a^{3}P(j-3,n-j)a^{3}$
\item $(a^{j+2}x^{n-j-2}) x^{2} a \in  a^{3}P(j-1,n-j-2)x^{2}a$
\item $(a^{j+2}x^{n-j-2}) xax \in  a^{3}P(j-1,n-j-2)xax$
\item $(a^{j+2}x^{n-j-2}) ax^{2}  \in  a^{3}P(j-1,n-j-2)ax^{2}$
\item $(a^{j+1}x^{n-j-1})xa^{2} \in a^{3}P(j-2,n-j-1)xa^{2}$
\item $(a^{j+1}x^{n-j-1})axa \in a^{3}P(j-2,n-j-1)axa$
\item $(a^{j+1}x^{n-j-1})a^{2}x \in a^{3}P(j-2,n-j-1)a^{2}x$.
\end{enumerate}
Using (\ref{e}), expand $\sigma_{j}$ as follows:
$$
a^{j}x^{n-j} \rightarrow - (  \sum_{i=j}^{n-1}r_{i} \underline{a^{3}P(j-3,i-j)} + \underline{a^{3}Q(j-3,n-j)} + \sum_{i=j+1}^{n-1}r_{i} \underline{(a^{2}x + axa + xa^{2})P(j-2,i-j-1)}  
$$
$$
+ \underline{(a^{2}x + axa + xa^{2})P(j-2,n-j-1)} + \sum_{i=j+2}^{n-1}r_{i} \underline{(ax^{2} + xax + x^{2}a)P(j-1,i-j-2)} +
$$
$$
 \underline{(ax^{2} + xax + x^{2}a)P(j-1,n-j-2)} + \sum_{i=j+3}^{n-1}r_{i} \underline{x^{3}P(j,i-j-3)} + \underline{x^{3}P(j,n-j-3)}   )  + r_{j}\underline{a^{n}}
$$
$$
= - (  \sum_{i=j}^{n-1}r_{i} \underline{a^{3}P(j-3,i-j)} + \underline{a^{3}Q(j-3,n-j)} + \sum_{i=j+1}^{n}r_{i} \underline{(a^{2}x + axa + xa^{2})P(j-2,i-j-1)}  
$$
$$
+ \sum_{i=j+2}^{n}r_{i} \underline{(ax^{2} + xax + x^{2}a)P(j-1,i-j-2)} + \sum_{i=j+3}^{n}r_{i} \underline{x^{3}P(j,i-j-3)}   )  + r_{j}\underline{a^{n}}.
$$
Postmultiplying this by $a^{3}$ and using Lemmas 2.6 and 2.7 to separate reducible and irreducible words, yields
\begin{equation}
\label{j2}
-a^{j}x^{n-j}a^{3} \rightarrow ( \sum_{i=j}^{n-1}r_{i} \underline{a^{3}P(j-3,i-j)a^{3}} + \underline{a^{3}Q(j-3,n-j)a^{3}}  
\end{equation}
$$
+ \sum_{i=j+1}^{n}r_{i} \underline{(a^{2}x + axa + xa^{2})P(j-2,i-j-1)a^{3}}  +
$$
$$
\sum_{i=j+2}^{n}r_{i} \underline{(ax^{2} + xax + x^{2}a)P(j-1,i-j-2)a^{3}} + \sum_{i=j+3}^{n}r_{i} \underline{x^{3}P(j,i-j-3)a^{3}} )  - r_{j}\underline{a^{n+3}}.
$$
Using (\ref{e}), expand $\sigma_{j+2}$ as
$$
a^{j+2}x^{n-j-2} \rightarrow - ( \sum_{i=j+2}^{n-1}r_{i} \underline{a^{3}P(j-1,i-j-2)} + \underline{a^{3}Q(j-1,n-j-2)}  + 
$$
$$
\sum_{i=j+3}^{n-1}r_{i} \underline{(a^{2}x + axa + xa^{2} )P(j,i-j-3)} + \underline{(a^{2}x + axa + xa^{2} )P(j,n-j-3)} + 
$$
$$
 \sum_{i=j+4}^{n-1}r_{i}  \underline{(ax^{2} + xax + x^{2}a )P(j+1, i-j-4)} + \underline{(ax^{2} + xax + x^{2}a )P(j+1, n-j-4)} +
$$
$$
 + \sum_{i=j+5}^{n-1}r_{i} \underline{x^{3}P(j+2,i-j-5)} + \underline{x^{3}P(j+2,n-j-5)} ) + r_{j+2}\underline{a^{n}}
$$
$$
= - ( \sum_{i=j+2}^{n-1}r_{i} \underline{a^{3}P(j-1,i-j-2)} + \underline{a^{3}Q(j-1,n-j-2)}  + 
$$
$$
\sum_{i=j+3}^{n}r_{i} \underline{(a^{2}x + axa + xa^{2} )P(j,i-j-3)} +   \sum_{i=j+4}^{n}r_{i}  \underline{(ax^{2} + xax + x^{2}a )P(j+1, i-j-4)} 
$$
$$
 + \sum_{i=j+5}^{n}r_{i} \underline{x^{3}P(j+2,i-j-5)} ) + r_{j+2}\underline{a^{n}}
$$

Thus, post multiplying this by $x^{2}a$, $xax$ and $ax^{2}$, and using Lemmas 2.6 and 2.7 to separate reducible and irreducible words, yield
\begin{equation}
\label{j3}
-(a^{j+2}x^{n-j-2}) x^{2} a \rightarrow ( \sum_{i=j+2}^{n-1}r_{i} a^{3}P(j-1,i-j-2)x^{2} a + \underline{a^{3}Q(j-1,n-j-2)x^{2} a}  + 
\end{equation}
$$
\sum_{i=j+3}^{n}r_{i} (a^{2}x + axa + xa^{2} )P(j,i-j-3)x^{2} a +   \sum_{i=j+4}^{n}r_{i}  (ax^{2} + xax + x^{2}a )P(j+1, i-j-4)x^{2} a
$$
$$
 + \sum_{i=j+5}^{n}r_{i} \underline{x^{3}P(j+2,i-j-5)x^{2} a} ) -  r_{j+2}\underline{x^{2} a^{n+1}},
$$

\begin{equation}
\label{j4}
-(a^{j+2}x^{n-j-2}) xax \rightarrow ( \sum_{i=j+2}^{n-1}r_{i} \underline{a^{3}P(j-1,i-j-2)xax} + \underline{a^{3}Q(j-1,n-j-2)xax}  + 
\end{equation}
$$
\sum_{i=j+3}^{n}r_{i} (a^{2}x + axa + xa^{2} )P(j,i-j-3)xax +   \sum_{i=j+4}^{n}r_{i}  \underline{(ax^{2} + xax + x^{2}a )P(j+1, i-j-4)xax} 
$$
$$
 + \sum_{i=j+5}^{n}r_{i} \underline{x^{3}P(j+2,i-j-5)xax} ) - r_{j+2}\underline{xaxa^{n}}
$$
and
\begin{equation}
\label{j5}
-(a^{j+2}x^{n-j-2}) ax^{2}  \rightarrow ( \sum_{i=j+2}^{n-1}r_{i} \underline{a^{3}P(j-1,i-j-2)ax^{2} } + \underline{a^{3}Q(j-1,n-j-2)ax^{2} }  + 
\end{equation}
$$
\sum_{i=j+3}^{n}r_{i} \underline{(a^{2}x + axa + xa^{2} )P(j,i-j-3)ax^{2} } +   \sum_{i=j+4}^{n}r_{i}  \underline{(ax^{2} + xax + x^{2}a )P(j+1, i-j-4)ax^{2} } 
$$
$$
 + \sum_{i=j+5}^{n}r_{i} \underline{x^{3}P(j+2,i-j-5)ax^{2} } )  -  r_{j+2}\underline{ax^{2} a^{n}}
$$
respectively. Again, using (\ref{e}), expand $\sigma_{j+1}$ as
$$
a^{j+1}x^{n-j-1} \rightarrow - ( \sum_{i=j+1}^{n-1}r_{i}\underline{a^{3}P(j-2,i-j-1)} + \underline{a^{3}Q(j-2,n-j-1)} + 
$$
$$
\sum_{i=j+2}^{n-1}r_{i} \underline{(a^{2}x + axa + xa^{2}) P(j,i-j-3)} + \underline{(a^{2}x + axa + xa^{2}) P(j,n-j-3)} + 
$$
$$
 \sum_{i=j+3}^{n-1}r_{i} \underline{(ax^{2} + xax + x^{2}a)P(j,i-j-3)}+ \underline{(ax^{2} + xax + x^{2}a)P(j,n-j-3)}+ 
$$
$$
\sum_{i=j+4}^{n-1}r_{i}  \underline{x^{3}P(j+1,i-j-4)} +  \underline{x^{3}P(j+1,n-j-4)})  + r_{j+1}\underline{a^{n}}.
$$
$$
= - ( \sum_{i=j+1}^{n-1}r_{i}\underline{a^{3}P(j-2,i-j-1)} + \underline{a^{3}Q(j-2,n-j-1)} + 
$$
$$
\sum_{i=j+2}^{n}r_{i} \underline{(a^{2}x + axa + xa^{2}) P(j,i-j-3)} +   \sum_{i=j+3}^{n}r_{i} \underline{(ax^{2} + xax + x^{2}a)P(j,i-j-3)}+ 
$$
$$
\sum_{i=j+4}^{n}r_{i}  \underline{x^{3}P(j+1,i-j-4)} )  + r_{j+1}\underline{a^{n}}.
$$
Hence post multiplying this by $xa^{2}$, $axa$ and $a^{2}x$, and using Lemmas 2.6 and 2.7 to separate reducible and irreducible words, yield
\begin{equation}
\label{j6}
-(a^{j+1}x^{n-j-1})xa^{2} \rightarrow ( \sum_{i=j+1}^{n-1}r_{i}\underline{a^{3}P(j-2,i-j-1)xa^{2} } + \underline{a^{3}Q(j-2,n-j-1)xa^{2} } + 
\end{equation}
$$
\sum_{i=j+2}^{n}r_{i} (a^{2}x + axa + xa^{2}) P(j,i-j-2)xa^{2}  +   \sum_{i=j+3}^{n}r_{i} \underline{(ax^{2} + xax + x^{2}a)P(j,i-j-3)xa^{2} }
$$
$$
+ \sum_{i=j+4}^{n}r_{i}  \underline{x^{3}P(j+1,i-j-4)xa^{2} } )  - r_{j+1}\underline{xa^{n+2}},
$$

\begin{equation}
\label{j7}
-(a^{j+1}x^{n-j-1})axa \rightarrow (  \sum_{i=j+1}^{n-1}r_{i}\underline{a^{3}P(j-2,i-j-1)axa} + \underline{a^{3}Q(j-2,n-j-1)axa} + 
\end{equation}
$$
\sum_{i=j+2}^{n}r_{i} \underline{(a^{2}x + axa + xa^{2}) P(j,i-j-2)axa} +   \sum_{i=j+3}^{n}r_{i} \underline{(ax^{2} + xax + x^{2}a)P(j,i-j-3)axa}
$$
$$
+ \sum_{i=j+4}^{n}r_{i}  \underline{x^{3}P(j+1,i-j-4)axa} )  - r_{j+1}\underline{axa^{n+1}}
$$
and
\begin{equation}
\label{j8}
-(a^{j+1}x^{n-j-1})a^{2}x \rightarrow ( \sum_{i=j+1}^{n-1}r_{i}\underline{a^{3}P(j-2,i-j-1)a^{2}x} + \underline{a^{3}Q(j-2,n-j-1)a^{2}x} + 
\end{equation}
$$
\sum_{i=j+2}^{n}r_{i} \underline{(a^{2}x + axa + xa^{2}) P(j,i-j-2)a^{2}x} +   \sum_{i=j+3}^{n}r_{i} \underline{(ax^{2} + xax + x^{2}a)P(j,i-j-3)a^{2}x}
$$
$$
+ \sum_{i=j+4}^{n}r_{i}  \underline{x^{3}P(j+1,i-j-4)a^{2}x} )  - r_{j+1}\underline{a^{2}xa^{n}}
$$
respectively.
The words in (\ref{j2})-(\ref{j8}) of legth $n+3$ which are reducible are as follows:
\begin{enumerate}
\item $x^{2}a^{j+2}x^{n-j-2}a \in x^{2}aP(j+1, n-j-4)x^{2}a$
\item $x(a^{j+2}x^{n-j-2})xa \in  xa^{2}P(j,n-j-3)x^{2}a$
\item $x(a^{j+2}x^{n-j-2})ax \in  xa^{2}P(j,n-j-3)xax$
\item $xa^{j+1}x^{n-j-1}a^{2} \in  xa^{2}P(j,n-j-3)xa^{2}$
\item $axa^{j+1}x^{n-j-1}a \in axaP(j,n-j-3)x^{2}a $.
\end{enumerate}
Using (\ref{f}), expand $\sigma_{j+2}$ as
$$
a^{j+2}x^{n-j-2} \rightarrow - ( \sum_{i=j+2}^{n-1}r_{i} \underline{a^{2}P(j-1,i-j-2)a} + \underline{a^{2}P(j-1,n-j-2)a} 
$$
$$
+ \sum_{i=j+3}^{n-1}r_{i}\underline{a^{2}P(j,i-j-3)x} + \underline{a^{2}Q(j,n-j-3)x} +  \sum_{i=j+3}^{n-1}r_{i}  ( \underline{axP(j,i-j-3)a} +
$$
$$
  \underline{xaP(j,i-j-3)a} ) + \underline{axP(j,n-j-3)a} + \underline{xaP(j,n-j-3)a} +  
$$
$$
\sum_{i=j+4}^{n-1}r_{i}  ( \underline{axP(j+1,i-j-4)x} +  \underline{xaP(j+1,i-j-4)x} +  \underline{x^{2}P(j+1,i-j-4)a} )
$$
$$
+ \underline{axP(j+1,n-j-4)x} +  \underline{xaP(j+1,n-j-4)x} +  \underline{x^{2}P(j+1,n-j-4)a}  + 
$$
$$
\sum_{i=j+5}^{n-1}r_{i}\underline{x^{2}P(j+2,i-j-5)x} + \underline{x^{2}P(j+2,n-j-5)x}  ) + r_{j+2}\underline{a^{n}} 
$$
$$
= - ( \sum_{i=j+2}^{n}r_{i} \underline{a^{2}P(j-1,i-j-2)a} + \sum_{i=j+3}^{n-1}r_{i}\underline{a^{2}P(j,i-j-3)x} + \underline{a^{2}Q(j,n-j-3)x} +  
$$
$$
 \sum_{i=j+3}^{n}r_{i}  ( \underline{axP(j,i-j-3)a} + \underline{xaP(j,i-j-3)a} ) +  \sum_{i=j+4}^{n}r_{i}  ( \underline{axP(j+1,i-j-4)x} +  
$$
$$
\underline{xaP(j+1,i-j-4)x} +  \underline{x^{2}P(j+1,i-j-4)a} )+ \sum_{i=j+5}^{n}r_{i}\underline{x^{2}P(j+2,i-j-5)x}   ) + r_{j+2}\underline{a^{n}} .
$$
Thus, premultiplying this by $x$ and postmultiplying by $xa$ and $ax$, and using Lemmas 2.6 and 2.7 to separate reducible and irreducible words, yield
\begin{equation}
\label{j9}
x(a^{j+2}x^{n-j-2})xa \rightarrow - ( \sum_{i=j+2}^{n}r_{i} \underline{xa^{2}P(j-1,i-j-2)axa} + \sum_{i=j+3}^{n-1}r_{i}\underline{xa^{2}P(j,i-j-3)x^{2}a} 
\end{equation}
$$
+ \underline{xa^{2}Q(j,n-j-3)x^{2}a} + \sum_{i=j+3}^{n}r_{i}  ( \underline{xaxP(j,i-j-3)axa} + \underline{x^{2}aP(j,i-j-3)axa} ) 
$$
$$
+    \sum_{i=j+4}^{n}r_{i}  ( \underline{xaxP(j+1,i-j-4)x^{2}a} +  x^{2}aP(j+1,i-j-4)x^{2}a +  
$$
$$
\underline{x^{3}P(j+1,i-j-4)axa} ) + \sum_{i=j+5}^{n}r_{i}\underline{x^{3}P(j+2,i-j-5)x^{2}a}   ) + r_{j+2}\underline{x^{2}a^{n+1}} 
$$
and
\begin{equation}
\label{j10}
x(a^{j+2}x^{n-j-2})ax \rightarrow - ( \sum_{i=j+2}^{n}r_{i} \underline{xa^{2}P(j-1,i-j-2)a^{2}x} + \sum_{i=j+3}^{n-1}r_{i}\underline{xa^{2}P(j,i-j-3)xax} 
\end{equation}
$$
 + \underline{xa^{2}Q(j,n-j-3)xax} +   \sum_{i=j+3}^{n}r_{i}  ( \underline{xaxP(j,i-j-3)a^{2}x} + \underline{x^{2}aP(j,i-j-3)a^{2}x} )   
$$
$$
+ \sum_{i=j+4}^{n}r_{i}  ( \underline{xaxP(j+1,i-j-4)xax} +   \underline{x^{2}aP(j+1,i-j-4)xax} +  
$$
$$
\underline{x^{3}P(j+1,i-j-4)a^{2}x} )+  \sum_{i=j+5}^{n}r_{i}\underline{x^{3}P(j+2,i-j-5)xax}   ) + r_{j+2}\underline{xaxa^{n}} 
$$
respectively.
Similarly, expand $\sigma_{j+1}$ using (\ref{f}) to get
$$
a^{j+1}x^{n-j-1} \rightarrow -  ( \sum_{i=j+1}^{n-1}r_{i} \underline{a^{2}P(j-2,i-j-1)a} + \underline{a^{2}P(j-2,n-j-1)a} + 
$$
$$
\sum_{i=j+2}^{n-1}r_{i}\underline{a^{2}P(j-1,i-j-2)x} + \underline{a^{2}Q(j-1,n-j-2)x} +   
$$
$$
\sum_{i=j+2}^{n-1}r_{i}  ( \underline{axP(j-1,i-j-2)a} + \underline{xaP(j-1,i-j-2)a} ) +
$$
$$
\underline{axP(j-1,n-j-2)a} + \underline{xaP(j-1,n-j-2)a} +
$$
$$
   \sum_{i=j+3}^{n-1}r_{i}  ( \underline{axP(j,i-j-3)x} +  \underline{xaP(j,i-j-3)x} +  \underline{x^{2}P(j,i-j-3)a} )+
$$
$$
\underline{axP(j,n-j-3)x} +  \underline{xaP(j,n-j-3)x} +  \underline{x^{2}P(j,n-j-3)a} +
$$
$$
 \sum_{i=j+4}^{n-1}r_{i}\underline{x^{2}P(j+1,i-j-4)x} + \underline{x^{2}P(j+1,n-j-4)x}  ) + r_{j+1}\underline{a^{n}} .
$$
$$
= -  ( \sum_{i=j+1}^{n}r_{i} \underline{a^{2}P(j-2,i-j-1)a} + \sum_{i=j+2}^{n-1}r_{i}\underline{a^{2}P(j-1,i-j-2)x} + 
$$
$$
\underline{a^{2}Q(j-1,n-j-2)x} +   \sum_{i=j+2}^{n}r_{i}  ( \underline{axP(j-1,i-j-2)a} + \underline{xaP(j-1,i-j-2)a} ) +
$$
$$
   \sum_{i=j+3}^{n}r_{i}  ( \underline{axP(j,i-j-3)x} +  \underline{xaP(j,i-j-3)x} +  \underline{x^{2}P(j,i-j-3)a} )+
$$
$$
 \sum_{i=j+4}^{n}r_{i}\underline{x^{2}P(j+1,i-j-4)x}   ) + r_{j+1}\underline{a^{n}} .
$$
Thus, premultiplying this by $x$ and postmultiplying by $a^{2}$ and using Lemmas 2.6 and 2.7 to separate reducible and irreducible words, yields
\begin{equation}
\label{j11}
xa^{j+1}x^{n-j-1}a^{2} \rightarrow -  ( \sum_{i=j+1}^{n}r_{i} \underline{xa^{2}P(j-2,i-j-1)a^{3}} + \sum_{i=j+2}^{n-1}r_{i}\underline{xa^{2}P(j-1,i-j-2)xa^{2}} + 
\end{equation}
$$
\underline{xa^{2}Q(j-1,n-j-2)xa^{2}} +   \sum_{i=j+2}^{n}r_{i}  ( \underline{xaxP(j-1,i-j-2)a^{3}} + \underline{x^{2}aP(j-1,i-j-2)a^{3}} ) 
$$
$$
+   \sum_{i=j+3}^{n}r_{i}  ( \underline{xaxP(j,i-j-3)xa^{2}} +  \underline{x^{2}aP(j,i-j-3)xa^{2}} +  \underline{x^{3}P(j,i-j-3)a^{3}} )+
$$
$$
 \sum_{i=j+4}^{n}r_{i}\underline{x^{3}P(j+1,i-j-4)xa^{2}}   ) + r_{j+1}\underline{xa^{n+2}} .
$$
Similarly, expand $\sigma_{j+1}$ using (\ref{g}) to get
$$
a^{j+1}x^{n-j-1} \rightarrow - ( \sum_{i=j+1}^{n-1}r_{i} \underline{aP(j-2,i-j-1)a^{2}} +\underline{aP(j-2,n-j-1)a^{2}} + 
$$
$$
\sum_{i=j+2}^{n-1}r_{i}(\underline{xP(j-1,i-j-2)a^{2}} + \underline{aP(j-1,i-j-2)ax} +  \underline{aP(j-1,i-j-2)xa} )  +
$$
$$
  \underline{xP(j-1,n-j-2)a^{2}} + \underline{aP(j-1,n-j-2)ax} +  \underline{aP(j-1,n-j-2)xa} + 
$$
$$
\sum_{i=j+3}^{n-1}r_{i} ( \underline{xP(j,i-j-3)ax}  +  \underline{xP(j,i-j-3)xa} ) + \underline{xP(j,n-j-3)ax}  +
$$
$$
  +  \underline{xP(j,n-j-3)xa} + \sum_{i=j+3}^{n-1}r_{i} \underline{aP(j,i-j-3)x^{2}} + \underline{aQ(j,n-j-3)x^{2}} + 
$$
$$
 \sum_{i=j+4}^{n-1}r_{i}  \underline{xP(j+1,i-j-4)x^{2}} + \underline{xP(j+1,n-j-4)x^{2}} ) + r_{j+1}\underline{a^{n}} 
$$
$$
=- ( \sum_{i=j+1}^{n}r_{i} \underline{aP(j-2,i-j-1)a^{2}} + \sum_{i=j+2}^{n}r_{i}(\underline{xP(j-1,i-j-2)a^{2}} +
$$
$$
 \underline{aP(j-1,i-j-2)ax} +  \underline{aP(j-1,i-j-2)xa} )  + \sum_{i=j+3}^{n}r_{i} ( \underline{xP(j,i-j-3)ax}  +
$$
$$
  \underline{xP(j,i-j-3)xa} ) +  \sum_{i=j+3}^{n-1}r_{i} \underline{aP(j,i-j-3)x^{2}} + \underline{aQ(j,n-j-3)x^{2}} + 
$$
$$
 \sum_{i=j+4}^{n}r_{i}  \underline{xP(j+1,i-j-4)x^{2}}  ) + r_{j+1}\underline{a^{n}}.
$$
Premultiplying this by $ax$ and postmultiplying by $a$ and using Lemmas 2.6 and 2.7 to separate reducible and irreducible words, yields
\begin{equation}
\label{j12}
axa^{j+1}x^{n-j-1}a \rightarrow - ( \sum_{i=j+1}^{n}r_{i} \underline{axaP(j-2,i-j-1)a^{3}} + \sum_{i=j+2}^{n}r_{i}(\underline{ax^{2}P(j-1,i-j-2)a^{3}} 
\end{equation}
$$
+ \underline{axaP(j-1,i-j-2)axa} +  \underline{axaP(j-1,i-j-2)xa^{2}} )  + \sum_{i=j+3}^{n}r_{i} ( \underline{ax^{2}P(j,i-j-3)axa} 
$$
$$
 + \underline{ax^{2}P(j,i-j-3)xa^{2}} ) +  \sum_{i=j+3}^{n-1}r_{i} \underline{axaP(j,i-j-3)x^{2}a} + \underline{axaQ(j,n-j-3)x^{2}a} 
$$
$$
 + \sum_{i=j+4}^{n}r_{i}  \underline{ax^{2}P(j+1,i-j-4)x^{2}a}  ) + r_{j+1}\underline{axa^{n+1}} 
$$
The reducible word $ - x^{2}a^{j+2}x^{n-j-2}a$ in (\ref{j9}) has an opposite sign with $x^{2}a^{j+2}x^{n-j-2}a$, the term (i) in the list above so they cancel out. Thus, the reduction process ends here and when we substitute (\ref{j2})-(\ref{j12}) into (\ref{j1}), we get
\begin{equation}
\label{j13}
a^{3}\omega_{j} \rightarrow (  \sum_{i=j+1}^{n}r_{i} \underline{a^{2}xP(j-2,i-j-1)a^{3}} +  \sum_{i=j+2}^{n}r_{i} ( \underline{a^{2}xP(j,i-j-2)(a^{2}x + axa +xa^{2})}   + 
\end{equation}
$$ 
 \underline{axaP(j,i-j-2)a^{2}x} ) +  \sum_{i=j+3}^{n}r_{i} ( \underline{a^{2}xP(j,i-j-3) (ax^{2} + xax + x^{2}a)}  +  \underline{axaP(j,i-j-3)(ax^{2}+xax)} 
$$
$$
+  \underline{xa^{2}P(j,i-j-3) ax^{2}} + \underline{ax^{2}P(j,i-j-3)a^{2}x} ) +  \sum_{i=j+4}^{n}r_{i}(\underline{ax^{2}P(j+1, i-j-4) (ax^{2} + xax)}    
$$
$$
+ \underline{(xax + x^{2}a)P(j+1, i-j-4) ax^{2}})  +     \sum_{i=j+5}^{n}r_{i}  \underline{x^{3}P(j+2,i-j-5) ax^{2}}  )
$$
$$
  -  (   r_{j+1} \underline{a^{2}xa^{n}} + r_{j+2} \underline{ax^{2}a^{n}} +  \sum_{i=j+3}^{n-1}r_{i}  a^{3}P(j,i-j)x^{3} + \underline{a^{3}Q(j,n-j)x^{3}}  ).
$$

Similarly, use (\ref{e}) to expand $\sigma_{j+3}$ as
$$
\omega_{j+3} \rightarrow - ( \sum_{i=j+3}^{n-1}r_{i}  \underline{a^{3}P(j,i-j-3)} + \underline{a^{3}Q(j,n-j-3)} + \sum_{i=j+4}^{n-1}r_{i} \underline{(a^{2}x + axa + xa^{2})P(j+1,i-j-4)} 
$$
$$
+ \underline{(a^{2}x + axa + xa^{2})P(j+1,n-j-4)}  +  \sum_{i=j+5}^{n}r_{i} \underline{(ax^{2} + xax + x^{2}a)P(j+2,i-j-5)}  
$$
$$
+ \sum_{i=j+6}^{n-1}r_{i} \underline{x^{3}P(j+3, i-j-6) )} + \underline{x^{3}P(j+3, n-j-6) )} ) + r_{j+3}\underline{a^{n}}
$$
$$
= - ( \sum_{i=j+3}^{n-1}r_{i}  \underline{a^{3}P(j,i-j-3)} + \underline{a^{3}Q(j,n-j-3)} + \sum_{i=j+4}^{n}r_{i} \underline{(a^{2}x + axa + xa^{2})P(j+1,i-j-4)} 
$$
$$
+  \sum_{i=j+5}^{n}r_{i} \underline{(ax^{2} + xax + x^{2}a)P(j+2,i-j-5)} + \sum_{i=j+6}^{n}r_{i} \underline{x^{3}P(j+3, i-j-6) )} ) + r_{j+3}\underline{a^{n}}.
$$
Post multiplying this by $x^{3}$ and using Lemmas 2.6 and 2.7 to separate reducible and irreducible words, yields
\begin{equation}
\label{j14}
\omega_{j+3}x^{3} \rightarrow - (  \sum_{i=j+3}^{n-1}r_{i}  a^{3}P(j,i-j-3)x^{3} + \underline{a^{3}Q(j,n-j-3)x^{3}} + 
\end{equation}
$$
\sum_{i=j+4}^{n}r_{i} (a^{2}x + axa + xa^{2})P(j+1,i-j-4)x^{3} +
$$
$$
  \sum_{i=j+5}^{n}r_{i} (ax^{2} + xax + x^{2}a)P(j+2,i-j-5)x^{3} + 
$$
$$
\sum_{i=j+6}^{n}r_{i} x^{3}P(j+3, i-j-6) )x^{3} ) + r_{j+3}\underline{a^{n+3}}.
$$
The reducible words in (\ref{j14}) of length $n+3$ are:
\begin{enumerate}
\item $x^{2}a(a^{j+2}x^{n-j-2}) \in  x^{2}aP(j+2,n-j-5)x^{3}$
\item $xax(a^{j+2}x^{n-j-2}) \in  xaxP(j+2,n-j-5)x^{3}$
\item $ax^{2}(a^{j+2}x^{n-j-2}) \in  ax^{2}P(j+2,n-j-5)x^{3}$
\item $xa^{2}(a^{j+1}x^{n-j-1}) \in xa^{2}P(j+1,n-j-4)x^{3}$
\item $axa(a^{j+1}x^{n-j-1}) \in axaP(j+1,n-j-4)x^{3}$
\item $a^{2}x(a^{j+1}x^{n-j-1}) \in a^{2}xP(j+1,n-j-4)x^{3}$
\item $x^{3}a^{j+3}x^{n-j-3} \in x^{3}P(j+3, n-j-6)x^{3}$.
\end{enumerate}
Using (\ref{d}), expand $\sigma_{j+2}$ as
$$
a^{j+2}x^{n-j-2} \rightarrow - ( \sum_{i=j+2}^{n-1}r_{i} \underline{P(j-1,i-j-2)a^{3}} + \underline{P(j-1,n-j-2)a^{3}} +  
$$
$$
\sum_{i=j+3}^{n-1}r_{i}  \underline{P(j,i-j-3)(a^{2}x + axa + xa^{2})} + \underline{P(j,n-j-3)(a^{2}x + axa + xa^{2})} +    
$$
$$
\sum_{i=j+4}^{n-1}r_{i} \underline{P(j+1,i-j-4)(ax^{2} + xax + x^{2}a)} + \underline{P(j+1,n-j-4)(ax^{2} + xax + x^{2}a)} +  
$$
$$
\sum_{i=j+5}^{n-1}r_{i} \underline{P(j+2,i-j-5)x^{3}}  +  \underline{Q(j+2,n-j-5)x^{3}} ) +   r_{j+2}\underline{a^{n}}
$$
$$
= - ( \sum_{i=j+2}^{n}r_{i} \underline{P(j-1,i-j-2)a^{3}} + \sum_{i=j+3}^{n}r_{i}  \underline{P(j,i-j-3)(a^{2}x + axa + xa^{2})} +    
$$
$$
\sum_{i=j+4}^{n}r_{i} \underline{P(j+1,i-j-4)(ax^{2} + xax + x^{2}a)} +  \sum_{i=j+5}^{n-1}r_{i} \underline{P(j+2,i-j-5)x^{3}}  
$$
$$
 +  \underline{Q(j+2,n-j-5)x^{3}} ) +   r_{j+2}\underline{a^{n}}.
$$
Thus, premultiplying this by $x^{2}a$, $xax$ and $ax^{2}$, and using Lemmas 2.6 and 2.7 to separate reducible and irreducible words, yield
\begin{equation}
\label{j15}
-x^{2}a(a^{j+2}x^{n-j-2}) \rightarrow ( \sum_{i=j+2}^{n}r_{i} \underline{x^{2}aP(j-1,i-j-2)a^{3}} + \sum_{i=j+3}^{n}r_{i}  \underline{x^{2}aP(j,i-j-3)(a^{2}x + axa + xa^{2})}   
\end{equation}
$$
+ \sum_{i=j+4}^{n}r_{i} x^{2}aP(j+1,i-j-4)(ax^{2} + xax + x^{2}a) +  \sum_{i=j+5}^{n-1}r_{i} \underline{x^{2}aP(j+2,i-j-5)x^{3}}   
$$
$$
+  \underline{x^{2}aQ(j+2,n-j-5)x^{3}} ) -   r_{j+2}\underline{x^{2}a^{n+1}},
$$

\begin{equation}
\label{j16}
-xax(a^{j+2}x^{n-j-2}) \rightarrow ( \sum_{i=j+2}^{n}r_{i} \underline{xaxP(j-1,i-j-2)a^{3}} + \sum_{i=j+3}^{n}r_{i}  \underline{xaxP(j,i-j-3)(a^{2}x + axa + xa^{2})}  
\end{equation}
$$
+ \sum_{i=j+4}^{n}r_{i} \underline{xaxP(j+1,i-j-4)(ax^{2} + xax + x^{2}a)} +  \sum_{i=j+5}^{n-1}r_{i} \underline{xaxP(j+2,i-j-5)x^{3}}  
$$
$$
 +  \underline{xaxQ(j+2,n-j-5)x^{3}} ) -   r_{j+2}\underline{xaxa^{n}}
$$
and
\begin{equation}
\label{j17}
-ax^{2}(a^{j+2}x^{n-j-2}) \rightarrow ( \sum_{i=j+2}^{n}r_{i} \underline{ax^{2}P(j-1,i-j-2)a^{3}} + \sum_{i=j+3}^{n}r_{i}  \underline{ax^{2}P(j,i-j-3)(a^{2}x + axa + xa^{2})}    
\end{equation}
$$
+ \sum_{i=j+4}^{n}r_{i} \underline{ax^{2}P(j+1,i-j-4)(ax^{2} + xax + x^{2}a)} +  \sum_{i=j+5}^{n-1}r_{i} \underline{ax^{2}P(j+2,i-j-5)x^{3}}  
$$
$$
+  \underline{ax^{2}Q(j+2,n-j-5)x^{3}} ) -   r_{j+2}\underline{ax^{2}a^{n}}
$$
respectively.
Again, using (\ref{d}), expand $\sigma_{j+1}$ as
$$
a^{j+1}x^{n-j-1} \rightarrow - ( \sum_{i=j+1}^{n-1}r_{i} \underline{P(j-2,i-j-1)a^{3}} + \underline{P(j-2,n-j-1)a^{3}} + 
$$
$$
\sum_{i=j+2}^{n-1}r_{i} \underline{P(j-1,i-j-2)(a^{2}x + axa + xa^{2})} + \underline{P(j-1,n-j-2)(a^{2}x + axa + xa^{2})} + 
$$
$$
\sum_{i=j+3}^{n-1}r_{i} \underline{P(j,i-j-3)(ax^{2} + xax + x^{2}a )}  + \underline{P(j,n-j-3)(ax^{2} + xax + x^{2}a )}  + 
$$
$$
 \sum_{i=j+4}^{n-1}r_{i}\underline{P(j+1,i-j-4) x^{3}} + \underline{Q(j+1,n-j-4) x^{3}} ) + r_{j+1}\underline{a^{n}} .
$$
$$
= - ( \sum_{i=j+1}^{n}r_{i} \underline{P(j-2,i-j-1)a^{3}} + \sum_{i=j+2}^{n}r_{i} \underline{P(j-1,i-j-2)(a^{2}x + axa + xa^{2})} + 
$$
$$
\sum_{i=j+3}^{n}r_{i} \underline{P(j,i-j-3)(ax^{2} + xax + x^{2}a )}  + \sum_{i=j+4}^{n-1}r_{i}\underline{P(j+1,i-j-4) x^{3}}  
$$
$$
+ \underline{Q(j+1,n-j-4) x^{3}} )+ r_{j+1}\underline{a^{n}} .
$$
Hence, premultiplying this by $xa^{2}$, $axa$ and $a^{2}x$, and using Lemmas 2.6 and 2.7 to separate reducible and irreducible words, yield
\begin{equation}
\label{j18}
- xa^{2}(a^{j+1}x^{n-j-1}) \rightarrow ( \sum_{i=j+1}^{n}r_{i} \underline{xa^{2}P(j-2,i-j-1)a^{3}} + \sum_{i=j+2}^{n}r_{i} xa^{2}P(j-1,i-j-2)(a^{2}x + axa + xa^{2})  
\end{equation}
$$
+ \sum_{i=j+3}^{n}r_{i} xa^{2}P(j,i-j-3)(ax^{2} + xax + x^{2}a )  + \sum_{i=j+4}^{n-1}r_{i}\underline{xa^{2}P(j+1,i-j-4) x^{3}} 
$$
$$
+ \underline{xa^{2}Q(j+1,n-j-4) x^{3}} )  -  r_{j+1}\underline{xa^{n+2}} ,
$$

\begin{equation}
\label{j19}
- axa(a^{j+1}x^{n-j-1}) \rightarrow ( \sum_{i=j+1}^{n}r_{i} \underline{ axaP(j-2,i-j-1)a^{3}} + \sum_{i=j+2}^{n}r_{i} \underline{ axaP(j-1,i-j-2)(a^{2}x + axa + xa^{2})}  
\end{equation}
$$
+ \sum_{i=j+3}^{n}r_{i} axaP(j,i-j-3)(ax^{2} + xax + x^{2}a )  + \sum_{i=j+4}^{n-1}r_{i}\underline{ axaP(j+1,i-j-4) x^{3}} 
$$
$$
 + \underline{ axaQ(j+1,n-j-4) x^{3}} )  -  r_{j+1}\underline{axa^{n+1}}
$$
and
\begin{equation}
\label{j20}
- a^{2}x(a^{j+1}x^{n-j-1}) \rightarrow ( \sum_{i=j+1}^{n}r_{i} \underline{a^{2}xP(j-2,i-j-1)a^{3}} + \sum_{i=j+2}^{n}r_{i} \underline{a^{2}xP(j-1,i-j-2)(a^{2}x + axa + xa^{2})}  
\end{equation}
$$
+ \sum_{i=j+3}^{n}r_{i} \underline{a^{2}xP(j,i-j-3)(ax^{2} + xax + x^{2}a )}  + \sum_{i=j+4}^{n-1}r_{i}\underline{a^{2}xP(j+1,i-j-4) x^{3}} 
$$
$$
 + \underline{a^{2}xQ(j+1,n-j-4) x^{3}} )  -  r_{j+1}\underline{a^{2}xa^{n}} 
$$
respectively.
Also, use (\ref{d}) to expand $\sigma_{j+3}$ as
$$
a^{j+3}x^{n-j-3} \rightarrow - ( \sum_{i=j+3}^{n-1}r_{i} \underline{P(j,i-j-3)a^{3}} + \underline{P(j,n-j-3)a^{3}} + 
$$
$$
\sum_{i=j+4}^{n-1}r_{i}\underline{P(j+1,i-j-4) ( a^{2}x + axa + xa^{2})} +  \underline{P(j+1,n-j-4) ( a^{2}x + axa + xa^{2})} +  
$$
$$
 \sum_{i=j+5}^{n-1}r_{i} \underline{P(j+2,i-j-5) (ax^{2} + xax + x^{2}a)} + \underline{P(j+2,n-j-5) (ax^{2} + xax + x^{2}a)} +  
$$
$$
\sum_{i=j+6}^{n-1}r_{i}  \underline{P(j+3,i-j-6) x^{3}}  +  \underline{Q(j+3,n-j-6) x^{3}} ) + r_{j+3}\underline{a^{n}}
$$
$$
= - ( \sum_{i=j+3}^{n}r_{i} \underline{P(j,i-j-3)a^{3}} + \sum_{i=j+4}^{n}r_{i}\underline{P(j+1,i-j-4) ( a^{2}x + axa + xa^{2})} +  
$$
$$
 \sum_{i=j+5}^{n}r_{i} \underline{P(j+2,i-j-5) (ax^{2} + xax + x^{2}a)} +  \sum_{i=j+6}^{n-1}r_{i}  \underline{P(j+3,i-j-6) x^{3}}  
$$
$$
+  \underline{Q(j+3,n-j-6) x^{3}} ) + r_{j+3}\underline{a^{n}}
$$
Premultiplying this by $x^{3}$ and using Lemmas 2.6 and 2.7 to separate reducible and irreducible words, yields
\begin{equation}
\label{j21}
- x^{3}a^{j+3}x^{n-j-3} \rightarrow ( \sum_{i=j+3}^{n}r_{i} \underline{x^{3}P(j,i-j-3)a^{3}} + \sum_{i=j+4}^{n}r_{i}\underline{x^{3}P(j+1,i-j-4) ( a^{2}x + axa + xa^{2})}   
\end{equation}
$$
+ \sum_{i=j+5}^{n}r_{i} \underline{x^{3}P(j+2,i-j-5) (ax^{2} + xax + x^{2}a)} +  \sum_{i=j+6}^{n-1}r_{i}  \underline{x^{3}P(j+3,i-j-6) x^{3}}  
$$
$$
+  \underline{x^{3}Q(j+3,n-j-6) x^{3}} ) - r_{j+3}\underline{x^{3}a^{n}}.
$$
The words in (\ref{j15})-(\ref{j21}) of length $n+3$ which are reducible are as follows:
\begin{enumerate}
\item $x^{2}a^{j+2}x^{n-j-2}a \in x^{2}aP(j+1, n-j-4)x^{2}a$
\item $x(a^{j+2}x^{n-j-2})xa \in  xa^{2}P(j,n-j-3)x^{2}a$
\item $x(a^{j+2}x^{n-j-2})ax \in  xa^{2}P(j,n-j-3)xax$
\item $xa^{j+1}x^{n-j-1}a^{2} \in  xa^{2}P(j,n-j-3)xa^{2}$
\item $axa^{j+1}x^{n-j-1}a \in axaP(j,n-j-3)x^{2}a $.
\end{enumerate}
Recall from (\ref{j9}),  (\ref{j10}), (\ref{j11}) and (\ref{j12}) that 
$$
x(a^{j+2}x^{n-j-2})xa \rightarrow - ( \sum_{i=j+2}^{n}r_{i} \underline{xa^{2}P(j-1,i-j-2)axa} + \sum_{i=j+3}^{n-1}r_{i}\underline{xa^{2}P(j,i-j-3)x^{2}a} +
$$
$$
 \underline{xa^{2}Q(j,n-j-3)x^{2}a} + \sum_{i=j+3}^{n}r_{i}  ( \underline{xaxP(j,i-j-3)axa} + \underline{x^{2}aP(j,i-j-3)axa} ) +
$$
$$
    \sum_{i=j+4}^{n}r_{i}  ( \underline{xaxP(j+1,i-j-4)x^{2}a} + x^{2}aP(j+1,i-j-4)x^{2}a +  \underline{x^{3}P(j+1,i-j-4)axa} )
$$
$$
+ \sum_{i=j+5}^{n}r_{i}\underline{x^{3}P(j+2,i-j-5)x^{2}a}   ) + r_{j+2}\underline{x^{2}a^{n+1}} ,
$$

$$
x(a^{j+2}x^{n-j-2})ax \rightarrow - ( \sum_{i=j+2}^{n}r_{i} \underline{xa^{2}P(j-1,i-j-2)a^{2}x} + \sum_{i=j+3}^{n-1}r_{i}\underline{xa^{2}P(j,i-j-3)xax} + 
$$
$$
 \underline{xa^{2}Q(j,n-j-3)xax} +   \sum_{i=j+3}^{n}r_{i}  ( \underline{xaxP(j,i-j-3)a^{2}x} + \underline{x^{2}aP(j,i-j-3)a^{2}x} ) +  
$$
$$
\sum_{i=j+4}^{n}r_{i}  ( \underline{xaxP(j+1,i-j-4)xax} +   \underline{x^{2}aP(j+1,i-j-4)xax} +  \underline{x^{3}P(j+1,i-j-4)a^{2}x} )+ 
$$
$$
\sum_{i=j+5}^{n}r_{i}\underline{x^{3}P(j+2,i-j-5)xax}   ) + r_{j+2}\underline{xaxa^{n}} ,
$$

$$
xa^{j+1}x^{n-j-1}a^{2} \rightarrow -  ( \sum_{i=j+1}^{n}r_{i} \underline{xa^{2}P(j-2,i-j-1)a^{3}} + \sum_{i=j+2}^{n-1}r_{i}\underline{xa^{2}P(j-1,i-j-2)xa^{2}} + 
$$
$$
\underline{xa^{2}Q(j-1,n-j-2)xa^{2}} +   \sum_{i=j+2}^{n}r_{i}  ( \underline{xaxP(j-1,i-j-2)a^{3}} + \underline{x^{2}aP(j-1,i-j-2)a^{3}} ) +
$$
$$
   \sum_{i=j+3}^{n}r_{i}  ( \underline{xaxP(j,i-j-3)xa^{2}} +  \underline{x^{2}aP(j,i-j-3)xa^{2}} +  \underline{x^{3}P(j,i-j-3)a^{3}} )+
$$
$$
 \sum_{i=j+4}^{n}r_{i}\underline{x^{3}P(j+1,i-j-4)xa^{2}}   ) + r_{j+1}\underline{xa^{n+2}} 
$$
and
$$
axa^{j+1}x^{n-j-1}a \rightarrow - ( \sum_{i=j+1}^{n}r_{i} \underline{axaP(j-2,i-j-1)a^{3}} + \sum_{i=j+2}^{n}r_{i}(\underline{ax^{2}P(j-1,i-j-2)a^{3}} +
$$
$$
 \underline{axaP(j-1,i-j-2)axa} +  \underline{axaP(j-1,i-j-2)xa^{2}} )  + \sum_{i=j+3}^{n}r_{i} ( \underline{ax^{2}P(j,i-j-3)axa}  +
$$
$$
  \underline{ax^{2}P(j,i-j-3)xa^{2}} ) +  \sum_{i=j+3}^{n-1}r_{i} \underline{axaP(j,i-j-3)x^{2}a} + \underline{axaQ(j,n-j-3)x^{2}a} + 
$$
$$
 \sum_{i=j+4}^{n}r_{i}  \underline{ax^{2}P(j+1,i-j-4)x^{2}a}  ) + r_{j+1}\underline{axa^{n+1}} 
$$
The reducible word $-x^{2}a^{j+2}x^{n-j-2}a$ in (\ref{j9}) appears with a sign opposite to that of $x^{2}a^{j+2}x^{n-j-2}a$, the term (i) in the list above which makes them cancel out. Hence, the reduction process ends here and when we substitute (\ref{j9})-(\ref{j11}) and (\ref{j15})-(\ref{j21}) into (\ref{j14}), we get
\begin{equation}
\label{j26}
\omega_{j+3}x^{3} \rightarrow ( \sum_{i=j+1}^{n}r_{i} \underline{a^{2}xP(j-2,i-j-1)a^{3}} +  \sum_{i=j+2}^{n}r_{i} ( \underline{a^{2}xP(j,i-j-3)(a^{2}x + axa +xa^{2})}   + 
\end{equation}
$$ 
 \underline{axaP(j,i-j-3)a^{2}x} ) +  \sum_{i=j+3}^{n}r_{i} ( \underline{a^{2}xP(j,i-j-3) (ax^{2} + xax + x^{2}a)}  +  \underline{axaP(j,i-j-3)(ax^{2}+xax)} 
$$
$$
+  \underline{xa^{2}P(j,i-j-3) ax^{2}} + \underline{ax^{2}P(j,i-j-3)a^{2}x} ) +  \sum_{i=j+4}^{n}r_{i}(\underline{ax^{2}P(j+1, i-j-4) (ax^{2} + xax)}    
$$
$$
+ \underline{(xax + x^{2}a)P(j+1, i-j-4) ax^{2}})  +     \sum_{i=j+5}^{n}r_{i}  \underline{x^{3}P(j+2,i-j-5) ax^{2}}  )
$$
$$
  -  (   r_{j+1} \underline{a^{2}xa^{n}} + r_{j+2} \underline{ax^{2}a^{n}} +  \sum_{i=j+3}^{n-1}r_{i}  a^{3}P(j,i-j)x^{3} + \underline{a^{3}Q(j,n-j)x^{3}}  ).
$$

Comparing (\ref{j14}) and (\ref{j26}), we conclude that the overlap ambiguity $\{ \omega_{j}, \omega_{j+3} \}$ is resolvable for $3 \leq j < n-4$

\noindent (d) Let $j=n-4$ and consider the overlap ambiguity $\{ \omega_{n-4}, \omega_{n-1} \}$. Using (\ref{d}), expand $\sigma_{n-4}$ as
$$
\omega_{n-4} \rightarrow - ( \sum_{i=n-4}^{n-1}r_{i} \underline{P(n-7,i-n+4)a^{3}} + \underline{P(n-7,4)a^{3}} + 
$$
$$
 \sum_{i=n-3}^{n-1}r_{i}\underline{P(n-6,i-n+3)(a^{2}x + axa + xa^{2})}   +  \underline{P(n-6,3)(a^{2}x + axa + xa^{2})}   +  
$$
$$
\sum_{i=n-2}^{n-1}r_{i}  \underline{P(n-5,i-n+2)(ax^{2} + xax + x^{2}a)}  +  \underline{P(n-5,2)(ax^{2} + xax + x^{2}a)}  + 
$$
$$
r_{n-1}\underline{P(n-4,0)x^{3}} +  \underline{Q(n-4,1)x^{3}}   ) + r_{n-4}\underline{a^{n}}
$$
$$
= - ( \sum_{i=n-4}^{n}r_{i} \underline{P(n-7,i-n+4)a^{3}} + \sum_{i=n-3}^{n}r_{i}\underline{P(n-6,i-n+3)(a^{2}x + axa + xa^{2})}   +  
$$
$$
\sum_{i=n-2}^{n}r_{i}  \underline{P(n-5,i-n+2)(ax^{2} + xax + x^{2}a)}  +  r_{n-1}\underline{P(n-4,0)x^{3}} +  \underline{Q(n-4,1)x^{3}}   ) + r_{n-4}\underline{a^{n}}
$$
Premultiply this by $a^{3}$ and use Lemmas 2.6 and 2.7 to separate reducible and irreducible words, yields
\begin{equation}
\label{n1}
a^{3}\omega_{n-4} \rightarrow - ( \sum_{i=n-4}^{n}r_{i} a^{3}P(n-7,i-n+4)a^{3} + \sum_{i=n-3}^{n}r_{i}a^{3}P(n-6,i-n+3)(a^{2}x + axa + xa^{2})   
\end{equation}
$$
+ \sum_{i=n-2}^{n}r_{i}  a^{3}P(n-5,i-n+2)(ax^{2} + xax + x^{2}a)  +  r_{n-1}a^{3}P(n-4,0)x^{3} 
$$
$$
+  \underline{a^{3}Q(n-4,1)x^{3}}   ) + r_{n-4}\underline{a^{n+3}}
$$
The reducible words above of length $n+3$ are as follows:
\begin{enumerate}
\item $a^{n-4}x^{4}a^{3} \in a^{3}P(n-7,4)a^{3}$
\item $(a^{n-2}x^{2})x^{2}a \in a^{3}P(n-5,2)x^{2}a $
\item $(a^{n-2}x^{2})xax \in a^{3}P(n-5,2)xax $
\item $(a^{n-2}x^{2})ax^{2} \in a^{3}P(n-5,2)ax^{2} $
\item $(a^{n-3}x^{3})xa^{2} \in a^{3}P(n-6,3)xa^{2} $
\item $(a^{n-3}x^{3})axa \in a^{3}P(n-6,3)axa $
\item $(a^{n-3}x^{3})a^{2}x \in a^{3}P(n-6,3)a^{2}x $.
\end{enumerate}
Using (\ref{e}) expand $\sigma_{n-4}$ as
$$
a^{n-4}x^{4} \rightarrow - ( \sum_{i=n-4}^{n-1}r_{i} \underline{a^{3}P(n-7,i-n+4)} + \underline{a^{3}Q(n-7,4)} + 
$$
$$
\sum_{i=n-3}^{n-1}r_{i}\underline{(a^{2}x + axa + x^{2}a)P(n-6,i-n+3)}    + \underline{(a^{2}x + axa + x^{2}a)P(n-6,3)}    +
$$
$$
+ \sum_{i=n-2}^{n-1}r_{i}  \underline{(ax^{2} + xax + x^{2}a)P(n-5,i-n+2)} + \underline{(ax^{2} + xax + x^{2}a)P(n-5,2)} +
$$
$$
 \sum_{i=n-1}^{n}r_{i}\underline{x^{3}P(n-4,i-n+1)} ) + r_{n-4}\underline{a^{n}}
$$
$$
= - ( \sum_{i=n-4}^{n-1}r_{i} \underline{a^{3}P(n-7,i-n+4)} + \underline{a^{3}Q(n-7,4)} + \sum_{i=n-3}^{n}r_{i}\underline{(a^{2}x + axa + x^{2}a)P(n-6,i-n+3)}    
$$
$$
+ \sum_{i=n-2}^{n}r_{i}  \underline{(ax^{2} + xax + x^{2}a)P(n-5,i-n+2)} + \sum_{i=n-1}^{n}r_{i}\underline{x^{3}P(n-4,i-n+1)} ) + r_{n-4}\underline{a^{n}}.
$$
Post multiplying this by $a^{3}$ and using Lemmas 2.6 and 2.7 to separate reducible and irreducible words, yields
\begin{equation}
\label{n2}
-a^{n-4}x^{4}a^{3} \rightarrow ( \sum_{i=n-4}^{n-1}r_{i} \underline{a^{3}P(n-7,i-n+4)a^{3}} + \underline{a^{3}Q(n-7,4)a^{3}} + 
\end{equation}
$$
\sum_{i=n-3}^{n}r_{i}\underline{(a^{2}x + axa + x^{2}a)P(n-6,i-n+3)a^{3}}     + \sum_{i=n-2}^{n}r_{i}  \underline{(ax^{2} + xax + x^{2}a)P(n-5,i-n+2)a^{3}}
$$
$$
 + \sum_{i=n-1}^{n}r_{i}\underline{x^{3}P(n-4,i-n+1)a^{3}} ) + r_{n-4}\underline{a^{n+3}}.
$$
Use (\ref{e}) to expand $\sigma_{n-2}$ as
$$
a^{n-2}x^{2} \rightarrow - ( \sum_{i=n-2}^{n-1}r_{i}   \underline{a^{3}P(n-5,i-n+2)} + \underline{a^{3}Q(n-5,2)} + \sum_{i=n-1}^{n}r_{i}\underline{(a^{2}x + axa + xa^{2})P(n-4,i-n+1)} 
$$
$$
+  \underline{(ax^{2} + xax + x^{2}a )P(n-3,0)} ) + r_{n-2}\underline{a^{n}} .
$$
Thus, post multiplying this by $x^{2}a$, $xax$ and $ax^{2}$, and using Lemmas 2.6 and 2.7 to separate reducible and irreducible words, yield
\begin{equation}
\label{n3}
- (a^{n-2}x^{2})x^{2}a \rightarrow ( \sum_{i=n-2}^{n-1}r_{i}   a^{3}P(n-5,i-n+2)x^{2}a + \underline{a^{3}Q(n-5,2)x^{2}a} +
\end{equation}
$$
 \sum_{i=n-1}^{n}r_{i} (a^{2}x + axa + xa^{2})P(n-4,i-n+1)x^{2}a  +  \underline{(ax^{2} + xax + x^{2}a )P(n-3,0)x^{2}a} )  -  r_{n-2}\underline{x^{2}a^{n+1}} ,
$$

\begin{equation}
\label{n4}
- (a^{n-2}x^{2})xax \rightarrow ( \sum_{i=n-2}^{n-1}r_{i}   \underline{a^{3}P(n-5,i-n+2)xax} + \underline{a^{3}Q(n-5,2)xax} + 
\end{equation}
$$
\sum_{i=n-1}^{n}r_{i}(a^{2}x + axa + xa^{2})P(n-4,i-n+1)xax   +  \underline{(ax^{2} + xax + x^{2}a )P(n-3,0)xax} ) -  r_{n-2}\underline{xaxa^{n}}
$$
and
\begin{equation}
\label{n5}
- (a^{n-2}x^{2})ax^{2} \rightarrow ( \sum_{i=n-2}^{n-1}r_{i}   \underline{a^{3}P(n-5,i-n+2)ax^{2}} + \underline{a^{3}Q(n-5,2)ax^{2}} + 
\end{equation}
$$
\sum_{i=n-1}^{n}r_{i}\underline{(a^{2}x + axa + xa^{2})P(n-4,i-n+1)ax^{2}} +  (ax^{2} + xax + x^{2}a )P(n-3,0)ax^{2} ) - r_{n-2}\underline{ax^{2}a^{n}} 
$$
respectively.
Also, using (\ref{e}), expand $\sigma_{n-3}$ as
$$
a^{n-3}x^{3}  \rightarrow - ( \sum_{i=n-3}^{n-1}r_{i} \underline{a^{3}P(n-6,i-n+3)} + \underline{a^{3}Q(n-6,3)}+  \underline{x^{3}P(n-3,0)} + 
$$
$$
\sum_{i=n-2}^{n-1}r_{i}\underline{(a^{2}x + axa + xa^{2})P(n-5,i-n+2)} + \underline{(a^{2}x + axa + xa^{2})P(n-5,2)} +  
$$
$$
\sum_{i=n-1}^{n}r_{i} \underline{(ax^{2} + xax + x^{2}a) P(n-4,i-n+1)}) + r_{n-3}\underline{a^{n}}.
$$
$$
= - ( \sum_{i=n-3}^{n-1}r_{i} \underline{a^{3}P(n-6,i-n+3)} + \underline{a^{3}Q(n-6,3)}+  \underline{x^{3}P(n-3,0)} + 
$$
$$
\sum_{i=n-2}^{n}r_{i}\underline{(a^{2}x + axa + xa^{2})P(n-5,i-n+2)} +  \sum_{i=n-1}^{n}r_{i} \underline{(ax^{2} + xax + x^{2}a) P(n-4,i-n+1)})
$$
$$
 + r_{n-3}\underline{a^{n}}.
$$
Hence, post multiplying this by $xa^{2}$, $axa$ and $a^{2}x$, and using Lemmas 2.6 and 2.7 to separate reducible and irreducible words, yield
\begin{equation}
\label{n6}
- (a^{n-3}x^{3})xa^{2}  \rightarrow ( \sum_{i=n-3}^{n-1}r_{i} \underline{a^{3}P(n-6,i-n+3)xa^{2}} + \underline{a^{3}Q(n-6,3)xa^{2}}+  \underline{x^{3}P(n-3,0)xa^{2}} +  
\end{equation}
$$
\sum_{i=n-2}^{n}r_{i}(a^{2}x + axa + xa^{2})P(n-5,i-n+2)xa^{2} +  \sum_{i=n-1}^{n}r_{i} \underline{(ax^{2} + xax + x^{2}a) P(n-4,i-n+1)xa^{2}})
$$
$$
 -  r_{n-3}\underline{xa^{n+2}},
$$

\begin{equation}
\label{n7}
- (a^{n-3}x^{3})axa  \rightarrow ( \sum_{i=n-3}^{n-1}r_{i} \underline{a^{3}P(n-6,i-n+3))axa} + \underline{a^{3}Q(n-6,3))axa}+   
\end{equation}
$$
\underline{x^{3}P(n-3,0))axa} + \sum_{i=n-2}^{n}r_{i}\underline{(a^{2}x + axa + xa^{2})P(n-5,i-n+2))axa} +  
$$
$$
\sum_{i=n-1}^{n}r_{i} \underline{(ax^{2} + xax + x^{2}a) P(n-4,i-n+1))axa})  -  r_{n-3}\underline{axa^{n+1}}
$$
and
\begin{equation}
\label{n8}
- (a^{n-3}x^{3})a^{2}x  \rightarrow  ( \sum_{i=n-3}^{n-1}r_{i} \underline{a^{3}P(n-6,i-n+3)a^{2}x} + \underline{a^{3}Q(n-6,3)a^{2}x}+  
\end{equation}
$$
x^{3}P(n-3,0)a^{2}x +  \sum_{i=n-2}^{n}r_{i}\underline{(a^{2}x + axa + xa^{2})P(n-5,i-n+2)a^{2}x} +  
$$
$$
\sum_{i=n-1}^{n}r_{i} \underline{(ax^{2} + xax + x^{2}a) P(n-4,i-n+1)a^{2}x}) -  r_{n-3}\underline{a^{2}xa^{n}}
$$
respectively.
The following words of length $n+3$ from (\ref{n2})-(\ref{n8}) are reducible:
\begin{enumerate}
\item $x^{2}a^{n-2}x^{2}a$ on right hand side of the relation for $- (a^{n-2}x^{2})x^{2}a$.
\item $x(a^{n-2}x^{2})xa \in xa^{2}P(n-4,1)x^{2}a $
\item $axa^{n-3}x^{3}a \in axaP(n-4,1)x^{2}a$
\item $xa^{n-2}x^{2}ax \in  xa^{2}P(n-4,1)xax $
\item $xa^{n-3}x^{3}a^{2} \in xa^{2}P(n-5,2)xa^{2}$
\item $x^{3}a^{n-1}x $  on right hand side of the relation for $- (a^{n-3}x^{3})a^{2}x$
\item $x^{2}a(a^{n-2}x^{2})$  on right hand side of the relation for $- (a^{n-2}x^{2})ax^{2}$
\item $xax(a^{n-2}x^{2})$  on right hand side of the relation for $- (a^{n-2}x^{2})ax^{2}$
\item $ax^{2}(a^{n-2}x^{2})$  on right hand side of the relation for $- (a^{n-2}x^{2})ax^{2}$.
\end{enumerate}
Using (\ref{f}), expand $\sigma_{n-2}$ as
$$
a^{n-2}x^{2} \rightarrow - ( \sum_{i=n-2}^{n-1}r_{i} \underline{a^{2}P(n-5,i-n+2)a} + \underline{a^{2}P(n-5,2)a} + 
$$
$$
\underline{x^{2}P(n-3,0)a} + \underline{axP(n-3,0)x} + \underline{xaP(n-3,0)x} + \sum_{i=n-1}^{n}r_{i} ( \underline{axP(n-4,i-n+1)a} 
$$
$$
+ \underline{xaP(n-4,i-n+1)a} )  +   r_{n-1} \underline{a^{2}P(n-4,0)x} +   \underline{a^{2}Q(n-4,1)x} ) + r_{n-2}\underline{a^{n}}
$$
$$
= - ( \sum_{i=n-2}^{n}r_{i} \underline{a^{2}P(n-5,i-n+2)a} + \underline{x^{2}P(n-3,0)a} + \underline{axP(n-3,0)x} + \underline{xaP(n-3,0)x} + 
$$
$$
\sum_{i=n-1}^{n}r_{i} ( \underline{axP(n-4,i-n+1)a} + \underline{xaP(n-4,i-n+1)a} )  +  
$$
$$
r_{n-1} \underline{a^{2}P(n-4,0)x} +   \underline{a^{2}Q(n-4,1)x} ) + r_{n-2}\underline{a^{n}}.
$$
Thus, premultiplying this by $x$ and postmultiplying this by $xa$ and $ax$, and using Lemmas 2.6 and 2.7 to separate reducible and irreducible words, yield
\begin{equation}
\label{n9}
x(a^{n-2}x^{2})xa \rightarrow - ( \sum_{i=n-2}^{n}r_{i} \underline{xa^{2}P(n-5,i-n+2)axa} + \underline{x^{3}P(n-3,0)axa} + 
\end{equation}
$$
\underline{xaxP(n-3,0)x^{2}a} +  x^{2}aP(n-3,0)x^{2}a + \sum_{i=n-1}^{n}r_{i} ( \underline{xaxP(n-4,i-n+1)axa} +
$$
$$
 \underline{x^{2}aP(n-4,i-n+1)axa} )     + r_{n-1} \underline{xa^{2}P(n-4,0)x^{2}a} +   \underline{xa^{2}Q(n-4,1)x^{2}a} ) + r_{n-2}\underline{x^{2}a^{n+1}}
$$
and
\begin{equation}
\label{n10}
x(a^{n-2}x^{2})ax \rightarrow - (\sum_{i=n-2}^{n}r_{i} \underline{xa^{2}P(n-5,i-n+2)a^{2}x} + x^{3}P(n-3,0)a^{2}x +  \underline{x^{2}aP(n-3,0)xax} + 
\end{equation}
$$
\underline{xaxP(n-3,0)xax} + \sum_{i=n-1}^{n}r_{i} ( \underline{xaxP(n-4,i-n+1)a^{2}x} + \underline{x^{2}aP(n-4,i-n+1)a^{2}x} )    
$$
$$
+ r_{n-1} \underline{xa^{2}P(n-4,0)xax} +   \underline{xa^{2}Q(n-4,1)xax} ) + r_{n-2}\underline{xaxa^{n}}
$$
respectively.
Similarly, using (\ref{g}), expand $\sigma_{n-3}$ as
$$
a^{n-3}x^{3} \rightarrow - ( \sum_{i=n-3}^{n-1}r_{i} \underline{aP(n-6,i-n+3)a^{2}} + \underline{aP(n-6,3)a^{2}} +   
$$
$$
\sum_{i=n-2}^{n-1}r_{i} ( \underline{aP(n-5,i-n+2)ax} + \underline{aP(n-5,i-n+2)xa} + \underline{xP(n-5,i-n+2)a^{2}} )
$$
$$
+ \underline{aP(n-5,2)ax} + \underline{aP(n-5,2)xa} + \underline{xP(n-5,2)a^{2}} + \sum_{i=n-1}^{n}r_{i} ( \underline{xP(n-4,i-n+1)ax} 
$$
$$
+ \underline{xP(n-4,i-n+1)xa} )  +  r_{n-1}\underline{aP(n-4,0)x^{2}} +\underline{aQ(n-4,1)x^{2}} + \underline{xP(n-3,0)x^{2}} ) + r_{n-3}\underline{a^{n}}
$$
$$
= - ( \sum_{i=n-3}^{n}r_{i} \underline{aP(n-6,i-n+3)a^{2}} +  \sum_{i=n-2}^{n}r_{i} ( \underline{aP(n-5,i-n+2)ax} + \underline{aP(n-5,i-n+2)xa} 
$$
$$
+ \underline{xP(n-5,i-n+2)a^{2}} ) + \sum_{i=n-1}^{n}r_{i} ( \underline{xP(n-4,i-n+1)ax} + \underline{xP(n-4,i-n+1)xa} )  
$$
$$
+  r_{n-1}\underline{aP(n-4,0)x^{2}} +\underline{aQ(n-4,1)x^{2}} + \underline{xP(n-3,0)x^{2}} ) + r_{n-3}\underline{a^{n}}.
$$
Pre and post multiplying this by $ax$ and $a$ respectively, and using Lemmas 2.6 and 2.7 to separate reducible and irreducible words, yields
\begin{equation}
\label{11}
axa^{n-3}x^{3}a \rightarrow - ( \sum_{i=n-3}^{n}r_{i} \underline{axaP(n-6,i-n+3)a^{3}} +  \sum_{i=n-2}^{n}r_{i} ( \underline{axaP(n-5,i-n+2)axa}  
\end{equation}
$$
+ \underline{axaP(n-5,i-n+2)xa^{2}}  + \underline{ax^{2}P(n-5,i-n+2)a^{3}} ) + \sum_{i=n-1}^{n}r_{i} ( \underline{ax^{2}P(n-4,i-n+1)axa} 
$$
$$
+ \underline{ax^{2}P(n-4,i-n+1)xa^{2}} )   +  r_{n-1}\underline{axaP(n-4,0)x^{2}a} +\underline{axaQ(n-4,1)x^{2}a} 
$$
$$
+ \underline{ax^{2}P(n-3,0)x^{2}a} ) + r_{n-3}\underline{axa^{n+1}}.
$$
Again, using (\ref{f}), expand $\sigma_{n-3}$ as
$$
a^{n-3}x^{3} \rightarrow - ( \sum_{i=n-3}^{n-1}r_{i}  \underline{a^{2}P(n-6,i-n+3)a} + \underline{a^{2}P(n-6,3)a} + 
$$
$$
\sum_{i=n-2}^{n-1}r_{i}  ( \underline{axP(n-5,i-n+2)a} + \underline{xaP(n-5,i-n+2)a})  + \underline{axP(n-5,2)a} + \underline{xaP(n-5,2)a}
$$
$$
+ \sum_{i=n-2}^{n-1}r_{i} \underline{a^{2}P(n-5,i-n+2)x}  + \underline{a^{2}Q(n-5,2)x} +\sum_{i=n-1}^{n}r_{i} ( \underline{x^{2}P(n-4,i-n+1)a} +  
$$
$$
 \underline{axP(n-4,i-n+1)x} + \underline{xaP(n-4,i-n+1)x}) +  \underline{x^{2}P(n-3,0)x} ) +  r_{n-3}\underline{a^{n}}
$$
$$
= - ( \sum_{i=n-3}^{n}r_{i}  \underline{a^{2}P(n-6,i-n+3)a} + \sum_{i=n-2}^{n}r_{i}  ( \underline{axP(n-5,i-n+2)a} + \underline{xaP(n-5,i-n+2)a})  
$$
$$
+ \sum_{i=n-1}^{n}r_{i} ( \underline{x^{2}P(n-4,i-n+1)a} +   \underline{axP(n-4,i-n+1)x} + \underline{xaP(n-4,i-n+1)x}) + 
$$
$$
\sum_{i=n-2}^{n-1}r_{i} \underline{a^{2}P(n-5,i-n+2)x}  + \underline{a^{2}Q(n-5,2)x} + \underline{x^{2}P(n-3,0)x} ) +  r_{n-3}\underline{a^{n}}.
$$
Pre and post multiplying this by $x$ and $a^{2}$ respectively and using Lemmas 2.6 and 2.7 to separate reducible and irreducible words, yields
\begin{equation}
\label{n12}
xa^{n-3}x^{3}a^{2} \rightarrow - ( \sum_{i=n-3}^{n}r_{i}  \underline{xa^{2}P(n-6,i-n+3)a^{3}} + \sum_{i=n-2}^{n}r_{i}  ( \underline{xaxP(n-5,i-n+2)a^{3}} + 
\end{equation}
$$
\underline{x^{2}aP(n-5,i-n+2)a^{3}})   + \sum_{i=n-2}^{n-1}r_{i} \underline{xa^{2}P(n-5,i-n+2)xa^{2}}  + \underline{xa^{2}Q(n-5,2)xa^{2}} + 
$$
$$
\sum_{i=n-1}^{n}r_{i} ( \underline{x^{3}P(n-4,i-n+1)a^{3}} +   \underline{xaxP(n-4,i-n+1)xa^{2}} + \underline{x^{2}aP(n-4,i-n+1)xa^{2}}) 
$$
$$
 + \underline{x^{3}P(n-3,0)xa^{2}} ) +  r_{n-3}\underline{xa^{n+2}}
$$
Also, using (\ref{d}), expand $\sigma_{n-2}$ as
$$
a^{n-2}x^{2} \rightarrow - ( \sum_{i=n-2}^{n-1}r_{i} \underline{P(n-5,i-n+2)a^{3}} + \underline{P(n-5,2)a^{3}} + 
$$
$$
\sum_{i=n-1}^{n}r_{i}\underline{P(n-4,i-n+1)(a^{2}x + axa + xa^{2})}   
$$
$$
 + \underline{P(n-3,0)(xax + x^{2}a)}  ) +  r_{n-3}\underline{a^{n}} .
$$
$$
= - ( \sum_{i=n-2}^{n}r_{i} \underline{P(n-5,i-n+2)a^{3}} + \sum_{i=n-1}^{n}r_{i}\underline{P(n-4,i-n+1)(a^{2}x + axa + xa^{2})}   
$$
$$
 + \underline{P(n-3,0)(xax + x^{2}a)}  ) +  r_{n-3}\underline{a^{n}} .
$$
Thus, premultiplying this by $x^{2}a$, $xax$ and $ax^{2}$, and using Lemmas 2.6 and 2.7 to separate reducible and irreducible words, yield
\begin{equation}
\label{n13}
x^{2}a(a^{n-2}x^{2}) \rightarrow - ( \sum_{i=n-2}^{n}r_{i} \underline{x^{2}aP(n-5,i-n+2)a^{3}} + \sum_{i=n-1}^{n}r_{i}\underline{x^{2}aP(n-4,i-n+1)(a^{2}x + axa + xa^{2})}  
\end{equation}
$$
 + x^{2}aP(n-3,0)(xax + x^{2}a)  ) +  r_{n-3}\underline{x^{2}a^{n+1}} ,
$$

\begin{equation}
\label{n14}
xax(a^{n-2}x^{2}) \rightarrow - ( \sum_{i=n-2}^{n}r_{i} \underline{xaxP(n-5,i-n+2)a^{3}} + \sum_{i=n-1}^{n}r_{i}\underline{xaxP(n-4,i-n+1)(a^{2}x + axa + xa^{2})}   
\end{equation}
$$
+ \underline{xaxP(n-3,0)(xax + x^{2}a)}  ) +  r_{n-3}\underline{xaxa^{n}} 
$$
and
\begin{equation}
\label{n15}
ax^{2}(a^{n-2}x^{2}) \rightarrow - ( \sum_{i=n-2}^{n}r_{i} \underline{ax^{2}P(n-5,i-n+2)a^{3}} + \sum_{i=n-1}^{n}r_{i}\underline{ax^{2}P(n-4,i-n+1)(a^{2}x + axa + xa^{2})}
\end{equation}
$$
 + \underline{ax^{2}P(n-3,0)(xax + x^{2}a)}  ) +  r_{n-3}\underline{ax^{2}a^{n}} 
$$
respectively.
The word in (\ref{n9})-(\ref{n15}) of length $n+3$ which are reducible are as follows:
\begin{enumerate}
\item $-x^{2}a^{n-2}x^{2}a \in x(a^{n-2}x^{2})xa$
\item $-x^{3}a^{n-1}x \in x(a^{n-2}x^{2})ax$
\item $-x^{2}a^{n-2}x^{2}a \in x^{2}a(a^{n-2}x^{2})$.
\end{enumerate}
But terms (i) and (ii) in the preceding list have opposite signs with the terms (i) and (iv) in the list of nine reducible words above so they cancel out. Thus, we only have to reduce the term (iii). Using (\ref{g}), expand $\sigma_{n-2}$ as
$$
a^{n-2}x^{2} \rightarrow - ( \sum_{i=n-2}^{n-1}r_{i} \underline{aP(n-5,i-n+2)a^{2}} + \underline{aP(n-5,2)a^{2}} +  
$$
$$
\sum_{i=n-1}^{n}r_{i}  ( \underline{aP(n-4,i-n+1)ax} + \underline{aP(n-4,i-n+1)xa} +  \underline{xP(n-4,i-n+1)a^{2}}) 
$$
$$
+  \underline{xP(n-3,0)ax} + \underline{xP(n-3,0)xa} ) + r_{n-2}\underline{a^{n}}
$$
$$
= - ( \sum_{i=n-2}^{n}r_{i} \underline{aP(n-5,i-n+2)a^{2}} + \sum_{i=n-1}^{n}r_{i}  ( \underline{aP(n-4,i-n+1)ax} + 
$$
$$
 \underline{aP(n-4,i-n+1)xa} +  \underline{xP(n-4,i-n+1)a^{2}}) + \underline{xP(n-3,0)ax} + \underline{xP(n-3,0)xa} ) + r_{n-2}\underline{a^{n}}.
$$
Pre and post multiplying this by $x^{2}$ and $a$ respectively, and using Lemmas 2.6 and 2.7 to separate reducible and irreducible words, yields
\begin{equation}
\label{n16}
-x^{2}a^{n-2}x^{2}a \rightarrow ( \sum_{i=n-2}^{n}r_{i} \underline{x^{2}aP(n-5,i-n+2)a^{3}} + \sum_{i=n-1}^{n}r_{i}  ( \underline{x^{2}aP(n-4,i-n+1)axa} + 
\end{equation}
$$
 \underline{x^{2}aP(n-4,i-n+1)xa^{2}} +  \underline{x^{3}P(n-4,i-n+1)a^{3}}) + \underline{x^{3}P(n-3,0)axa} 
$$
$$
+ \underline{x^{3}P(n-3,0)xa^{2}} ) -  r_{n-2}\underline{x^{2}a^{n+1}}.
$$
Hence, the reduction process ends here and when we substitute (\ref{n2})-(\ref{n16}) into (\ref{n1}), we get
\begin{equation}
\label{n17}
a^{3}\omega_{n-4} \rightarrow ( \sum_{i=n-3}^{n}r_{i} \underline{a^{2}xP(n-6,i-n+3)a^{3}} + \sum_{i=n-2}^{n}r_{i} ( \underline{a^{2}xP(n-5,i-n+2)(a^{2}x + axa + xa^{2} )}   
\end{equation}
$$
+  \underline{axaP(n-5,i-n+2)a^{2}x}  )  + \sum_{i=n-1}^{n}r_{i} ( \underline{a^{2}xP(n-4,i-n+1)(ax^{2} + xax + x^{2}a)} + 
$$
$$
 \underline{axaP(n-4,i-n+1)(ax^{2} + xax)} +  \underline{xa^{2}P(n-4,i-n+1)ax^{2}} + \underline{x^{3}P(n-4,i-n+1)a^{3}} ) +  
$$
$$
 \underline{x^{3}P(n-3,0)(axa + xa^{2}} + r_{n-2}\underline{xaxa^{n}}  ) - ( \sum_{i=n-2}^{n}r_{i} \underline{xax(P(n-5,i-n+2)a^{3}}  +
$$
$$
\sum_{i=n-1}^{n}r_{i} ( \underline{ax^{2}P(n-4,i-n+1)(axa + xa^{2})} + \underline{xaxP(n-4,i-n+1)(a^{2}x+ axa + xa^{2})}  +   
$$
$$
\underline{x^{2}aP(n-4,i-n+1)a^{2}x} ) +   r_{n-1}a^{3}P(n-4,0)x^{3}   +    \underline{a^{3}Q(n-4,1)x^{3}} +  \underline{xaxP(n-3,0)( xax + x^{2}a )}
$$
$$
+  \underline{x^{2}aP(n-3,0)xax}    + \underline{ax^{2}(P(n-3,0)x^{2}a + P(n-5,2)a^{3}) }  + r_{n-3}\underline{a^{2}xa^{n}}  ).
$$

Similarly, usind (\ref{d}), expand $\sigma_{n-1}$ as
$$
\omega_{n-1} \rightarrow - ( r_{n-1}\underline{a^{3}P(n-4,0)} + \underline{a^{3}Q(n-4,1)} +  \underline{( a^{2}x + axa + xa^{2} )P(n-3,0)} ) + r_{n-1} \underline{a^{n}}.
$$
Post multiplying this by $x^{3}$  and using Lemmas 2.6 and 2.7 to separate reducible and irreducible words, yields
\begin{equation}
\label{n18}
\omega_{n-1}x^{3} \rightarrow - (  r_{n-1} a^{3}P(n-4,0)x^{3} + \underline{a^{3}Q(n-4,1)x^{3}} +  ( a^{2}x + axa + xa^{2} )P(n-3,0)x^{3} ) + r_{n-1} \underline{a^{n}x^{3}} .
\end{equation}
The following words in (\ref{n18}) of length $n+3$ are reducible:
\begin{enumerate}
\item $xa^{n-1}x^{3}$
\item $axa^{n-2}x^{3}$
\item $a^{2}xa^{n-3}x^{3}$.
\end{enumerate}
Using (\ref{d}), expand $\sigma_{n-3}$ as
$$
 a^{n-3}x^{3} \rightarrow - ( \sum_{i=n-3}^{n-1} r_{i} \underline{P( n-6,i-n+3 )a^{3}} + \underline{P( n-6,3 )a^{3}} +
$$
$$
 \sum_{i=n-2}^{n-1}r_{i} \underline{P(n-5,i-n+2)(a^{2}x + axa + xa^{2} )}  + \underline{P(n-5,2)(a^{2}x + axa + xa^{2} )} +
$$
$$
\sum_{i=n-1}^{n} r_{i} \underline{P(n-4,i-n+1)(ax^{2} + xax + x^{2}a)}  ) +  r_{n-3} \underline{a^{n}} .
$$
$$
 = - ( \sum_{i=n-3}^{n} r_{i} \underline{P( n-6,i-n+3 )a^{3}} +  \sum_{i=n-2}^{n}r_{i} \underline{P(n-5,i-n+2)(a^{2}x + axa + xa^{2} )}  + 
$$
$$
\sum_{i=n-1}^{n} r_{i} \underline{P(n-4,i-n+1)(ax^{2} + xax + x^{2}a)}  ) +  r_{n-3} \underline{a^{n}} .
$$
Thus, premultiplying this by $xa^{2}$, $axa$ and $a^{2}x$, and using Lemmas 2.6 and 2.7 to separate reducible and irreducible words, yield
\begin{equation}
\label{n19}
-xa^{2}(a^{n-3}x^{3}) \rightarrow ( \sum_{i=n-3}^{n} r_{i} \underline{xa^{2}P( n-6,i-n+3 )a^{3}} +  \sum_{i=n-2}^{n}r_{i} \underline{xa^{2}P(n-5,i-n+2)(a^{2}x + axa + xa^{2} )}  
\end{equation}
$$
+ \sum_{i=n-1}^{n} r_{i} xa^{2}P(n-4,i-n+1)(ax^{2} + xax + x^{2}a)  ) -  r_{n-3} \underline{xa^{n+2}},
$$

\begin{equation}
\label{n20}
-axa(a^{n-3}x^{3}) \rightarrow ( \sum_{i=n-3}^{n} r_{i} \underline{axaP( n-6,i-n+3 )a^{3}} +  \sum_{i=n-2}^{n}r_{i} \underline{axaP(n-5,i-n+2)(a^{2}x + axa + xa^{2} )}   
\end{equation}
$$
+  \sum_{i=n-1}^{n} r_{i} \underline{axaP(n-4,i-n+1)(ax^{2} + xax + x^{2}a)}  ) -  r_{n-3} \underline{axa^{n+1}}
$$
and
\begin{equation}
\label{n21}
-a^{2}x(a^{n-3}x^{3}) \rightarrow ( \sum_{i=n-3}^{n} r_{i} \underline{a^{2}xP( n-6,i-n+3 )a^{3}} +  \sum_{i=n-2}^{n}r_{i} \underline{a^{2}xP(n-5,i-n+2)(a^{2}x + axa + xa^{2} )}  
\end{equation}
$$
+ \sum_{i=n-1}^{n} r_{i} \underline{a^{2}xP(n-4,i-n+1)(ax^{2} + xax + x^{2}a)}  ) -  r_{n-3} \underline{a^{2}xa^{n}} 
$$
respectively.
The reducible words in (\ref{n19})-(\ref{n21}) of length $n+3$ are as follows:
\begin{enumerate}
\item $xa^{n-2}x^{2}xa   \in xa^{2}P(n-4,1)x^{2}a $
\item $xa^{n-2}x^{2}ax \in  xa^{2}P(n-4,1)xax $
\item $ xa^{n-3}x^{3}a^{2} \in xa^{2}P(n-4,1)ax^{2} $
\item $axa^{n-3}x^{3}a \in axaP(n-4,1)x^{2}a$.
\end{enumerate}
Recall frrom (\ref{n9})-(\ref{n12}) that
$$
x(a^{n-2}x^{2})xa \rightarrow - ( \sum_{i=n-2}^{n}r_{i} \underline{xa^{2}P(n-5,i-n+2)axa} + \underline{x^{3}P(n-3,0)axa} + \underline{xaxP(n-3,0)x^{2}a} +  
$$
$$
x^{2}aP(n-3,0)x^{2}a + \sum_{i=n-1}^{n}r_{i} ( \underline{xaxP(n-4,i-n+1)axa} + \underline{x^{2}aP(n-4,i-n+1)axa} )    
$$
$$
+ r_{n-1} \underline{xa^{2}P(n-4,0)x^{2}a} +   \underline{xa^{2}Q(n-4,1)x^{2}a} ) + r_{n-2}\underline{x^{2}a^{n+1}}.
$$

$$
x(a^{n-2}x^{2})ax \rightarrow - (\sum_{i=n-2}^{n}r_{i} \underline{xa^{2}P(n-5,i-n+2)a^{2}x} + x^{3}P(n-3,0)a^{2}x +  \underline{x^{2}aP(n-3,0)xax} + 
$$
$$
\underline{xaxP(n-3,0)xax} + \sum_{i=n-1}^{n}r_{i} ( \underline{xaxP(n-4,i-n+1)a^{2}x} + \underline{x^{2}aP(n-4,i-n+1)a^{2}x} )    
$$
$$
+ r_{n-1} \underline{xa^{2}P(n-4,0)xax} +   \underline{xa^{2}Q(n-4,1)xax} ) + r_{n-2}\underline{xaxa^{n}}.
$$

$$
axa^{n-3}x^{3}a \rightarrow - ( \sum_{i=n-3}^{n}r_{i} \underline{axaP(n-6,i-n+3)a^{3}} +  \sum_{i=n-2}^{n}r_{i} ( \underline{axaP(n-5,i-n+2)axa} + 
$$
$$
\underline{axaP(n-5,i-n+2)xa^{2}}  + \underline{ax^{2}P(n-5,i-n+2)a^{3}} ) + \sum_{i=n-1}^{n}r_{i} ( \underline{ax^{2}P(n-4,i-n+1)axa} + 
$$
$$
\underline{ax^{2}P(n-4,i-n+1)xa^{2}} )   +  r_{n-1}\underline{axaP(n-4,0)x^{2}a} +\underline{axaQ(n-4,1)x^{2}a} + \underline{ax^{2}P(n-3,0)x^{2}a} ) 
$$
$$
+ r_{n-3}\underline{axa^{n+1}}
$$
and
$$
xa^{n-3}x^{3}a^{2} \rightarrow - ( \sum_{i=n-3}^{n}r_{i}  \underline{xa^{2}P(n-6,i-n+3)a^{3}} + \sum_{i=n-2}^{n}r_{i}  ( \underline{xaxP(n-5,i-n+2)a^{3}} + 
$$
$$
\underline{x^{2}aP(n-5,i-n+2)a^{3}}) +  \sum_{i=n-2}^{n-1}r_{i} \underline{xa^{2}P(n-5,i-n+2)xa^{2}}  + \underline{xa^{2}Q(n-5,2)xa^{2}}  + \
$$
$$
\sum_{i=n-1}^{n}r_{i} ( \underline{x^{3}P(n-4,i-n+1)a^{3}} +   \underline{xaxP(n-4,i-n+1)xa^{2}} +  \underline{x^{2}aP(n-4,i-n+1)xa^{2}})
$$
$$
 + \underline{x^{3}P(n-3,0)xa^{2}} )  +  r_{n-3}\underline{xa^{n+2}}.
$$

The reducible words in (\ref{n9})-(\ref{n12}) of length $n+3$ are:
\begin{enumerate}
\item $- x^{2}a^{n-2}x^{2}a \in x(a^{n-2}x^{2})xa$
\item $ -x^{3}a^{n-1}x  \in x(a^{n-2}x^{2})ax$.
\end{enumerate}
Recall from (\ref{n16}) that
$$
-x^{2}a^{n-2}x^{2}a \rightarrow ( \sum_{i=n-2}^{n}r_{i} \underline{x^{2}aP(n-5,i-n+2)a^{3}} + \sum_{i=n-1}^{n}r_{i}  ( \underline{x^{2}aP(n-4,i-n+1)axa} + 
$$
$$
 \underline{x^{2}aP(n-4,i-n+1)xa^{2}} +  \underline{x^{3}P(n-4,i-n+1)a^{3}}) + \underline{x^{3}P(n-3,0)axa} 
$$
$$
+ \underline{x^{3}P(n-3,0)xa^{2}} ) -  r_{n-2}\underline{x^{2}a^{n+1}}.
$$

Also, using (\ref{d}), expand $\sigma_{n-1}$ as
$$
 a^{n-1}x \rightarrow - ( \sum_{i=n-1}^{n}r_{i}  \underline{P(n-4,i-n+1)a^{3}} + \underline{P(n-3,0)(axa + xa^{2})} ) + r_{n-1}\underline{a^{n}}.
$$
Premultiplying this by $x^{3}$ and using Lemmas 2.6 and 2.7 to separate reducible and irreducible words, yields
\begin{equation}
\label{n22}
 -x^{3}a^{n-1}x \rightarrow ( \sum_{i=n-1}^{n}r_{i}  \underline{x^{3}P(n-4,i-n+1)a^{3}} + \underline{x^{3}P(n-3,0)(axa + xa^{2})} ) - r_{n-1}\underline{x^{3}a^{n}}.
\end{equation}
Thus, the reduction process ends here and when we substitute (\ref{n9})-(\ref{n12}), (\ref{n16}), (\ref{n19})-(\ref{n22}) into (\ref{n18}), we get
\begin{equation}
\label{n23}
\omega_{n-1}x^{3}  \rightarrow ( \sum_{i=n-3}^{n}r_{i} \underline{a^{2}xP(n-6,i-n+3)a^{3}} + \sum_{i=n-2}^{n}r_{i} ( \underline{a^{2}xP(n-5,i-n+2)(a^{2}x + axa + xa^{2} )}   
\end{equation}
$$
+  \underline{axaP(n-5,i-n+2)a^{2}x}  )  + \sum_{i=n-1}^{n}r_{i} ( \underline{a^{2}xP(n-4,i-n+1)(ax^{2} + xax + x^{2}a)} + 
$$
$$
 \underline{axaP(n-4,i-n+1)(ax^{2} + xax)} +  \underline{xa^{2}P(n-4,i-n+1)ax^{2}} + \underline{x^{3}P(n-4,i-n+1)a^{3}} ) +  
$$
$$
 \underline{x^{3}P(n-3,0)(axa + xa^{2}} + r_{n-2}\underline{xaxa^{n}}  ) - ( \sum_{i=n-2}^{n}r_{i} \underline{xax(P(n-5,i-n+2)a^{3}}  +
$$
$$
\sum_{i=n-1}^{n}r_{i} ( \underline{ax^{2}P(n-4,i-n+1)(axa + xa^{2})} + \underline{xaxP(n-4,i-n+1)(a^{2}x+ axa + xa^{2})}  +   
$$
$$
\underline{x^{2}aP(n-4,i-n+1)a^{2}x} ) +   r_{n-1}a^{3}P(n-4,0)x^{3}   +    \underline{a^{3}Q(n-4,1)x^{3}} +  \underline{xaxP(n-3,0)( xax + x^{2}a )}
$$
$$
+  \underline{x^{2}aP(n-3,0)xax}    + \underline{ax^{2}(P(n-3,0)x^{2}a + P(n-5,2)a^{3}) }  + r_{n-3}\underline{a^{2}xa^{n}}  ).
$$
Comparing (\ref{n17}) and (\ref{n23}), we conclude that the overlap ambiguity $\{ \omega_{n-4}, \omega_{n-1} \}$ is resolvable.

\end{document}